\documentclass[11pt, twoside]{article}
\usepackage{amsfonts}

\usepackage{amsmath,amssymb}
\usepackage{multicol}
\usepackage{graphicx}
\usepackage{color}

\def\bar{\overline}

\numberwithin{equation}{section}
\newtheorem{theorem}{Theorem}[section]
\newtheorem{lemma}[theorem]{Lemma}

\newtheorem{definition}[theorem]{Definition}
\newtheorem{remark}[theorem]{Remark}
\newtheorem{proposition}[theorem]{Proposition}
\newenvironment{proof}[1][Proof]{\noindent\textbf{#1.} }{\hfill $\Box$}
\allowdisplaybreaks
\allowdisplaybreaks \numberwithin{equation}{section}
 \makeatletter
\begin{document}

\author{ Wei-Jie Sheng\thanks{Corresponding author (E-mail address: 
shengwj09@hit.edu.cn).} \   and  Xin-Tian Zhang \\~\\
\footnotesize{School of Mathematics, Harbin Institute of Technology}, \\
\footnotesize{Harbin, Heilongjiang, 150001, People's Republic of China}
}

\title{\textbf{ Curved fronts of combustion reaction-diffusion equations } } 

\date{}
\maketitle

\textbf{Abstract}: This paper is concerned with curved fronts of combustion reaction-diffusion equations in $\mathbb{R}^N$ $(N\geq2)$. By mixing finite
planar fronts and constructing suitable super- and subsolutions, 
we prove the existence, uniqueness and stability of polytope-like curved fronts in 
$\mathbb{R}^N$. Besides, we show that these curved fronts are transition fronts.



\section{Introduction}
\noindent
In this paper, we focus on the curved fronts of the following reaction-diffusion equation
\begin{equation}\label{1.1}
u_t-\Delta u=f(u), \quad(t, z) \in \mathbb{R} \times \mathbb{R}^N,
\end{equation}
where $N \geq 2$, $u=u(t, z)$, $u_t=\frac{\partial u}{\partial t}$, $\Delta:=\partial^2 / \partial 
z_1^2+\partial^2 / \partial z_2^2+\cdots+\partial^2 / \partial z_N^2$ is the Laplace operator and $f(u)$ is the nonlinear term.
In the sequel, we always make the following assumptions.
\begin{description}
	\item[(\textbf{F1})] 
	$f(u)$ is of class
	$C^{1+\vartheta}\left([-\sigma, 1+\sigma], \mathbb{R}\right)$ for some $\vartheta\in(0, 1]$ and 
	$\sigma\in(0,1)$.  $f(u)$ satisfies that
	\begin{equation}\label{1.2}
		\begin{aligned}
		\exists \theta \in(0,1) \text { such that }	f(u) \equiv 0 \text { in }[0, \theta]\cup \{1\}, \; f(u)>0 
			\text { in }(\theta, 1) \text { and } f^{\prime}(1)<0.
		\end{aligned}
	\end{equation}
	For mathematical convenience, we also suppose that $f(u) \equiv 0$ in $[-\sigma, 0)$ and thereby $f^{\prime}(0) = 0$.
	
	\item[(\textbf{F2})] There exist a constant $c_{f} > 0$ and an decreasing function $U\in C^{2}(\mathbb{R})$ such that 
	\begin{equation}\label{1.3}
		\left\{\begin{array}{l}
			U^{\prime \prime}(\mathcal{D})+c_{f} U^{\prime}(\mathcal{D})+f(U(\mathcal{D}))=0, \quad \mathcal{D}\in\mathbb{R}, \\
			U(-\infty)=1, \quad U(+\infty)=0,\\ U^{\prime}(\mathcal{\mathcal{D}})<0,\quad 
			\mathcal{D}\in \mathbb{R}
		\end{array}\right.
	\end{equation}
	and
	\begin{equation*}
		\lim _{\mathcal{D} \rightarrow+\infty} \frac{U^{ \prime}(\mathcal{D})}{U(\mathcal{D})}=-c_{f}.
	\end{equation*}
\end{description}
According to \cite{Bu2018}, we can get the following conclusions.
\begin{proposition}
	Under the assumptions $(F1)$ and $(F2)$, there exist positive constants $L_{1}$, $L_{2}$, 
	$L_{3}$, $L_{4}$ and $\beta_{0}$ such that 
	\begin{equation}\label{1.5}
		\begin{gathered}
			L_1 e^{-c_{f} \mathcal{D}} \leq U(\mathcal{D}), \left|U^{\prime}(\mathcal{D})\right|, 
			\left|U^{\prime \prime}(\mathcal{D})\right| \leq L_2 e^{-c_{f} \mathcal{D}}, \quad \forall 
			\mathcal{D}>0, \\
			L_{4} e^{\beta_{0} \mathcal{D}}\leq |U(\mathcal{D})-1|, 
			\left|U^{\prime}(\mathcal{D})\right|,\left|U^{\prime \prime}(\mathcal{D})\right| \leq L_3 
			e^{\beta_{0} \mathcal{D}}, \quad \forall \mathcal{D}<0 ,
		\end{gathered}
	\end{equation}
	and
	\begin{equation}\label{1.6}
		\lim _{\mathcal{D} \rightarrow+\infty} \frac{U^{ 
		\prime}(\mathcal{D})}{U(\mathcal{D})}=-c_{f},\quad \lim _{\mathcal{D} 
		\rightarrow+\infty}\left| \frac{U^{\prime 
		\prime}(\mathcal{D})}{U(\mathcal{D})}\right|=c_{f}^2.
	\end{equation}
	Moreover, there exists a positive constant $\gamma_{\star}$ such that $ \gamma_{\star}\leq 
	\min 
	\{\theta / 4,(1-\theta)/2, \sigma / 4\}$ and
	\begin{equation}\label{1.4}
		\frac{3}{2} f^{\prime}(1) \leq f^{\prime}(u)\leq\frac{1}{2} f^{\prime}(1), \quad \forall u \in 
		\left[1-2\gamma_{\star}, 1+2\gamma_{\star}\right].
	\end{equation}
\end{proposition}
 
Based on the shape of  the level sets, traveling fronts are classified as planar traveling fronts and 
non-planar traveling fronts. We call traveling fronts whose level sets are hyperplanes 
orthogonal to the propagation direction as planar traveling fronts. 
Because planar traveling fronts have 
simple forms and good geometric properties, they have been studied and concerned by many 
scholars. For the research on the existence, stability and properties of planar traveling fronts, one 
can refer to \cite{Aronson2016,Berestycki 1985,BerestyckiH2002,Chen1997,Fife1977} and 
references therein.
However,   in the high dimensional space, the propagation phenomena become more complex and 
level sets of traveling fronts also become more diverse. We refer to the traveling fronts that have non-planar level sets as non-planar traveling fronts or curved fronts. 
For instance, Bonnet and Hamel \cite{bn1999} investigated the existence of $V$-shaped traveling fronts for combustion reaction-diffusion equations in $\mathbb{R}^2$, which is the first study of the conical premixed Bunsen flames. Wang 
and Bu \cite{Wang2016} established the existence of $V$-shaped traveling fronts for combustion 
reaction-diffusion equations in $\mathbb{R}^2$. Bu and Wang \cite{Bu2018} proved the global 
stability of $V$-shaped traveling fronts obtained in \cite{Wang2016} in $\mathbb{R}^2$. El 
Smaily \cite{El Smaily M2018} obtained the existence and uniqueness of conical fronts for a 
reaction-advection-diffusion equation with a combustion nonlinearity in $\mathbb{R}^2$. Hamel, 
Monneau and Roquejoffre \cite{Hamel2004} studied the stability of conical-shaped solutions of a 
class of reaction-diffusion equations in $\mathbb{R}^2$. Taniguchi \cite{M. Taniguchi2007} 
established the existence of the pyramidal traveling fronts for 
bistable reaction-diffusion equations in $\mathbb{R}^3$. Wang and 
Bu \cite{Wang2016} also obtained the existence of pyramidal 
traveling fronts for combustion reaction-diffusion equations in $\mathbb{R}^3$. Furthermore, Bu and Wang \cite{Bu2017} considered the stability of pyramidal traveling fronts for combustion reaction-diffusion equations in $\mathbb{R}^3$.
More related results on non-planar traveling fronts, we refer to 
\cite{Brazhnik2000,El Smaily M2011,Hamel2000,Hamel2005} and references therein.

Below we review  definitions  of the transition fronts and their global mean speed
introduced by Berestycki and Hamel in the pioneering work \cite{H. 
Berestycki2012}. One can also reefer to Shen
\cite{Shen2004} for the one-dimensional setting. For any two subsets $A\in \mathbb{R}^N$ 
and $B\in \mathbb{R}^N$, we define
\begin{equation*}
d(A, B)=\inf \{|z-y|; (z, y) \in A \times B\} \quad \text{and} \quad d(z, A)=d(\{z\}, A),
\end{equation*}
where $z \in \mathbb{R}^N$ and $|\cdot|$ is the Euclidean norm in $\mathbb{R}^N$. Let $\left(\Omega_t^{-}\right)_{t \in \mathbb{R}}$ and $\left(\Omega_t^{+}\right)_{t \in \mathbb{R}}$ be two families of open nonempty subsets of $\mathbb{R}^N$ satisfying 
\begin{equation}\label{1.7}
\forall t \in \mathbb{R}, \quad\left\{\begin{array}{l}
	\Omega_t^{+} \cap \Omega_t^{-}=\emptyset, \\
	\partial \Omega_t^{+}=\partial \Omega_t^{-}=: \Gamma_t, \\
	\Omega_t^{+} \cup \Omega_t^{-} \cup \Gamma_t=\mathbb{R}^N, \\
	\sup \left\{d\left(z, \Gamma_t\right) ; z \in \Omega_t^{+}\right\}=\sup \left\{d\left(z, \Gamma_t\right) ; z \in \Omega_t^{-}\right\}=+\infty
\end{array}\right.
\end{equation}
and
\begin{equation}\label{1.8}
\left\{\begin{array}{l}
	\inf \left\{\sup \left\{d\left(y, \Gamma_t\right) ; y \in \Omega_t^{+},|y-z| \leq r\right\} ; t \in \mathbb{R}, z \in \Gamma_t\right\} \rightarrow+\infty \\
	\inf \left\{\sup \left\{d\left(y, \Gamma_t\right) ; y \in \Omega_t^{-},|y-z| \leq r\right\} ; t \in \mathbb{R}, z \in \Gamma_t\right\} \rightarrow+\infty
\end{array} \text { as } r \rightarrow +\infty .\right.
\end{equation}
From \eqref{1.7} and \eqref{1.8}, it can be seen that $\Gamma_t$ divides $\mathbb{R}^N$ into two parts $\Omega_t^{-}$ and $\Omega_t^{+}$, and $\Omega_t^{ \pm}$ are unbounded. Moreover, for each $t \in \mathbb{R}$, $\Omega_t^{ \pm}$ contain points that are infinitely far from $\Gamma_t$. 
According to \eqref{1.8}, for any constant $M>0$, there is a positive constant $r_M$ such that for all $t \in \mathbb{R}$ and $z \in \Gamma_t$, there exist $y^{ \pm}=y_{t, z}^{ \pm} \in \mathbb{R}^N$ such that
\begin{equation*}
y^{ \pm} \in \Omega_t^{ \pm},\left|z-y^{ \pm}\right| \leq r_M \text { and } d\left(y^{ \pm}, \Gamma_t\right) \geq M .
\end{equation*}
Furthermore, it is assumed that the sets $\Gamma_t$ are composed of a finite number of graphs, that is, there exists an integer $n \geq 1$ such that, for any $t \in \mathbb{R}$, there are $n$ open subsets $\omega_{i, t} \subset \mathbb{R}^{N-1}(1 \leq i \leq n)$, $n$ continuous maps $\psi_{i, t}: \omega_{i, t} \rightarrow \mathbb{R}$ $(1 \leq i \leq n)$ and $n$ rotations $R_{i, t}$ of $\mathbb{R}^N$ $(1 \leq i \leq n)$, such that
\begin{equation}\label{1.9}
\Gamma_t \subset \bigcup_{1 \leq i \leq n} R_{i, t}\left(\left\{z \in \mathbb{R}^N ; \left(z_{1}, z_{2}, \ldots, z_{N-1}\right) \in \omega_{i, t}, z_N=\psi_{i, t}\left(z_{1}, z_{2}, \ldots, z_{N-1}\right)\right\}\right).
\end{equation}
\begin{definition}(\cite{H. Berestycki2012})\label{De 1.1}
	 For equation \eqref{1.1}, a transition front connecting 0 and 1 is a classical (time-global) 
	 solution $u$ of \eqref{1.1} such that $u \not \equiv 0,1$, and there exist some sets 
	 $\left(\Omega_t^{ \pm}\right)_{t \in \mathbb{R}}$ and $\left(\Gamma_t\right)_{t \in 
	 \mathbb{R}}$ satisfying \eqref{1.7}, \eqref{1.8} and \eqref{1.9}, and for any $\varepsilon>0$, 
	 there exists $M_{\varepsilon}>0$ such that
\begin{equation*}
	\left\{\begin{array}{l}
		\forall t \in \mathbb{R}, \forall z \in \Omega_t^{+},\left(d \left(z, \Gamma_t\right) \geq M_{\varepsilon}\right) \Longrightarrow(u(t, z) \geq 1-\varepsilon), \\
		\forall t \in \mathbb{R}, \forall z \in \Omega_t^{-},\left(d \left(z, \Gamma_t\right) \geq M_{\varepsilon}\right) \Longrightarrow(u(t, z) \leq \varepsilon) .
	\end{array}\right.
\end{equation*}
Moreover, $u$ is said to have a global mean speed $\gamma$ $(\geq 0)$ if
\begin{equation*}
	\frac{d\left(\Gamma_t, \Gamma_s\right)}{|t-s|} \rightarrow \gamma \text { as }|t-s| \rightarrow+\infty.
\end{equation*}
\end{definition}

We would like to mention some results on transition fronts inspired our study.
For the bistable case, in the seminal work  \cite{Hamel2016} Hamel firstly 
studied the existence and further  properties of shape-changing  entire solutions for 
reaction-diffusion equations in $\mathbb R^N$. 
Sheng and Guo \cite{Sheng2018.2}, Sheng, Wang
and Wang \cite{Sheng2018,Sheng2021}
generalized the results of \cite{Hamel2016} to time periodic bistable scalar 
equations/systems; Recently,
Guo and Wang 
\cite{H. Guo2024} showed the existence, uniqueness and stability of shape-changing  entire 
solutions for reaction-diffusion 
equations by mixing finite planar fronts. 
Bu, Guo and Wang \cite{Bu2019} considered the existence of shape-changing entire solutions
of combustion reaction-diffusion equations in $\mathbb{R}^2$, and then trivially extended to 
$\mathbb{R}^N$, in spirit of \cite{Hamel2016}.
Nevertheless, this method may not be applicable to obtain shape-changing entire solutions 
in $\mathbb{R}^N(N\geq 3)$.  In this paper, we aim to obtain the existence, uniqueness and 
stability of polytope-like  curved fronts of Eq.\eqref{1.1}. The main idea comes from \cite{H. 
Guo2024} for bistable reaction-diffusion equations. However, given that the nonlinearity is 
degenerate, particular attention must be paid to addressing the difficulties induced by this 
degeneracy. A common approach involves careful selection of weighted functions to serve as the 
tail in the construction of sub- and super-solutions.

The rest of this paper is organized as follows. In Section $2$, we state our main results. In Section 
$3$, we construct suitable super- and subsolutions to prove the existence of polytope-like curved 
fronts of combustion reaction-diffusion equations in $\mathbb{R}^N$ $(N \geq 2)$, and prove 
their monotonicity and uniqueness. In Section $4$, we study the stability of the curved fronts 
obtained in Section $3$. 

\section{Main Results}
\noindent
First of all, let us introduce some notations that will be used later in the text. For any $e \in \mathbb{S}^{N-1}$ and $\tau \in \mathbb{R}$, denote the hyperplane
\begin{equation*}
P(e, \tau)=\left\{(t, z) \in \mathbb{R} \times \mathbb{R}^N ; z \cdot e-c_f t+\tau=0\right\},	
\end{equation*}
which is also regarded as a hyperplane moving over time in $\mathbb{R}^N$. Each pair $(e, \tau)$ determines a planar front $U\left(z \cdot e-c_f t+\tau\right)$. Take $n \geq 2$ pairs $(e_i, \tau_i)$ of $\mathbb{S}^{N-1} \times \mathbb{R}$ such that $e_i \neq e_j$ for any $i \neq j$ and there exists an $e_0 \in \mathbb{S}^{N-1}$ such that  $e_i \cdot e_0>0$ for any $i\in\{1, \ldots, n\}$.
Denote $P_i:=P\left(e_i, \tau_i\right)$. Since $e_i \cdot e_0>0$ for any $i\in\{1, \ldots, n\}$ and $e_i \neq e_j$ for any $i \neq j$, it holds that $P_i$ and $P_j$ are not parallel to each other. Let $\mathcal{P}\left(P_1, \ldots, P_n\right)$ be the polytope enclosed by $P_1, \cdots, P_n$, namely,
\begin{equation*}
\mathcal{P}\left(P_1, \ldots, P_n\right):=\left\{(t, z) \in \mathbb{R} \times \mathbb{R}^N ; \min _{1 \leq i \leq n}\left\{z \cdot e_i-c_f t+\tau_i\right\} \geq 0\right\}.
\end{equation*}
For all $i$, there is $e_i \cdot e_0>0$, which guarantees the unboundedness. For convenience, we abbreviate $\mathcal{P}\left(P_1, \ldots, P_n\right)$ as $\mathcal{P}$. $\mathcal{P}$ can also be regarded as a polytope in $\mathbb{R}^N$ that moves over time. According to the definition of the polytope $\mathcal{P}$, it can be verified that $\mathcal{P}$ is convex in $\mathbb{R}^{N+1}$, and its time slice $\mathcal{P}_t$ is also convex in $\mathbb{R}^N$ for each $t$. Furthermore, the boundary of $\mathcal{P}$ is defined as follows
\begin{equation*}
\partial \mathcal{P}:=\left\{(t, z) \in \mathbb{R} \times \mathbb{R}^N ; \min _{1 \leq i \leq n}\left\{z \cdot e_i-c_f t+\tau_i\right\}=0\right\}.
\end{equation*}
The joint part of $P_i$ and $\partial \mathcal{P}$ is called the facet of the surface and is denoted by $\widetilde{P}_i := P_i \cap \partial \mathcal{P}$. If we consider $\mathcal{P}$ as a polytope moving in $\mathbb{R}^N$ over time, then there exists a real number $T \in \mathbb{R}$ such that for each $t \leq T$, there are $n$ facets on the surface $\partial \mathcal{P}_t$ of the polytope $\mathcal{P}_t$, meaning that the joint part of each $P_{i, t}$ with $\partial \mathcal{P}_t$ is non-empty. However, as $t$ increases, the number of facets may be less than $n$, namely, $\widetilde{P}_{i, t}=\emptyset$ for some $i \in \{1, \ldots, n\}$ if $t \gg 1$. Note that we use $S_t$ to represent the time slice of the set $S$ at time $t$. The intersection line of any two facets is called a ridge. Define $\mathcal{R}_{ij} = \widetilde{P}_i \cap \widetilde{P}_j, i \neq j$ to be the ridges, let $\mathcal{R}$ be the set of all ridges and denote $\widehat{P}_i$ as the projection of $\widetilde{P}_i$ on the $(t, x)$-plane. 

Define
\begin{equation*}
\underline{V}(t, z)=\max _{1 \leq i \leq n}\left\{U\left(z \cdot e_i-c_{f} t+\tau_i\right)\right\}.
\end{equation*}
Obviously, it can be concluded that $\underline{V}(t, z)$ is a subsolution of \eqref{1.1}. Now we present main results.
\begin{theorem}[Existence]\label{Theorem 2.1}
Suppose that $(F1)$ and $(F2)$ hold. Take $n$ vectors $e_i$ $(i\in\{1, \ldots, n\})$ of $\mathbb{S}^{N-1}$ such that $e_i \cdot e_0>0$ for some $e_0 \in \mathbb{S}^{N-1}$ and $e_i \neq e_j$ for $i \neq j$, and $n$ constants $\tau_i$ $(i\in\{1, \ldots, n\})$. Let $\boldsymbol{e}=\left(e_1, \ldots, e_n\right)$ and $\boldsymbol{\tau}=\left(\tau_1, \ldots, \tau_n\right)$. Then there exists an entire solution $V_{\boldsymbol{e}, \boldsymbol{\tau}}(t, z):=V\left(t, z ; e_1, \ldots, e_n, \tau_1, \ldots, \tau_n\right)$ of \eqref{1.1} satisfying
\begin{equation*}
\underline{V}(t, z)<V_{\boldsymbol{e}, \boldsymbol{\tau}}(t, z)<1, \quad \forall(t, z) \in \mathbb{R} \times \mathbb{R}^N,
\end{equation*}
and
\begin{equation*}
	\frac{\left|V_{\boldsymbol{e}, \boldsymbol{\tau}}(t, z)-\underline{V}(t, z)\right|}{\min \left\{1, e^{-v^{\star} \min_{1\leq i\leq n} \left\{\frac{z \cdot e_i-c_{f} t+\tau_i}{e_i \cdot e_{0}}\right\}}\right\}} \rightarrow 0 \quad\text{uniformly as } d((t, z), \mathcal{R}) \rightarrow+\infty,
\end{equation*} 
where $v^{\star}$ is a positive constant.

Moreover, $V_{\boldsymbol{e}, \boldsymbol{\tau}}(t, z)$ satisfies the following two properties:
\begin{description}
\item[(i)] For any $\left(t_0, z_0\right) \in \mathbb{R} \times \mathbb{R}^N$ such that $\min _{i=\{1, \ldots, n\}}\left\{z_0 \cdot e_i-c_f t_0\right\} \geq 0$,
\begin{equation*}
V_{\boldsymbol{e}, \boldsymbol{\tau}}\left(t-t_0, z-z_0\right) \geq V_{\boldsymbol{e}, \boldsymbol{\tau}}(t, z).
\end{equation*}

\item[(ii)] For fixed $\left(e_1, \ldots, e_n\right) \in \mathbb{S}^{N-1} \times \cdots \times \mathbb{S}^{N-1}$, $V_{\boldsymbol{e}, \boldsymbol{\tau}}(t, z)$ are decreasing in $\tau_i \in \mathbb{R}$ for every $i \in\{1, \ldots, n\}$.
\end{description}
\end{theorem}
\begin{remark}\label{Remark 2.2}
It can be easily verified that $V_{\boldsymbol{e}, \boldsymbol{\tau}}$ in Theorem \ref{Theorem 2.1} is a transition front connecting $0$ and $1$ with sets
\begin{equation*}
\Gamma_t:=\left\{z \in \mathbb{R}^N: \min _{1 \leq i \leq n}\left\{z \cdot e_i-c_{f} t+\tau_i\right\}=0\right\},
\end{equation*}
\begin{equation*}
\Omega_t^{+}:=\left\{z \in \mathbb{R}^N: \min _{1 \leq \leq n}\left\{z \cdot e_i-c_{f} t+\tau_i\right\} < 0\right\},
\end{equation*}
and
\begin{equation*}
	\Omega_t^{-}:=\left\{z \in \mathbb{R}^N: \min _{1 \leq \leq n}\left\{z \cdot e_i-c_{f} t+\tau_i\right\} > 0\right\}.
\end{equation*}
\end{remark}

\begin{theorem}[Uniqueness]\label{Theorem 2.3}
 Suppose that $(F1)$ and $(F2)$ hold. Let $V_{\boldsymbol{e}, \boldsymbol{\tau}}(t, z)$ be given in Theorem \ref{Theorem 2.1}. If there exists an entire solution $\widetilde{V}(t, z)$ of \eqref{1.1} satisfying
	\begin{equation*}
		\left|\widetilde{V}(t, z)-\underline{V}(t, z)\right| \rightarrow 0 \quad\text{uniformly as } d\left((t, z), \mathcal{R}\right) \rightarrow+\infty.
	\end{equation*}
	then $\widetilde{V}(t, z) \equiv V_{\boldsymbol{e}, \boldsymbol{\tau}}(t, z)$ in $\mathbb{R} \times \mathbb{R}^N$. Furthermore, $V_{\boldsymbol{e}, \boldsymbol{\tau}}(t, z)$ depend continuously on $\left(\tau_1, \ldots, \tau_n\right) \in \mathbb{R}^n$ in the sense of $\mathcal{T}$. That is, for any compact set $K \subset \mathbb{R}^N$, the functions $V_{\boldsymbol{e}, \boldsymbol{\tau}}$, $\nabla V_{\boldsymbol{e}, \boldsymbol{\tau}}$, $\nabla^2 V_{\boldsymbol{e}, \boldsymbol{\tau}}$ and $\partial_t V_{\boldsymbol{e}, \boldsymbol{\tau}}$ converge uniformly in $K$ to $V_{\boldsymbol{e}, \boldsymbol{\tau}_0}$, $\nabla V_{\boldsymbol{e}, \boldsymbol{\tau}_0}$, $\nabla^2 V_{\boldsymbol{e}, \boldsymbol{\tau}_0}$ and $\partial_t V_{\boldsymbol{e}, \boldsymbol{\tau}_0}$ as $\boldsymbol{\tau} \rightarrow \boldsymbol{\tau}_0 \in \mathbb{R}^n$.
\end{theorem}

Moreover, we prove that the curved front $V_{\boldsymbol{e}, \boldsymbol{\tau}}(t, z)$ given in Theorem \ref{Theorem 2.1}
is asymptotically stable. To this end, we consider the following
Cauchy problem:
\begin{equation}\label{2.5}
	\begin{cases}\partial_t u-\Delta u=f(u), &  t>0 , z \in \mathbb{R}^N, \\ u(t, z)=u_0 (z), &  t=0 , z \in \mathbb{R}^N.\end{cases}
\end{equation}
Then we can obtain the stability of the curved front $V_{\boldsymbol{e}, \boldsymbol{\tau}}(t, z)$ given in Theorem \ref{Theorem 2.1}.
\begin{theorem}[Stability]\label{Theorem 2.4}
 Suppose that $(F1)$ and $(F2)$ hold. Let $V_{\boldsymbol{e}, \boldsymbol{\tau}}(t, z)$ be given in Theorem \ref{Theorem 2.1}. Assume that $u_0 \in C\left(\mathbb{R}^N,[0,1]\right)$ satisfies
	\begin{equation*}
		\underline{V}(0, z) \leq u_0(z)
	\end{equation*}
	for all $z \in \mathbb{R}^N$, and
	\begin{equation*}
		\frac{\left|u_0(z)-\underline{V}(0, z)\right|}{\min \left\{1, e^{- v \min_{1\leq i\leq n} 
				\left\{\frac{z \cdot e_i+\tau_i}{e_i \cdot e_{0}} \right\}}\right\}}\rightarrow 0 \quad\text{uniformly as } 
		d\left(z, \mathcal{R}_{0}\right) \rightarrow+\infty
	\end{equation*}
	for some constant $v>0$, where $\mathcal{R}_0$ is the time slice of the ridges $\mathcal{R}$ at $t=0$. Then the solution $u(t, z)$ of Cauchy problem \eqref{2.5} for $t \geq 0$ with initial condition $u(0, z)=u_0(z)$ satisfies
	\begin{equation*}
		\lim _{t \rightarrow+\infty}\|u(t,  \cdot)-V_{\boldsymbol{e}, \boldsymbol{\tau}}(t, \cdot)\|_{L^{\infty}\left(\mathbb{R}^N\right)}=0.
	\end{equation*}
\end{theorem}

\section{Existence and uniqueness}
\noindent
This section aims to prove the proof of Theorem \ref{Theorem 2.1}, that is, we use $\left\{\left(e_i, \tau_i\right)\right\}_{i\in\{1, \ldots, n\}}$ to construct an entire solution. Since the equation \eqref{1.1} remains unchanged under coordinate rotation, without loss of generality, we assume that $e_0=(0,0, \ldots, 1)$. Define $y:= z_N$, then \eqref{1.1} can be rewritten as
\begin{equation}\label{3.1}
\partial_t u-\Delta_{x, y} u=f(u), \quad(t, x , y) \in \mathbb{R}\times\mathbb{R}^{N-1}\times \mathbb{R}.
\end{equation}
\subsection{A hypersurface with asymptotic planes}
\noindent
Take $n$ $(\geq 2)$ unit vectors $\nu_i \in \mathbb{S}^{N-2}$ and $n$ angles $\theta_i\in (0, \pi / 2]$ $(i=1, \ldots, n)$ such that $(\nu_i,\theta_i)\neq(\nu_j,\theta_j)$ for any $i \neq j$. Define $e_i=\left(\nu_i \cos \theta_i, \sin \theta_i\right)$ for all $i \in\{1, \ldots, n\}$ and recall that $\tau_i$ $(i=1,2, \ldots, n)$ are constants. Take a positive constant $\alpha$, which will be used as a scaling parameter in this paper. For any $(t, x, y) \in\mathbb{R}\times \mathbb{R}^{N-1}\times \mathbb{R}$ and all $i \in\{1, \ldots, n\}$, define
\begin{equation*}
q_i(t, x, y, \alpha)=x \cdot \nu_i \cos \theta_i+y \sin \theta_i-c_f t+\alpha \tau_i .
\end{equation*}
Associated to each $q_i$ is the hyperplane in $\mathbb{R}\times\mathbb{R}^N$,
\begin{equation*}
Q_i:=\left\{(t, x, y) \in \mathbb{R} \times \mathbb{R}^N ; q_i(t, x, y, \alpha)=0\right\}.
\end{equation*}
Note that $q_i$ is the signed distance function to $Q_i$ and $\left\{y=\psi_i(t, x, \alpha)\right\}$ is a graph in the $y$-direction, where
\begin{equation*}
\psi_i(t, x, \alpha)=-x \cdot \nu_i \cot \theta_i+\frac{c_f}{\sin \theta_i} t-\frac{\alpha}{\sin \theta_i} \tau_i .
\end{equation*}
Let $\mathcal{Q}$ be the polytope enclosed by $Q_i(i=1,2, \ldots, n)$, namely,
\begin{equation*}
\mathcal{Q}:=\left\{(t, x, y) \in \mathbb{R}\times\mathbb{R}^N ; \min _{1 \leq i \leq n} q_i(t, x, y, \alpha) \geq 0\right\}.
\end{equation*}
Define $\partial \mathcal{Q}$ is the boundary of $\mathcal{Q}$ and $\partial \mathcal{Q}$ has the form $\{y=\psi(x)\}$, where
\begin{equation*}
\psi(t, x, \alpha):=\max _{1 \leq i \leq n} \psi_i(t, x, \alpha).
\end{equation*}
Hence, $\partial \mathcal{Q}=\cup_{i=1}^n \widetilde{Q}_i$, where $\widetilde{Q}_i=\partial \mathcal{Q} \cap Q_i$ are its facets. Let $G_{i j}:= \widetilde{Q}_i \cap \widetilde{Q}_j$ $(i \neq j)$ be the ridge as the intersection of $\widetilde{Q}_i$ and $\widetilde{Q}_j$ for $i \neq j$. Define $\widehat{G}$ as the projection of $G$ on the $(t, x)$-plane, namely, 
\begin{equation*}
\widehat{G}:=\left\{(t, x)\in  \mathbb{R}\times\mathbb{R}^{N-1} ; \text { there exists one } y \in \mathbb{R} \text { such that }(t, x, y) \in G \right\}.
\end{equation*}
Let $\widehat{Q}_i$ be the projection of $\widetilde{Q}_i$ on the $(t, x)$-plane, that is,
\begin{equation*}
\begin{aligned}
	\widehat{Q}_i & :=\left\{ (t, x)\in  \mathbb{R}\times\mathbb{R}^{N-1} ; \text { there exist } y \in \mathbb{R} \text { such that }(t, x, y) \in \widetilde{Q}_i\right\} \\
	& =\left\{(t, x) \in  \mathbb{R}\times \mathbb{R}^{N-1} ; \psi_i(t, x, \alpha)=\max _{1 \leq j \leq n} \psi_j(t, x, \alpha)\right\} .
\end{aligned}
\end{equation*}
It follows from the graph properties of $\partial \mathcal{Q}$ that $\cup_{i=1}^n \partial \widehat{Q}_i=\widehat{G}$ and $\cup_{i=1}^n \widehat{Q}_i=\mathbb{R}\times\mathbb{R}^{N-1}$.
Let $y=\varphi(t, x, \alpha)$ be the surface determined by 
\begin{equation}\label{3.2}
	\sum_{i=1}^n e^{-q_i(t, x,  y, \alpha)}=1.
\end{equation}
By the implicit function theorem, we know such a function exists and $\varphi \in C^{\infty}\left(\mathbb{R}\times\mathbb{R}^{N-1}\right)$. Define $\Sigma :=\{y=\varphi(t, x, \alpha)\}$ as the graph of $\varphi$, $\hat{q}_i(t, x, \alpha):=q_i(t, x, \varphi(t, x, \alpha))$ and 
\begin{equation*}
h(t, x, \alpha):=\sum_{i, j \in\{1, \ldots, n\} ; i\neq j} e^{-\left(\hat{q}_i(t, x, \alpha)+\hat{q}_j(t, x, \alpha)\right)} \quad\text{for } (t, x)\in\mathbb{R}\times\mathbb{R}^{N-1},
\end{equation*}
which implies $h(t, x, \alpha)$ is a measurement of flatness for the graph $\Sigma$. It should be pointed out that the above construction of the hypersurface comes from \cite{H. Guo2024}. According to \cite{H. Guo2024}, the surface $\Sigma :=\{y=\varphi(t, x, \alpha)\}$ has the following properties.
\begin{lemma}\cite[Lemma 3.1]{H. Guo2024}\label{Lemma 3.1}
	The graph $\Sigma$ satisfies the following properties:
\begin{description}
\item[(i)] $\Sigma \subset \mathcal{Q}$;
\item[(ii)] $\Sigma$ stays at finite distance from $\partial \mathcal{Q}$, or equivalently $\sup _{(t, x) \in\mathbb{R}\times\mathbb{R}^{N-1}}|\varphi-\psi|<+\infty$;
\item[(iii)] there is a constant $C>0$ such that
	\begin{equation}\label{3.3}
		|\varphi(t, x, \alpha)-\psi(t, x, \alpha)| \leq C \exp \left\{-\frac{1}{C} d((t, x), \widehat{G})\right\} \text { or } C h(t, x, \alpha), \quad \forall (t, x) \in \mathbb{R}\times\mathbb{R}^{N-1}.
	\end{equation}
\end{description}
\end{lemma}
\begin{lemma}\cite[Lemma 3.2]{H. Guo2024} \label{Lemma 3.2}
There exists a constant $C_{1}>0$ independent of $\alpha$ and $\tau_i$ such that for each $i \in\{1, \ldots, n\}$,
\begin{equation*}
 	 \left|\partial_t \varphi-\frac{c_f}{\sin \theta_i}\right|+\left|\nabla \varphi+ \nu_i \cot \theta_i\right| \leq C_{1} h \quad\text{in } \widehat{Q}_i, 
\end{equation*}
\begin{equation*}
 	 \frac{1}{C_{1}} h \leq \frac{\partial_t \varphi}{\sqrt{1+|\nabla \varphi|^2}}-c_f \leq C_{1} h \quad\text{in } \mathbb{R} \times \mathbb{R}^{N-1},
\end{equation*}
 and 
\begin{equation*}
 \left|\nabla \partial_t \varphi\right|+\left|\nabla^2 \varphi\right|+\left|\nabla^3 \varphi\right| \leq C_{1} h \quad\text{in } \mathbb{R} \times \mathbb{R}^{N-1} .
\end{equation*}
\end{lemma}
As a matter of fact, one has $\min _{1 \leq i \leq n} q_i(t, x, y, \alpha) \geq 0$ for every point $(x, y) \in \Sigma$ by \eqref{3.2}. Besides, $\varphi(t, x, \alpha) \geq \psi_i(t, x, \alpha)$ for any fixed $i \in\{1, \ldots, n\}$ and $(t, x) \in \widehat{Q}_i$ by (i) of Lemma \ref{Lemma 3.1} and the definition of $\mathcal{Q}$. Thus, by definitions of $\psi(t, x, \alpha)$ and $\widehat{Q}_i$, we obtain $\varphi(t, x, \alpha) \geq \psi(t, x, \alpha)$ in $\mathbb{R}\times\mathbb{R}^{N-1}$. By \cite[Section 3.1]{H. Guo2024}, $h(t, x, \alpha)$ is decaying exponentially in all $\widehat{Q}_i$ as $d((t, x), \widehat{G}) \rightarrow+\infty$, namely, $\hat{q}_i(t, x, \alpha) \rightarrow 0$ and $\hat{q}_j(t, x, \alpha) \rightarrow+\infty$ for all $j \neq i$ and $(t, x) \in \widehat{Q}_i$ as $\operatorname{dist}((t, x), \widehat{G}) \rightarrow+\infty$. Moreover, the surface $y=\varphi(t, x, \alpha)$ has the following properties.
\subsection{Construction of supersolutions}
\noindent
We define two functions $\xi$ and $\eta$, where
\begin{equation}\label{3.7}
	\xi(t, x, y)=\frac{y-\frac{1}{\alpha} \varphi(\alpha t, \alpha x, \alpha)}{\sqrt{1+|\nabla \varphi(\alpha t, \alpha x, \alpha)|^2}} \text{ and } \eta(t, x, y)=y-\frac{1}{\alpha} \varphi(\alpha t, \alpha x, \alpha).
\end{equation}
By direct calculation, one has
\begin{equation*}
-\partial_t \xi-c_f=\frac{\partial_t \varphi}{\sqrt{1+|\nabla \varphi|^2}}-c_f+\alpha \xi \frac{\nabla \varphi \cdot \nabla \partial_t \varphi}{1+|\nabla \varphi|^2}.
\end{equation*}
Moreover, one has the following lemma.
\begin{lemma}\label{Lemma 3.3}\cite[Lemma 3.4]{H. Guo2024}
There exists a positive constant $C_{2}$ independent of $\alpha$ such that for $(t, x, y) \in \mathbb{R} \times \mathbb{R}^N$,
\begin{equation*}
	 \frac{1}{C_{2}} h(\alpha t, \alpha x, \alpha) \leq \frac{\partial_t \varphi}{\sqrt{1+|\nabla \varphi|^2}}-c_f \leq C_{2} h(\alpha t, \alpha x, \alpha),
\end{equation*}
\begin{equation*}
	\left|\frac{\alpha \nabla \varphi \cdot \nabla \partial_t \varphi}{1+|\nabla \varphi|^2} \xi\right| \leq C_{2} \alpha h(\alpha t, \alpha x, \alpha)|\xi|,
\end{equation*}
\begin{equation*}
		 \left|1-\left|\nabla_x \xi\right|^2-\left|\partial_y \xi\right|^2\right| \leq C_{2} \alpha h(\alpha t, \alpha x, \alpha)\left(|\xi|+\xi^2\right) 
\end{equation*}
and
\begin{equation*}
		 \left|\Delta_x \xi\right| \leq C_{2} \alpha h(\alpha t, \alpha x, \alpha)(1+|\xi|+\alpha|\xi|).
\end{equation*}
\end{lemma}

Define a smooth function $\omega(s)\in C^\infty(\mathbb{R})$ satisfying $\omega^{\prime}(s) \geq 0$ and
\begin{equation}\label{3.12}
	\begin{cases}\omega(s)=0,  \text { if } s \leq-1, \\ 0<\omega(s)<1,  \text { if } s \in(-1,1), \\ \omega(s)=1,  \text { if } s \geq 1.\end{cases}
\end{equation}
For the rest of this paper, we denote $\zeta:=\alpha x$ and $\zeta_i:=\alpha x_i$ for any $i \in\{1, \ldots, N-1\}$.
\begin{lemma}\label{Lemma 3.4}
   There exists a constant $\beta^*>0$ such that for any $\beta \in\left(0, 
    \beta^*\right]$ there exist positive constants $\varepsilon_0^{+}(\beta)$ and $\alpha_0^{+}(\beta, 
    \varepsilon)$ such that for any 
    $0<\varepsilon<\varepsilon_0^{+}(\beta)$ and any $0<\alpha<\alpha_0^{+}(\beta, \varepsilon)$, the function
	\begin{equation}\label{3.13}
	\overline{V}(t, x, y):=\min \left\{U(\xi)+\varepsilon h(\alpha t, \alpha x, \alpha) \times\left[U^\beta(\eta) \omega(\eta)+(1-\omega(\eta))\right], 1\right\}
    \end{equation}
	is a supersolution of Eq.\eqref{3.1}. Moreover,
 \begin{equation}\label{3.14}
 		\left|\bar{V}(t, x, y)-\underline{V}(t, x, y)\right| \leq 2\varepsilon \quad\text {uniformly as } d\left((t, x, y), 
 \mathcal{R}\right) \rightarrow+\infty,
 \end{equation}
 and
	\begin{equation}\label{3.15}
		\bar{V}(t, x, y) \geq \underline{V}(t, x, y) \quad \text{in}\  \mathbb{R}\times\mathbb{R}^N.
	\end{equation}

\end{lemma}
\begin{proof} We divide the proof into three steps.

{\it Step 1: proof of $\bar{V}(t, x, y)$ being a supersolution.} To show that $\bar{V}(t, x, y)$ is a supersolution, we only to prove the domain where $\bar{V}(t, x, y)<1$. Our approach is to find 
two numbers $X^{\prime}>1$ and $X^{\prime \prime}>1$, and prove the inequality
	\begin{equation*}
	\mathcal{L} \bar{V}(t, x, y):=\partial_t \bar{V}(t, x, y)-\Delta_{x, y} \bar{V}(t, x, y)-f\left(\bar{V}(t, x, y)\right) \geq 0, \quad \forall(t, x, y) \in \mathbb{R}\times\mathbb{R}^N,
	\end{equation*}
	in three cases $\eta>X^{\prime}$, $\eta<-X^{\prime \prime}$, and $\eta \in\left[-X^{\prime \prime}, X^{\prime}\right]$, respectively. As a matter of convenience, denote
	\begin{equation*}
	I_1:=\left(\partial_t-\Delta_{x, y}\right)U(\xi) \text { and } I_2:=\left(\partial_t-\Delta_{x, y}\right)\left(\varepsilon h(\alpha t, \alpha x, \alpha) \times U^\beta(\eta)\right).
    \end{equation*}
	Then we can obtain that
	\begin{equation}\label{3.17}
		\begin{aligned}
			I_1= &  U^{\prime} (\xi)\partial_t \xi
			- U^{\prime} (\xi)\left(\sum_{k=1}^{N-1}\partial_{x_{k}x_{k}} \xi+\partial_{yy} \xi\right)- U^{\prime\prime} (\xi)\left(\sum_{k=1}^{N-1}(\partial_{x_{k}} \xi)^2+(\partial_{y} \xi)^2\right)\\
			=& U^{\prime} (\xi)(\partial_t \xi +c_{f})- U^{\prime} (\xi)\left(\sum_{k=1}^{N-1}\partial_{x_{k}x_{k}} \xi+\partial_{yy} \xi\right)\\
			&- U^{\prime\prime} (\xi)\left(\sum_{k=1}^{N-1}(\partial_{x_{k}} \xi)^2+(\partial_{y} \xi)^2-1\right)+f(U(\xi)),
		\end{aligned}
	\end{equation}
where
\begin{align}\label{3.18}
\partial_y \xi  
=\frac{1}{\sqrt{1+|\nabla \varphi|^2}}\quad\text{and}\quad
\nabla_x \xi  =-\frac{\nabla \varphi}{\sqrt{1+|\nabla \varphi|^2}}-\alpha \frac{ \nabla^2 \varphi 
\cdot \nabla \varphi}{1+|\nabla \varphi|^2} \xi.
\end{align}
Besides, by \eqref{1.5}, one has that $\left| U^{\prime}(\xi) \right|,\left| U^{\prime}(\xi)  \xi\right|$, $\left|U^{\prime\prime}(\xi)  \xi\right|$ and 
$\left|U^{\prime\prime}(\xi)  \xi^2\right|$ are bounded for $\forall\xi \in 
\mathbb{R}$.
By calculating $I_2$, we know that
\begin{equation}\label{3.19}
	\begin{aligned}
		I_2= & \varepsilon \alpha \partial_{\alpha t}h(\alpha t, \alpha x, \alpha) U^{\beta}(\eta)+\varepsilon \beta h(\alpha t, \alpha x, \alpha) U^{\beta-1}(\eta) U^{\prime}(\eta) \partial_{t} \eta\\
		&-\varepsilon \alpha^2  U^\beta (\eta) \sum_{k=1}^{N-1}(\partial_{\zeta_{k}\zeta_{k}}h(\alpha t, \alpha x, \alpha))  -2 \varepsilon \beta \alpha  U^{\beta-1} (\eta) U^{\prime} (\eta) \times\left(\sum_{k=1}^{N-1}\partial_{\zeta_{k}}h(\alpha x) \partial_{x_{k}}\eta\right) \\
		& -\varepsilon \beta(\beta-1) h(\alpha t, \alpha x, \alpha) U^{\beta-2}(\eta) (U^{\prime}(\eta))^{2} \times\left(\sum_{k=1}^{N-1}\left(\partial_{x_k}\eta\right)^2+\left(\partial_{y}\eta\right)^2\right) \\
		& -\varepsilon \beta h(\alpha t, \alpha x, \alpha) U^{\beta-1}(\eta) U^{\prime \prime}(\eta)\times\left(\sum_{k=1}^{N-1}(\partial_{x_k}\eta)^2+(\partial_{y}\eta)^2\right)\\
		& -\varepsilon \beta h(\alpha t, \alpha x, \alpha) U^{\beta-1}(\eta) U^{\prime}(\eta)\times\left(\sum_{k=1}^{N-1}\partial_{x_k x_k}\eta+\partial_{yy}\eta\right).
	\end{aligned}
\end{equation}

\textbf{Case 1:} $\eta(t, x, y)>X^{\prime}$, where $X^{\prime}>1$ is to be chosen.
Thus, $\bar{V}(t, x, y)=U(\xi)+\varepsilon h(\alpha t, \alpha x, \alpha) U^\beta(\eta)$, then
\begin{equation*}
	\mathcal{L} \bar{V}(t, x, y)=I_1+I_2-f\left(\bar{V}\right).
\end{equation*}
Calculating the derivatives of $\eta(t, x, y)$, we get
\begin{equation}\label{3.20}
	\left\{\begin{array}{l}
		\eta_t=-\partial_{\alpha t}\varphi(\alpha t, \alpha x, \alpha),\\
		\eta_y=1,\;\eta_{y y}=0 ,\\
		\eta_{x_k}=-\partial_{\zeta_k}\varphi(\alpha t, \alpha x, \alpha) ,\\
		\eta_{x_k x_k}=-\alpha \partial_{\zeta_k \zeta_k} \varphi(\alpha t, \alpha x, \alpha) ,\\
		\sum_{k=1}^{N-1}\eta_{x_k}^2+\eta_y^2=\sum_{k=1}^{N-1}(\partial_{\zeta_k}\varphi (\alpha t, \alpha x, \alpha))^2+1.
	\end{array}\right.
\end{equation}
By virtue of \eqref{3.19} and \eqref{3.20}, it holds that
\begin{align}\label{3.21}
	I_2= & \varepsilon \alpha \partial_{\alpha t}h(\alpha t, \alpha x, \alpha) U^{\beta}(\eta)+\varepsilon \beta h(\alpha t, \alpha x, \alpha) U^{\beta-1}(\eta) U^{\prime}(\eta) (-\partial_{\alpha t}\varphi(\alpha t, \alpha x, \alpha)) \nonumber\\
	&-\varepsilon \alpha^2  U^\beta(\eta) \sum_{k=1}^{N-1}(\partial_{\zeta_{k}\zeta_{k}}h(\alpha t, \alpha x, \alpha))  -2 \varepsilon \beta \alpha  U^{\beta-1}(\eta) U^{\prime}(\eta) \times\left(\sum_{k=1}^{N-1}\partial_{\zeta_{k}}h(\alpha x) (-\partial_{\zeta_k}\varphi(\alpha t, \alpha x, \alpha))\right) \nonumber\\
	& -\varepsilon \beta(\beta-1) h(\alpha t, \alpha x, \alpha) U^{\beta-2}(\eta) (U^{\prime}(\eta))^{2} \times\left(\sum_{k=1}^{N-1}(\partial_{\zeta_k}\varphi (\alpha t, \alpha x, \alpha))^2+1\right) \nonumber\\
	& -\varepsilon \beta h(\alpha t, \alpha x, \alpha) U^{\beta-1}(\eta) U^{\prime \prime}(\eta)\times\left(\sum_{k=1}^{N-1}(\partial_{\zeta_k}\varphi (\alpha t, \alpha x, \alpha))^2+1\right)\nonumber\\
	& -\varepsilon \beta h(\alpha t, \alpha x, \alpha) U^{\beta-1}(\eta) U^{\prime}(\eta) \times\left(\sum_{k=1}^{N-1}(-\alpha \partial_{\zeta_k \zeta_k} \varphi(\alpha t, \alpha x, \alpha))\right)\nonumber\\
	=& -\varepsilon \alpha U^{\beta}(\eta) \left[\partial_{\alpha t}h(\alpha t, \alpha x, \alpha)+\alpha \sum_{k=1}^{N-1}(\partial_{\zeta_{k}\zeta_{k}}h(\alpha t, \alpha x, \alpha)) \right.\nonumber\\
	&\quad\left.-2  \beta \frac{U^{\prime}(\eta)}{U(\eta)} \left(\sum_{k=1}^{N-1}\partial_{\zeta_{k}}h(\alpha x) (\partial_{\zeta_k}\varphi(\alpha t, \alpha x, \alpha))\right) -\beta h(\alpha t, \alpha x, \alpha) \frac{U^{\prime}(\eta)}{U(\eta)}\sum_{k=1}^{N-1}\alpha \partial_{\zeta_k \zeta_k} \varphi(\alpha t, \alpha x, \alpha)\right]\nonumber\\
	&-\varepsilon\beta  U^{\beta}(\eta) h(\alpha t, \alpha x, \alpha) \left[\frac{U^{\prime}(\eta)}{U(\eta)}\partial_{\alpha t}\varphi(\alpha t, \alpha x, \alpha)+(\beta-1)\frac{(U^{\prime}(\eta))^{2}}{U^{2}(\eta)}\left(\sum_{k=1}^{N-1}(\partial_{\zeta_k}\varphi (\alpha t, \alpha x, \alpha))^2+1\right)\right.\nonumber\\
	&\left.\quad+\frac{U^{\prime \prime}(\eta)}{U(\eta)}\left(\sum_{k=1}^{N-1}(\partial_{\zeta_k}\varphi (\alpha t, \alpha x, \alpha))^2+1\right) \right]\nonumber\\
	=& : J_1+J_2.
\end{align}
Since \eqref{3.7}, Lemma \ref{Lemma 3.2} and the fact $h=\sum_{i, j \in\{1, \ldots, n\} ; i\neq j}e^{-\left(\hat{q}_i+\hat{q}_j\right)}\le 1$, one gets
\begin{equation}\label{3.22}
\xi \leq \eta=\xi \sqrt{1+|\nabla \varphi(\alpha t, \alpha x, \alpha)|^2} \leq \xi \sqrt{(C_{1}+\max_{1\le i\le n} \{|\nu_i \cot \theta_i|\})^2+1}.    
\end{equation}
It follows from \eqref{1.6} and Lemma \ref{Lemma 3.2} that there exists a large enough constant $X_1^{\prime}>1$ such that
\[
\left|\frac{U^{\prime \prime}(\eta)}{U(\eta)}\left(\sum_{k=1}^{N-1}(\partial_{\zeta_k}\varphi (\alpha t, \alpha x, \alpha))^2+1\right)-c_{f}^{2}\left(\sum_{k=1}^{N-1}(\partial_{\zeta_k}\varphi (\alpha t, \alpha x, \alpha))^2+1\right)\right|\leq\frac{c_{f}^{2}}{16},
\]
\[		
\left|c_{f}^{2}\left(\sum_{k=1}^{N-1}(\partial_{\zeta_k}\varphi (\alpha t, \alpha x, \alpha))^2+1\right)-\frac{(U^{\prime}(\eta))^{2}}{U^{2}(\eta)}\left(\sum_{k=1}^{N-1}(\partial_{\zeta_k}\varphi (\alpha t, \alpha x, \alpha))^2+1\right)\right|\leq\frac{c_{f}^{2}}{16},
\]
\[		
\left|\frac{(U^{\prime}(\eta))^{2}}{U^{2}(\eta)}\left(\sum_{k=1}^{N-1}(\partial_{\zeta_k}\varphi (\alpha t, \alpha x, \alpha))^2+1\right)\right|\leq  2 c_{f}^{2}\left(\sum_{k=1}^{N-1}(\partial_{\zeta_k}\varphi (\alpha t, \alpha x, \alpha))^2+1\right)
\]
and
\[
\begin{aligned}
	- c_{f}-\frac{c_{f}}{16}<\frac{U^{\prime}(\eta)}{U(\eta)}< -c_{f}+\frac{c_{f}}{16},
\end{aligned}
\]
for all $\eta \in\left(X_1^{\prime},+\infty\right)$.
Then it is easy to get that
\begin{equation*}
	\begin{aligned}
		&\quad\left|\frac{U^{\prime \prime}(\eta)}{U(\eta)}\left(\sum_{k=1}^{N-1}(\partial_{\zeta_k}\varphi (\alpha t, \alpha x, \alpha))^2+1\right)-\frac{(U^{\prime}(\eta))^{2}}{U^{2}(\eta)}\left(\sum_{k=1}^{N-1}(\partial_{\zeta_k}\varphi (\alpha t, \alpha x, \alpha))^2+1\right)\right|\\
		&\leq\left|\frac{U^{\prime \prime}(\eta)}{U(\eta)}\left(\sum_{k=1}^{N-1}(\partial_{\zeta_k}\varphi (\alpha t, \alpha x, \alpha))^2+1\right)-c_{f}^{2}\left(\sum_{k=1}^{N-1}(\partial_{\zeta_k}\varphi (\alpha t, \alpha x, \alpha))^2+1\right)\right|\\
		&\quad+\left|c_{f}^{2}\left(\sum_{k=1}^{N-1}(\partial_{\zeta_k}\varphi (\alpha t, \alpha x, \alpha))^2+1\right)-\frac{(U^{\prime}(\eta))^{2}}{U^{2}(\eta)}\left(\sum_{k=1}^{N-1}(\partial_{\zeta_k}\varphi (\alpha t, \alpha x, \alpha))^2+1\right)\right|\leq 2\times\frac{c_{f}^{2}}{16}=\frac{c_{f}^{2}}{8}.
	\end{aligned}
\end{equation*}

Define
\begin{equation*}
	\beta_1^*:=\frac{1}{4 \left((C_{1}+\max_{1\le i\le n} \{|\nu_i \cot \theta_i|\})^2+1\right)},
\end{equation*}
then for any $\beta \in\left(0, \beta_1^*\right]$, one has
\begin{equation}\label{3.23}
	\begin{aligned}
		J_2&\ge-\varepsilon \beta U^\beta (\eta) h(\alpha t, \alpha x, \alpha)\left[\left(-c_{f}+\frac{c_{f}}{16}\right)c_{f}+ 2 \beta c_{f}^2\left(\sum_{k=1}^{N-1}(\partial_{\zeta_k}\varphi(\alpha t, \alpha x, \alpha))^2+1\right)+\frac{c_{f}^{2}}{8}\right]\\
		&\ge -\varepsilon \beta U^\beta (\eta) h(\alpha t, \alpha x, \alpha)\left(\frac{c_{f}^{2}}{2 }-\frac{13}{16 }c_{f}^{2} \right)
		> \varepsilon  U^\beta  (\eta) h(\alpha t, \alpha x, \alpha)\times \beta\frac{c_{f}^{2}}{4 }
	\end{aligned}
\end{equation}
for all $\eta \in\left(X_1^{\prime},+\infty\right)$. By Lemma \ref{Lemma 3.3}, there exists a constant $C_3>0$ such that
	\begin{align}\label{3.24}
		\left|\nabla h(\alpha t, \alpha x, \alpha)\right| &=\left|\sum_{i, j \in\{1, \ldots, n\}, i \neq j} 
		e^{-\left(\hat{q}_j+\hat{q}_k\right)} \mid\nu_i \cos \theta_i+\nu_j \cos 
		\theta_j+\nabla \varphi(\alpha t, \alpha x, \alpha)\left(\sin \theta_i+\sin \theta_j\right) \mid\right| \\
		& \leq  C_3 h(\alpha t, \alpha x, \alpha) \nonumber
	\end{align}
and
\begin{equation}\label{3.25}
	\begin{aligned}
		& \left|\Delta h(\alpha t, \alpha x, \alpha)\right| \\
		& =\left|\sum_{i, j \in\{1, \ldots, n\}, i \neq j} e^{-\left(\hat{q}_j+\hat{q}_k\right)}\left|\nu_i \cos \theta_i+\nu_j \cos \theta_j+\nabla \varphi(\alpha t, \alpha x, \alpha)\left(\sin \theta_i+\sin \theta_j\right)\right|^2 \right.\\
		&\quad\left.+\sum_{i, j \in\{1, \ldots, n\}, i \neq j} e^{-\left(\hat{q}_j+\hat{q}_k\right)} \Delta \varphi(\alpha t, \alpha x, \alpha)\left(\sin \theta_i+\sin \theta_j\right)\right|\\
		&\leq  C_3 h(\alpha t, \alpha x, \alpha) .
	\end{aligned}
\end{equation}
Besides, by the definition of $h$, one has $\partial_t h(t, x, \alpha)=2 c_{f} h(t, x, \alpha)$. By \eqref{3.21} and \eqref{3.23}-\eqref{3.25}, there is a sufficiently small constant $\alpha_1^{+}(\beta)>0$ 
such that, for arbitrary $0<\alpha \leq \alpha_1^{+}(\beta)$, it holds that
\begin{equation*}
	\begin{aligned}
	J_1 \geq & -\varepsilon \alpha U^{\beta} (\eta) \left[2 c_{f} h(\alpha t, \alpha x, \alpha)+\alpha  \left|\Delta h(\alpha t, \alpha x, \alpha)\right| +2  \beta \left|\frac{U^{\prime}(\eta)}{U(\eta)}\right| \left(\left|\nabla h(\alpha t, \alpha x, \alpha)\right| \cdot \left|\nabla \varphi(\alpha t, \alpha x, \alpha)\right|\right)\right.\\
	&\quad\left. +\beta \alpha h(\alpha t, \alpha x, \alpha) \left|\frac{U^{\prime}(\eta)}{U(\eta)}\right|  \left|\Delta \varphi(\alpha t, \alpha x, \alpha)\right|\right]\\
	> & - \varepsilon  U^\beta (\eta) h(\alpha t, \alpha x, \alpha)\times \beta\frac{c_{f}^{2}}{8 }
    \end{aligned}
\end{equation*}
and
\begin{equation}\label{3.26}
	I_2=J_1+J_2>\varepsilon  U^\beta (\eta) h(\alpha t, \alpha x, \alpha)\times \beta\frac{c_{f}^{2}}{8 }, \quad\forall\eta \in\left(X_1^{\prime},+\infty\right).
\end{equation}
By virtue of 
$U^{\prime}< 0$, \eqref{1.5}, \eqref{3.17}, \eqref{3.18}, \eqref{3.20}, Lemma \ref{Lemma 3.2} and Lemma \ref{Lemma 3.3}, we have
\begin{equation*}
(\partial_t \xi+c_{f}) U^{\prime}(\xi)\ge  U^{\prime}(\xi) \left(-\frac{1}{C_{2}} h(\alpha t, \alpha x, \alpha)+C_{2} \alpha h(\alpha t, \alpha x, \alpha)|\xi| \right)\ge 
-C_{2} L_2 \alpha h(\alpha t, \alpha x, \alpha)|\xi|  e^{-c_{f} \xi},
\end{equation*}
\begin{equation*}
	-U^{\prime} (\xi)\left(\sum_{k=1}^{N-1}\partial_{x_{k}x_{k}} \xi+\partial_{yy} 
	\xi\right)\ge-\left|U^{\prime} (\xi)\Delta_{x, y}\xi\right|\ge-C_{2} L_2  \alpha h(\alpha t, \alpha x, \alpha)(1+|\xi|+\alpha|\xi|)e^{-c_{f} \xi}
\end{equation*}
and
\begin{equation*}
	-U^{\prime\prime}(\xi)\left(\sum_{k=1}^{N-1}(\partial_{x_{k}} \xi)^2+(\partial_{y} \xi)^2-1\right) \ge- C_{2} L_2  \alpha h(\alpha t, \alpha x, \alpha)\left(|\xi|+\xi^2\right) e^{-c_{f} \xi}
\end{equation*}
for all $\xi \in[0,+\infty)$. According to  \eqref{3.17}, if we set $0<\alpha<1$, then there is a constant $\Lambda_1>0$ such that
\begin{equation}\label{3.27}
	I_1 \geq-\alpha\Lambda_1 h(\alpha t, \alpha x, \alpha) e^{-\frac{c_{f}}{2} \xi}+f\left(U(\xi)\right) , \quad \forall\xi \in[0,+\infty).
\end{equation}
It follows from \eqref{1.5} that there exists a large enough constant $X_2^{\prime}>1$ such that
\begin{equation}\label{3.28}
U(\eta) \geq \frac{L_1}{2} e^{-c_{f} \eta}, \quad \forall\eta \in\left(X_2^{\prime},+\infty\right).
\end{equation}
Set
\begin{equation*}
	\beta_2^*:=\frac{1}{4 \sqrt{(C_1 +\max_{1\le i\le n} \{|\nu_i \cot \theta_i|\})^2+1}}.
\end{equation*}
Since \eqref{3.22} and \eqref{3.28}, for any $\beta \in\left(0, \beta_2^*\right]$, we obtain 
\begin{align}\label{3.29}
U^\beta(\eta) \geq \left(\frac{L_1}{2}\right)^\beta e^{-\beta c_{f} \eta}\geq \left(\frac{L_1}{2}\right)^\beta e^{- \frac{c_{f} \eta}{4 \sqrt{(C_1 +\max_{1\le i\le n} \{|\nu_i \cot \theta_i|\})^2+1}}}\geq\left(\frac{L_1}{2}\right)^\beta e^{-\frac{c_{f}}{4} \xi}
\end{align}
for all $\eta \in\left(X_2^{\prime},+\infty\right)$. By virtue of \eqref{1.2}, \eqref{3.22} and $\bar{V}(t, x, y)=U(\xi)+\varepsilon h(\alpha t, \alpha x, \alpha) U^\beta(\eta)$, if set $\varepsilon<\varepsilon_{1}^{+}:=\frac{\gamma_{\star}}{2}$, then there exists a sufficiently large constant $X_3^{\prime}>1$ such that $U(\xi), \bar{V} \leq 2 \gamma_{\star}$ for $\xi> X_{3}^{\prime}$. Therefore
\begin{align}\label{3.30}
	f\left(U(\xi)\right)-f\left(\bar{V}\right)= 0,\quad\forall\xi \in\left(X_3^{\prime},+\infty\right).
\end{align}
Let $X^{\prime}\geq \max\{X_1^{\prime},X_2^{\prime},X_3^{\prime}\sqrt{(C_1 +\max_{1\le i\le n} \{|\nu_i \cot \theta_i|\})^2+1}\}$ be large enough and $\beta^{*}:=\min\{\beta_1^*, \beta_2^*\}$. By \eqref{3.26}-\eqref{3.30}, if $\alpha<\alpha^{+}_2(\beta, \varepsilon):=\varepsilon \beta\frac{c_{f}^{2}}{8 \Lambda_1} \left(\frac{L_1}{2}\right)^\beta $, one gets
\begin{equation*}
\begin{aligned}
	\mathcal{L} \bar{V} & =I_1+I_2-f\left(\bar{V}\right) \\
	& \geq-\alpha\Lambda_1 h(\alpha t, \alpha x, \alpha) e^{-\frac{c_{f}}{2} \xi}+\varepsilon  U^\beta (\eta) h(\alpha t, \alpha x, \alpha)\times \beta\frac{c_{f}^{2}}{8 } \\
	& \geq-\alpha\Lambda_1 h(\alpha t, \alpha x, \alpha) e^{-\frac{c_{f}}{2} \xi}+\varepsilon \beta\frac{c_{f}^{2}}{8 } \left(\frac{L_1}{2}\right)^\beta h(\alpha t, \alpha x, \alpha)  e^{-\frac{c_{f}}{4} \xi} >0, \quad \forall \eta\in\left(X^{\prime},+\infty\right).
\end{aligned}
\end{equation*}

\textbf{Case 2:} $\eta(t, x, y)<-X^{\prime\prime}$, where $X^{\prime\prime}>1$ is to be chosen.
There is $\bar{V}(t, x, y)=U(\xi)+\varepsilon h(\alpha t, \alpha x, \alpha)$. It follows from \eqref{3.17}, $U^{\prime}< 0$ and Lemma \ref{Lemma 3.3} that
\begin{equation}\label{3.31}
	\begin{aligned}
		\mathcal{L} \bar{V}= & I_1-\varepsilon \alpha^2   \sum_{k=1}^{N-1}\partial_{\zeta_{k}\zeta_{k}}h(\alpha t, \alpha x, \alpha)-f\left(\bar{V}\right) \\
		&\ge U^{\prime} (\xi)(\partial_t \xi +c_{f})- U^{\prime} (\xi)\left(\sum_{k=1}^{N-1}\partial_{x_{k}x_{k}} \xi+\partial_{yy} \xi\right)- U^{\prime\prime} (\xi)\left(\sum_{k=1}^{N-1}(\partial_{x_{k}} \xi)^2+(\partial_{y} \xi)^2-1\right)\\
		&\quad-\varepsilon \alpha^2 \sum_{k=1}^{N-1}\partial_{\zeta_{k}\zeta_{k}}h(\alpha t, \alpha x, \alpha)+f\left(U(\xi)\right)-f\left(\bar{V}\right).
	\end{aligned}
\end{equation} 
Define $\varepsilon_2^{+}:=\gamma_{\star} / 6$, where $\gamma_{\star}$ is defined in 
\eqref{1.4}. By \eqref{1.5} and \eqref{1.4}, for any $\varepsilon<\varepsilon_2^{+}$, there exists a sufficiently large positive constant  
$X^{\prime \prime}_{1}>1$ such that $U(\xi), \bar{V} \in [1-2\gamma_{\star},1+2\gamma_{\star}]$ for $\xi<-X^{\prime\prime}_{1}$, then
\begin{equation}\label{3.32}
	f\left(U(\xi)\right)-f\left(\bar{V}\right) \geq -\frac{f^{\prime}(1)}{2}  \varepsilon h(\alpha t, \alpha x, \alpha),\quad \forall\xi\in\left(-\infty,-X^{\prime \prime}_{1}\right).
\end{equation}
Let $X^{\prime \prime}\geq X^{\prime \prime}_{1} \sqrt{(C_1 +\max_{1\le i\le n} \{|\nu_i \cot \theta_i|\})^2+1}$ be large enough. It follows from \eqref{1.5}, \eqref{3.18}, \eqref{3.34}, \eqref{3.25}, 
\eqref{3.31}, \eqref{3.32} and Lemma \ref{Lemma 3.2} that there exists a constant $\Lambda_2>0$ such that
\begin{equation}\label{3.33}
\begin{aligned}
	\mathcal{L} \bar{V}&\geq -\Lambda_2 \alpha h(\alpha t, \alpha x, \alpha) -\varepsilon \alpha^2   \sum_{k=1}^{N-1}\partial_{\zeta_{k}\zeta_{k}}h(\alpha t, \alpha x, \alpha)-\frac{f^{\prime}(1)}{2} \varepsilon h(\alpha t, \alpha x, \alpha)\\
	&\geq h(\alpha t, \alpha x, \alpha)\left( -\Lambda_2 \alpha-\frac{f^{\prime}(1)}{2} \varepsilon -\varepsilon \alpha^2 C_3\right)  
\end{aligned}
\end{equation}
for all $\eta \in\left(-\infty,-X^{\prime \prime}\right)$. Obviously, there exists a constant $\alpha_3^{+}(\varepsilon)>0$ such that
\begin{equation}\label{3.34}
	-\Lambda_2 \alpha-\frac{f^{\prime}(1)}{2} \varepsilon -\varepsilon \alpha^2 C_3 > 0, \quad\forall \alpha \in\left(0, \alpha_3^{+}(\varepsilon)\right).
\end{equation}
Letting $\varepsilon_0^{+} < \varepsilon_2^{+}$ and $\alpha_0^{+} < \alpha_3^{+}(\varepsilon)$, one gets $\mathcal{L} \bar{V}>0$ in Case $2$ by \eqref{3.33} and \eqref{3.34}.

\textbf{Case 3:} $-X^{\prime \prime} \leq \eta(t, x, y) \leq X^{\prime}$.
\begin{align*}
	&\left(\partial_t-\Delta_{x, y}\right)\left(\varepsilon h(\alpha t, \alpha x, \alpha)\left[U^\beta(\eta) \omega(\eta)+(1-\omega(\eta))\right]\right)\\
	&=\varepsilon \alpha \partial_{\alpha t}h(\alpha t, \alpha x, \alpha)\left[U^\beta (\eta) \omega(\eta)+(1-\omega(\eta))\right]+\varepsilon \beta h(\alpha t, \alpha x, \alpha) U^{\beta-1} (\eta) U^{\prime}(\eta) \omega(\eta)\partial_{t}\eta\\
	&\quad+\varepsilon h(\alpha t, \alpha x, \alpha)U^{\beta}(\eta)\omega^{\prime}(\eta)\partial_{t}\eta +\varepsilon h(\alpha t, \alpha x, \alpha)(-\omega^{\prime}(\eta))\partial_{t}\eta\\
	&\quad-\varepsilon\alpha^{2}\left(\sum_{k=1}^{N-1}\partial_{\zeta_{k}\zeta_{k}}h(\alpha t, \alpha x, \alpha)\right)\left[U^\beta (\eta) \omega(\eta)+(1-\omega(\eta))\right]\\
	&\quad-2\varepsilon\alpha\beta U^{\beta-1}U^{\prime}(\eta) \omega(\eta)\left(\sum_{k=1}^{N-1}\partial_{\zeta_{k}}h(\alpha t, \alpha x, \alpha)\partial_{x_k}\eta\right)-2\varepsilon\alpha U^{\beta}(\eta)\omega^{\prime}(\eta)\left(\sum_{k=1}^{N-1}\partial_{\zeta_{k}}h(\alpha t, \alpha x, \alpha)\partial_{x_k}\eta\right)\\
	&\quad+2\varepsilon\alpha \omega^{\prime}(\eta)\left(\sum_{k=1}^{N-1}\partial_{\zeta_{k}}h(\alpha t, \alpha x, \alpha)\partial_{x_k}\eta\right)\\
	&\quad-\varepsilon \beta(\beta-1) h(\alpha t, \alpha x, \alpha)U^{\beta-2}(\eta)(U^{\prime}(\eta))^2\omega(\eta)\left(\sum_{k=1}^{N-1}(\partial_{x_k}\eta)^2+(\partial_{y}\eta)^2\right)\\
	&\quad-\varepsilon\beta h(\alpha t, \alpha x, \alpha) U^{\beta-1}(\eta)U^{\prime\prime}(\eta)\omega(\eta)\left(\sum_{k=1}^{N-1}(\partial_{x_k}\eta)^2+(\partial_{y}\eta)^2\right)\\
	&\quad-\varepsilon\beta h(\alpha t, \alpha x, \alpha) U^{\beta-1}(\eta)U^{\prime}(\eta)\omega(\eta)\left(\sum_{k=1}^{N-1}\partial_{x_k x_k}\eta+\partial_{yy}\eta\right)\\
	&\quad-2\varepsilon\beta h(\alpha t, \alpha x, \alpha) 
	U^{\beta-1}(\eta)U^{\prime}(\eta)\omega^{\prime}(\eta)\left(\sum_{k=1}^{N-1}(\partial_{x_k}\eta)^2
	+(\partial_{y}\eta)^2\right)\\
&\quad-\varepsilon h(\alpha t, \alpha x, \alpha) 
U^{\beta}(\eta)\omega^{\prime\prime}(\eta)\left(\sum_{k=1}^{N-1}
(\partial_{x_k}\eta)^2+(\partial_{y}\eta)^2\right)\\
&\quad-\varepsilon h(\alpha t, \alpha x, \alpha) 
U^{\beta}(\eta)\omega^{\prime}(\eta)\left(\sum_{k=1}^{N-1}
\partial_{x_k x_k}\eta+\partial_{yy}\eta\right) \qquad\qquad\qquad\\
	&\quad+\varepsilon h(\alpha t, \alpha x, \alpha) 
	\omega^{\prime\prime}(\eta)\left(\sum_{k=1}^{N-1}(\partial_{x_k}\eta)^2
+(\partial_{y}\eta)^2\right)+\varepsilon h(\alpha t, \alpha x, \alpha) 
\omega^{\prime}(\eta)\left(\sum_{k=1}^{N-1}\partial_{x_k x_k}\eta
+\partial_{yy}\eta\right).
\end{align*}
Since $\eta$ is bounded, $h(\alpha t, \alpha x, \alpha)\le 1$, \eqref{1.5}, \eqref{3.7}, \eqref{3.12}, 
\eqref{3.17}, \eqref{3.18}, \eqref{3.14}, Lemma \ref{Lemma 3.2} and Lemma \ref{Lemma 3.3}, one has
\[
\varepsilon \alpha \partial_{\alpha t}h(\alpha t, \alpha x, \alpha)\left[U^\beta (\eta) \omega(\eta)+(1-\omega(\eta))\right]= 2 \varepsilon \alpha c_{f} h(\alpha t, \alpha x, \alpha)\left[U^\beta(\eta) \omega(\eta)+(1-\omega(\eta))\right]\geq 0,
\]
\[\varepsilon\beta  h(\alpha t, \alpha x, \alpha)
U^{\beta-1}(\eta)U^{\prime}(\eta)\omega(\eta)\partial_{t}\eta=\varepsilon\beta h(\alpha t, \alpha x, \alpha) U^{\beta-1}(\eta)U^{\prime}(\eta)\omega(\eta)\left(-\partial_{\alpha t}\varphi\right)>0,
\]
\[
\begin{aligned}
	\varepsilon h(\alpha t, \alpha x, \alpha)U^{\beta}(\eta)\omega^{\prime}(\eta)\partial_{t}\eta+\varepsilon h(\alpha t, \alpha x, \alpha)(-\omega^{\prime}(\eta))\partial_{t}\eta=	\varepsilon h(\alpha t, \alpha x, \alpha)(U^{\beta}(\eta) -1)\omega^{\prime}(\eta)\left(-\partial_{\alpha t}\varphi\right)\geq 0,
\end{aligned}
\]
\[
-\varepsilon\alpha^{2}\left(\sum_{k=1}^{N-1}\partial_{\zeta_{k}\zeta_{k}}h(\alpha t, \alpha x, \alpha)\right)\left[U^\beta (\eta)
\omega(\eta)+(1-\omega(\eta))\right]\ge-\varepsilon\alpha^{2} C_3 h(\alpha t, \alpha x, \alpha),
\]
\[\begin{aligned}
&-2\varepsilon\alpha\beta 
U^{\beta-1}(\eta) U^{\prime}(\eta)\omega(\eta)\left(\sum_{k=1}^{N-1}\partial_{\zeta_{k}}h(\alpha t, \alpha x, \alpha)\partial_{x_k}\eta\right) \\
&\quad\ge-2\varepsilon\alpha\beta  C_3 h(\alpha t, \alpha x, \alpha) (C_1 h(\alpha t, \alpha x, \alpha)+\max_{1\le i\le n} \{|\nu_i \cot \theta_i|\}) \max\{L_2 e^{-c_{f} |\eta|}, L_3 e^{-\beta_{0} |\eta|}\},
\end{aligned}\]
\[
\begin{aligned}
&-2\varepsilon\alpha U^{\beta}(\eta) \omega^{\prime}(\eta)\left(\sum_{k=1}^{N-1}\partial_{\zeta_{k}}h(\alpha t, \alpha x, \alpha)\partial_{x_k}\eta\right)\\
&\quad\ge -2 \varepsilon\alpha C_3 h(\alpha t, \alpha x, \alpha)\omega^{\prime}(\eta) (C_1 h(\alpha t, \alpha x, \alpha)+\max_{1\le i\le n} \{|\nu_i \cot \theta_i|\}),
\end{aligned}
\]
\[
\begin{aligned}
&2\varepsilon\alpha \omega^{\prime}(\eta)\left(\sum_{k=1}^{N-1}\partial_{\zeta_{k}}h(\alpha t, \alpha x, \alpha)\partial_{x_k}\eta\right)\\
&\quad\ge -2\varepsilon\alpha C_3 h(\alpha t, \alpha x, \alpha)\omega^{\prime}(\eta) (C_1 h(\alpha t, \alpha x, \alpha)+\max_{1\le i\le n} \{|\nu_i \cot \theta_i|\}),
\end{aligned}
\]
\[
\begin{aligned}
&-\varepsilon\beta(\beta-1) h(\alpha t, \alpha x, \alpha)U^{\beta-2}(\eta)(U^{\prime}(\eta))^2\omega(\eta)
\left(\sum_{k=1}^{N-1}(\partial_{x_k}\eta)^2+(\partial_{y}\eta)^2\right)\\
&\quad\ge-\varepsilon|\beta(\beta-1)| h(\alpha t, \alpha x, \alpha) \left(L_2 e^{-c_{f} |\eta|}\right)^{2} \left[1+(C_1 h(\alpha t, \alpha x, \alpha)+\max_{1\le i\le n} \{|\nu_i \cot \theta_i|\})^2\right],
\end{aligned}
\]
\[
\begin{aligned}
&-\varepsilon\beta h(\alpha t, \alpha x, \alpha) U^{\beta-1}(\eta) U^{\prime \prime}(\eta) \omega(\eta)\left(\sum_{k=1}^{N-1}(\partial_{x_k}\eta)^2+(\partial_{y}\eta)^2\right)\\
&\quad\ge-\varepsilon \beta h(\alpha t, \alpha x, \alpha) \max\{L_2 e^{-c_{f} |\eta|}, L_3 e^{-\beta_{0} |\eta|}\}\left[1+(C_1 h(\alpha x)+\max_{1\le i\le n} \{|\nu_i \cot \theta_i|\})^2\right],
\end{aligned}
\]
\[
\begin{aligned}
&-\varepsilon\beta h(\alpha t, \alpha x, \alpha) U^{\beta-1}(\eta) U^{\prime}(\eta)\omega(\eta)\left(\sum_{k=1}^{N-1}\partial_{x_k x_k}\eta+\partial_{yy}\eta\right)\\
&\quad\ge-\varepsilon\alpha \beta   (N-1)C_1 h(\alpha t, \alpha x, \alpha)\max\{L_2 e^{-c_{f} |\eta|}, L_3 e^{-\beta_{0} |\eta|}\} ,
\end{aligned}
\]
\[
\begin{aligned}
&-2\varepsilon\beta h(\alpha t, \alpha x, \alpha) U^{\beta-1}(\eta) U^{\prime}(\eta) \omega^{\prime}(\eta)\left(\sum_{k=1}^{N-1}(\partial_{x_k}\eta)^2+(\partial_{y}\eta)^2\right)\\
&\quad\ge-2\varepsilon\beta  h(\alpha t, \alpha x, \alpha) \omega^{\prime}(\eta)\max\{L_2 e^{-c_{f} |\eta|}, L_3 e^{-\beta_{0} |\eta|}\} \left[1+(C_1 h(\alpha x)+\max_{1\le i\le n} \{|\nu_i \cot \theta_i|\})^2\right],
\end{aligned}
\]
\[
\begin{aligned}
	&-\varepsilon h(\alpha t, \alpha x, \alpha) U^{\beta}(\eta) \omega^{\prime\prime}(\eta)\left(\sum_{k=1}^{N-1}(\partial_{x_k}\eta)^2+(\partial_{y}\eta)^2\right)\\
	&\quad\ge-\varepsilon h(\alpha t, \alpha x, \alpha)|\omega^{\prime\prime}(\eta)|\left[1+(C_1 h(\alpha x)+\max_{1\le i\le n} \{|\nu_i \cot \theta_i|\})^2\right],
\end{aligned}
\]
\[
\begin{aligned}
	-\varepsilon h(\alpha t, \alpha x, \alpha) U^{\beta}(\eta) \omega^{\prime}(\eta)\left(\sum_{k=1}^{N-1}\partial_{x_k x_k}\eta+\partial_{yy}\eta\right)\ge-\varepsilon \alpha N C_{1} \omega^{\prime}(\eta) h(\alpha t, \alpha x, \alpha) ,
\end{aligned}
\]
\[
\begin{aligned}
	&\varepsilon h(\alpha t, \alpha x, \alpha) \omega^{\prime\prime}(\eta)\left(\sum_{k=1}^{N-1}(\partial_{x_k}\eta)^2+(\partial_{y}\eta)^2\right)\\
	&\quad\ge-\varepsilon h(\alpha t, \alpha x, \alpha)|\omega^{\prime\prime}(\eta)|\left[1+(C_1 h(\alpha x)+\max_{1\le i\le n} \{|\nu_i \cot \theta_i|\})^2\right]
\end{aligned}
\]
and
\[
\begin{aligned}
\varepsilon h(\alpha t, \alpha x, \alpha) \omega^{\prime}(\eta)\left(\sum_{k=1}^{N-1}\partial_{x_k x_k}\eta+\partial_{yy}\eta\right)\ge-\varepsilon \alpha N C_{2} \omega^{\prime}(\eta) h(\alpha t, \alpha x, \alpha).
\end{aligned}
\]
Hence, there exists a constant $\Lambda_3>0$ such that for any $\eta \in \left[-X^{\prime \prime}, X^{\prime}\right]$,
\begin{align}\label{3.35}
	\left(\partial_t-\Delta_{x, y}\right)\left(\varepsilon h(\alpha t, \alpha x, \alpha)\left[U^\beta(\eta) \omega(\eta)+(1-\omega(\eta))\right]\right)>-\Lambda_3 \varepsilon h(\alpha t, \alpha x, \alpha).
\end{align}
It follows from \eqref{3.17} and \eqref{3.35} that
\begin{equation}\label{3.36}
\begin{aligned}
	\mathcal{L} \bar{V}> & I_1-\Lambda_3 \varepsilon h(\alpha t, \alpha x, \alpha)-f\left(\bar{V}\right) \\
	= & U^{\prime} (\xi)(\partial_t \xi +c_{f})- U^{\prime} (\xi)\left(\sum_{k=1}^{N-1}\partial_{x_{k}x_{k}} \xi+\partial_{yy} \xi\right)- U^{\prime\prime} (\xi)\left(\sum_{k=1}^{N-1}(\partial_{x_{k}} \xi)^2+(\partial_{y} \xi)^2-1\right)\\
	&-\Lambda_3 \varepsilon h(\alpha t, \alpha x, \alpha)+f(U(\xi))-f\left(\bar{V}\right)
\end{aligned}
\end{equation}
for all $\xi \in\left[-X^{\prime \prime}, X^{\prime}\right]$. By the 
boundedness of $\xi$, \eqref{1.5}, \eqref{3.12}, \eqref{3.18}, Lemma \ref{Lemma 3.2} and Lemma \ref{Lemma 3.3}, we obtain
\[
U^{\prime} (\xi)(\partial_t \xi +c_{f})\geq U^{\prime} (\xi) \left(-\frac{1}{C_{2}} h(\alpha t, \alpha x, \alpha)+C_{2} \alpha h(\alpha t, \alpha x, \alpha)|\xi|\right)
\]
\[
\begin{aligned}
-U^{\prime} (\xi)\left(\sum_{k=1}^{N-1}\partial_{x_{k}x_{k}} \xi+\partial_{yy} \xi\right)\ge - C_{2} \alpha h(\alpha t, \alpha x, \alpha)\left|U^{\prime}(\xi)\right| (1+|\xi|+\alpha|\xi|),
\end{aligned}
\]
and
\[
\begin{aligned}
-U^{\prime\prime} (\xi)\left(\sum_{k=1}^{N-1}(\partial_{x_{k}} \xi)^2+(\partial_{y} \xi)^2-1\right)\ge-C_{2} \alpha h(\alpha t, \alpha x, \alpha)\left|U^{\prime\prime}(\xi)\right|\left(|\xi|+\xi^2\right).
\end{aligned}
\]
Thus there exists a constant $\Lambda_4>0$ such that 
\begin{equation}\label{3.37}
\begin{aligned}
\mathcal{L} \bar{V}> &-\Lambda_4 \alpha h(\alpha t, \alpha x, \alpha)-U^{\prime}(\xi)\frac{1}{C_{2}} h(\alpha t, \alpha x, \alpha)-\Lambda_3 \varepsilon h(\alpha t, \alpha x, \alpha)+f(U(\xi))-f\left(\bar{V}\right)
\end{aligned}
\end{equation}
for all $\xi\in\left[-X^{\prime \prime}, X^{\prime}\right]$. By \eqref{1.2} there is a constant $\Lambda_5>0$ such that
\begin{equation}\label{3.38}
f\left(U(\xi)\right)-f\left(\bar{V}\right)>-\Lambda_5 \varepsilon h(\alpha t, \alpha x, \alpha).
\end{equation}
It follows from \eqref{3.36}-\eqref{3.38} that
\begin{equation}\label{3.39}
\mathcal{L} \bar{V}>h(\alpha t, \alpha x, \alpha) \times\left[\frac{1}{C_{2}}\left|U^{\prime}(\xi)\right|-\Lambda_4 \alpha-\Lambda_5 \varepsilon-\Lambda_3 \varepsilon\right], \quad \forall\xi\in\left[-X^{\prime \prime}, X^{\prime}\right].
\end{equation}
Denote
\begin{equation}\label{3.40}
\alpha_4^{+}(\varepsilon):=\varepsilon \;\text { and }\; \varepsilon_3^{+}(\beta):=\frac{1}{C_{2}(\Lambda_3+\Lambda_4+\Lambda_5)}\min_{\xi\in\left[-X^{\prime \prime}, X^{\prime}\right]}\left|U^{\prime}(\xi)\right|.
\end{equation}
By \eqref{3.39} and \eqref{3.40}, set $\varepsilon_0^{+} \leq \varepsilon_3^{+}(\beta)$ and $\alpha_0^{+} \leq \alpha_4^{+}(\varepsilon)$, then $\mathcal{L} \bar{V}>0$ holds in Case $3$. 

{\it Step 2: proof of \eqref{3.14}.}  We consider two cases.

\textbf{Case 1:} For $(t, x, y)\in\mathbb{R}\times\mathbb{R}^{N}$ such that $d((t, x, y), \partial \mathcal{P})\geq \rho$ as $\rho\rightarrow+\infty$, fix an $\alpha \in(0, \alpha_0^{+})$ and suppose that there exists an $\varepsilon \in\left(0, \varepsilon_0 \right)$ and a sequence of points $\left(t_k, x_k, y_k\right)$ such that $d\left(\left(t_k, x_k, y_k\right), \partial \mathcal{P}\right) \rightarrow+\infty$ and
\begin{equation}\label{3.41}
	\left|\bar{V}\left(t_k, x_k, y_k\right)-\underline{V}\left(t_k, x_k, y_k\right)\right|> 2 \varepsilon .
\end{equation}
 
If the sequence $\left(t_k, x_k, y_k\right)$ satisfies
\begin{equation*}
\min _{1 \leq i \leq n}\left\{x_k \cdot \nu_i \cos \theta_i+y_k \sin \theta_i-c_f t_k+\tau_i\right\} \rightarrow+\infty,
\end{equation*}
then one has $\underline{V}\left(t_k, x_k, y_k\right) \rightarrow 0$. Since $\alpha>0$, $\min _{1 \leq i \leq n}\left\{\alpha x_k \cdot \nu_i \cos \theta_i+ \alpha y_k \sin \theta_i- \alpha c_f t_k+ \alpha \tau_i\right\} \rightarrow+\infty$. Therefore $\left(\alpha t_k, \alpha x_k, \alpha y_k\right) \in \mathcal{Q}$ and
\begin{equation*}
d\left(\left(\alpha t_k, \alpha x_k, \alpha y_k\right), \partial \mathcal{Q}\right) \rightarrow+\infty \text { as } k \rightarrow+\infty .
\end{equation*}
Since the surface $\{y=\varphi(t, x, \alpha)\}$ is bounded away from $\partial \mathcal{Q}$, one has $\alpha y_k-\varphi\left(\alpha t_k, \alpha x_k, \alpha\right) \rightarrow+\infty$ and thereby $\xi\left(t_k, x_k, y_k\right) \rightarrow+\infty$. Then we obtain that $U\left(\xi\left(t_k, x_k, y_k\right)\right) \rightarrow 0$ and
\begin{equation*}
\begin{aligned}
	 \left|\bar{V}\left(t_k, x_k, y_k\right)-\underline{V}\left(t_k, x_k, y_k\right)\right|\leq \varepsilon\left(h\left(\alpha t_k, \alpha x_k, \alpha\right)+\frac{1}{2}\right) \leq\left(1+\frac{1}{2}\right) \varepsilon \quad\text{for } k \text { large enough}.
\end{aligned}
\end{equation*}
This contradicts \eqref{3.41}.

If the sequence $\left(t_k, x_k, y_k\right)$ satisfies
\begin{equation*}
\min _{1 \leq i \leq n}\left\{x_k \cdot \nu_i \cos \theta_i+y_k \sin \theta_i-c_f t_k+\tau_i\right\} \rightarrow-\infty,
\end{equation*}
we can also get in the same way that $U\left(\xi\left(t_k, x_k, y_k\right)\right) \rightarrow 1$ and $\underline{V}\left(t_k, x_k, y_k\right) \rightarrow 1$ as $k \rightarrow+\infty$. Thus one has 
\begin{equation*}
	\begin{aligned}
		\left|\bar{V}\left(t_k, x_k, y_k\right)-\underline{V}\left(t_k, x_k, y_k\right)\right|\leq \varepsilon\left(h\left(\alpha t_k, \alpha x_k, \alpha\right)+\frac{1}{2}\right) \leq\left(1+\frac{1}{2}\right) \varepsilon \quad\text{for } k \text { large enough},
	\end{aligned}
\end{equation*}
which also led to the same contradiction.

\textbf{Case 2:} For $(t, x, y)\in\mathbb{R}\times\mathbb{R}^{N}$ such that $d((t, x, y), \partial \mathcal{P}) < +\infty$ and $d((t, x, y), \mathcal{R}) \geq \rho$ as $\rho\rightarrow+\infty$, there exists some $i \in\{1, \ldots, n\}$ such that
\begin{equation*}
\begin{aligned}
	\left|x \cdot \nu_i \cos \theta_i+y \sin \theta_i-c_f t+\tau_i\right| \text { is bounded and } x \cdot \nu_j \cos \theta_j +y \sin \theta_j-c_f t+\tau_j \rightarrow+\infty \text { for } j \neq i.
\end{aligned}
\end{equation*}
Thus $\underline{V}(t, x, y)=U\left(x \cdot \nu_i \cos \theta_i+y \sin \theta_i-c_f t+\tau_i\right)$. Since $\alpha>0$, it holds that
\begin{equation*}
	q_i(\alpha t, \alpha x, \alpha y, \alpha) \text { is bounded and } q_j(\alpha t, \alpha x, \alpha y, \alpha) \rightarrow+\infty \text { for } j \neq i .
\end{equation*}
This implies $d((\alpha t, \alpha x, \alpha y), \partial \mathcal{Q}) < +\infty$ and $d((\alpha t, \alpha x, \alpha y), G) \rightarrow+\infty$, then we obtain that $(\alpha t, \alpha x) \in \widehat{Q}_i$ and $d((\alpha t, \alpha x), \widehat{G}) \rightarrow+\infty$. It follows from Lemmas \ref{Lemma 3.2} and \ref{Lemma 3.3} that
\begin{equation*}
\left|\varphi(\alpha t, \alpha x, \alpha)-\left(-\alpha x \cdot \nu_i \cot \theta_i+\frac{c_f}{\sin \theta_i} \alpha t+\frac{\alpha \tau_i}{\sin \theta_i}\right)\right| \rightarrow 0
\end{equation*}
and
\begin{equation*}
\left|\nabla \varphi(\alpha t, \alpha x, \alpha)+\nu_i \cot \theta_i\right| \rightarrow 0 .
\end{equation*}
Thus
\begin{equation*}
\xi(t, x, y) \rightarrow x \cdot \nu_i \cos \theta_i+y \sin \theta_i-c_f t+\tau_i
\end{equation*}
and 
\begin{equation*}
\left|\bar{V}\left(t, x, y\right)-\underline{V}\left(t, x, y\right)\right|\leq \left|U(\xi)-U\left(x \cdot \nu_i \cos \theta_i+y \sin \theta_i-c_f t+\tau_i\right)\right|+\varepsilon h\left(\alpha t, \alpha x, \alpha\right) \leq \frac{3}{2}\varepsilon.
\end{equation*}

{\it Step 3: proof of \eqref{3.15}.} 
Fixing an $i \in\{1, \ldots, n\}$, let
\begin{equation*}
\xi_i(t, x, y)=x \cdot \nu_i \cos \theta_i+y \sin \theta_i-c_f t+\tau_i .
\end{equation*}
If $\eta \leq \xi_i$, since $U$ is a decreasing function, then we know
\begin{equation*}
\bar{V}(t, x, y)=U(\xi)+\varepsilon h(\alpha t, \alpha x, \alpha) \geq U\left(\xi_i\right)=U\left(x \cdot \nu_i \cos \theta_i+y \sin \theta_i-c_f t+\tau_i\right) .
\end{equation*}

Because $\underline{V}=U\left(\xi_i\right)$ for $(t, x)\in\widehat{P}_i$, if $\xi(t, x, y)>\xi_i(t, x, y)$, then
\begin{align*}
	0 & <\xi(t, x, y)-\xi_i(t, x, y) \\
	& =\frac{y-\frac{1}{\alpha} \varphi(\alpha t, \alpha x, \alpha)}{\sqrt{1+|\nabla \varphi (\alpha t, \alpha x, \alpha)|^2}}-\left(x \cdot \nu_i \cos \theta_i+y \sin \theta_i-c_f t+\tau_i\right) \\
	& =\left(\frac{1}{\sin \theta_i \sqrt{1+|\nabla \varphi(\alpha t, \alpha x, \alpha)|^2}}-1\right) \xi_i(t, x, y) \\
	&\quad-\frac{\alpha x \cdot \nu_i \cot \theta_i-\alpha c_f t / \sin \theta_i+\alpha \tau_i / \sin \theta_i+\varphi(\alpha t, \alpha x, \alpha)}{\alpha \sqrt{1+|\nabla \varphi(\alpha t, \alpha x, \alpha)|^2}}
\end{align*}
According to Lemma \ref{Lemma 3.1}, we know that $\Sigma \subset \mathcal{Q}$, and thereby $\varphi(\alpha t, \alpha x, \alpha) \geq \psi_i(\alpha t, \alpha x, \alpha)$ for any fixed $i \in\{1, \ldots, n\}$ and $(\alpha t, \alpha x) \in \widehat{Q}_i$. Then it holds that
\begin{align*}
	\xi_i(t, x, y)&<\xi(t, x, y) \\
	&=\frac{y-\frac{1}{\alpha} \varphi(\alpha t, \alpha x, \alpha)}{\sqrt{1+|\nabla \varphi(\alpha t, \alpha x, \alpha)|^2}} \leq \frac{x \cdot \nu_i \cos \theta_i+y \sin \theta_i-c_f t+\tau_i}{\sin \theta_i \sqrt{1+|\nabla \varphi(\alpha t, \alpha x, \alpha)|^2}} =\frac{\xi_i(t, x, y)}{\sin \theta_i \sqrt{1+|\nabla \varphi(\alpha t, \alpha x, \alpha)|^2}}.
\end{align*}
It follows from \cite[(3.7) and (3.8)]{H. Guo2024}, Lemma \ref{Lemma 3.2} and Lemma \ref{Lemma 3.3} that there exist constants $M>0$ and $A_{1}>0$ such that
\begin{equation*}
	0\leq \left|\sqrt{1+|\nabla \varphi (\alpha t, \alpha x, \alpha)|^2}-\frac{1}{\sin \theta_i}\right| \leq M h(\alpha t, \alpha x, \alpha)
\end{equation*}
and
\begin{equation*}
\left|\xi_i(t, x, y)\right| \geq \frac{\varphi(\alpha t, \alpha x, \alpha)-\psi_{i}(\alpha t, \alpha x, \alpha)}{\alpha \sqrt{1+|\nabla \varphi (\alpha t, \alpha x, \alpha)|^2}}\left|\frac{1}{\sin \theta_i \sqrt{1+|\nabla \varphi (\alpha t, \alpha x, \alpha)|^2}}-1\right|^{-1} \geq \frac{A_1}{\alpha}>0.
\end{equation*}
Therefore
\begin{equation*}
\begin{aligned}
	&\bar{V}(t, x, y)-U\left(\xi_i\right) \\
	& =U(\bar{\xi})-U\left(\xi_i\right)+\varepsilon h(\alpha t, \alpha x, \alpha)\left[U^\beta(\eta) \omega(\eta)+(1-\omega(\eta))\right] \\
	& =U^{\prime}\left(\xi^{\prime}\right)\left(\xi-\xi_i\right)+\varepsilon h(\alpha t, \alpha x, \alpha)\left[U^\beta(\eta) \omega(\eta)+(1-\omega(\eta))\right] \\
	& \geq U^{\prime}\left(\xi^{\prime}\right) \xi_i\left(\frac{1}{\sin \theta_i \sqrt{1+|\nabla \varphi(\alpha t, \alpha x, \alpha)|^2}}-1\right)+\varepsilon h(\alpha t, \alpha x, \alpha)\left[U^\beta(\eta) \omega(\eta)+(1-\omega(\eta))\right]  \\
	& \geq-\left|U^{\prime}\left(\xi^{\prime}\right)\xi_{i}^{2}\right| \frac{\alpha}{A_1} M   h(\alpha t, \alpha x, \alpha)+\varepsilon h(\alpha t, \alpha x, \alpha),
\end{aligned}
\end{equation*}
where $\xi^{\prime}$ is between $\xi_i$ and $\xi$. Since $\left|U^{\prime}\left(\xi^{\prime}\right) \xi_i^2\right|$ is bounded, there is a constant $\alpha_{5}^{+}(\varepsilon)>0$ small enough such that for any $0<\alpha<\alpha_{5}^{+}(\varepsilon)$, $\bar{V}(t, x, y) - U(\xi_i) \geq 0$. Since  $\cup_{i=1}^n \widehat{P}_i=\mathbb{R}\times\mathbb{R}^{N-1}$, we can get that 
\begin{equation*}
\bar{V}(t, x, y) \geq \max _{i \in\{1, \ldots, n\}}\left\{U\left(x \cdot \nu_i \cos \theta_i+y \sin \theta_i-c_f t+\tau_i\right)\right\}=\underline{V}(t, x, y), \quad \forall(t, x, y) \in \mathbb{R} \times \mathbb{R}^N.
\end{equation*}
Let $0<\varepsilon_0^{+} \leq \min \left\{\varepsilon_1^{+}, \varepsilon_2^{+}, \varepsilon_3^{+}(\beta)\right\}$, $\beta^*=\min \left\{1, \beta_1^*, \beta_2^*\right\}$ and
\begin{equation*}
	0<\alpha_0^{+} \leq \min \left\{\alpha_1^{+}(\beta), \alpha_2^{+}(\beta, \varepsilon), \alpha_3^{+}(\varepsilon), \alpha_4^{+}(\varepsilon), \alpha_5^{+}(\varepsilon), \frac{\kappa c_{f}}{2 M C_{2}}\right\},
\end{equation*}
where $\kappa=\kappa(X^{\prime}, X^{\prime \prime}):=\min_{\xi\in[-X^{\prime\prime}, X^{\prime}]}\left|U^{\prime}(\xi)\right| > 0$. This completes the proof of the lemma. 
\end{proof}

We point out that there exist positive constants $v_{\star}$ and $C_{\star}$ such that
\begin{equation}\label{3.45}
\frac{\left|\underline{V}(t, x, y)\right|+\left|\bar{V}(t, x, y)\right|}{\min \left\{1, e^{-2 v_{\star} \min_{1\leq i\leq n} \left\{x \cdot \nu_i \cot \theta_i+y-\frac{c_f}{\sin \theta_i} t+\frac{\tau_i}{\sin \theta_i}   \right\}}\right\}} \leq C_{\star},\quad \forall(t, x, y)\in\mathbb{R} \times \mathbb{R}^N.
\end{equation}
Actually, we only need to consider $\min_{1\leq i\leq n} \left\{x \cdot \nu_i \cot \theta_i+y-\frac{c_f}{\sin \theta_i} t+\frac{\tau_i}{\sin \theta_i}   \right\}>0$. 

\textbf{Case 1:} $\min_{1\leq i\leq n} \left\{x \cdot \nu_i \cot \theta_i+y-\frac{c_f}{\sin \theta_i} 
t+\frac{\tau_i}{\sin \theta_i} \right\}>0$ and $\xi(t, x, y)>1$.
By \eqref{1.5}, one has
\begin{equation*}
	\left|\underline{V}(t, x, y)\right|\leq L_{2} e^{-c_f \min_{1 \leq i \leq n}\{x \cdot \nu_i \cos \theta_i+y \sin \theta_i-c_f t+ \tau_i\} }\leq L_{2} e^{-2v\min_{1\leq i\leq n} \left\{x \cdot \nu_i \cot \theta_i+y-\frac{c_f}{\sin \theta_i} t+\frac{\tau_i}{\sin \theta_i}   \right\}} 
\end{equation*}
for all $0< v \leq c_f \min_{1 \leq i \leq n} \{\sin\theta_i\} / 2$. Denote $\widetilde{C}:=\widetilde{C}(\alpha)=\frac{\sup _{(\alpha t, \alpha x) \in\mathbb{R}\times\mathbb{R}^{N-1}}|\varphi(\alpha t, \alpha x, \alpha)-\psi(\alpha t, \alpha x, \alpha)|}{\alpha} < +\infty$. It follows from \eqref{3.3} and the definition of $h(t, x, \alpha)$ that $\alpha \widetilde{C} \leq C \sup _{(\alpha t, \alpha x) \in\mathbb{R}\times\mathbb{R}^{N-1}} h(\alpha t, \alpha x, \alpha) \leq C$. Then we can obtain that
	\begin{align*}
		& \left|U(\xi)\right|+\left|U^\beta(\eta)\right|\\
		&\leq L_{2} e^{-c_f \xi}+\left(L_{2} e^{-c_f \eta}\right)^\beta\\
		&\leq L_{2}e^{-\alpha c_f \frac{y-\frac{1}{\alpha} \varphi(\alpha t, \alpha x, \alpha)}{\sqrt{1+|\nabla \varphi(\alpha t, \alpha x, \alpha)|^2}}}+\left(L_{2} e^{-\alpha c_f (y-\frac{1}{\alpha} \varphi(\alpha t, \alpha x, \alpha))}\right)^\beta\\
		&\leq L_{2} e^{-\alpha c_f (y-\frac{1}{\alpha} \varphi(\alpha t, \alpha x, \alpha))\frac{1}{\sqrt{1+|\nabla \varphi(\alpha t, \alpha x, \alpha)|^2}}}+L_{2}^\beta e^{-\alpha \beta c_f (y-\frac{1}{\alpha} \varphi(\alpha t, \alpha x, \alpha))}\\
		&\leq L_{2} e^{-\alpha c_f(y-(\psi(\alpha t, \alpha x, \alpha)/ \alpha +\sup_{\mathbb{R}^{N-1}}|\varphi(\alpha t, \alpha x, \alpha)-\psi(\alpha t, \alpha x, \alpha)| / \alpha))\frac{1}{\sqrt{1+|\nabla \varphi(\alpha t, \alpha x, \alpha)|^2}}}+L_{2}^{\beta}e^{-\alpha \beta c_f(y-\frac{1}{\alpha} \varphi(\alpha t, \alpha x, \alpha))}\\
		&\leq L_{2} e^{-\alpha c_f (y-(\max_{1\leq i \leq n}\{-x \cdot \nu_i \cot \theta_i+\frac{c_f}{\sin \theta_i} t-\frac{\tau_i}{\sin \theta_i} \}+\widetilde{C}))\frac{1}{\sqrt{1+|\nabla \varphi(\alpha t, \alpha x, \alpha)|^2}}}+L_{2}^{\beta}e^{-\alpha \beta c_f(y-\varphi(\alpha t, \alpha x, \alpha) / \alpha)}\\
		&\leq L_{2} e^{-\alpha c_f (y+(\min_{1\leq i \leq n}\{x \cdot \nu_i \cot \theta_i-\frac{c_f}{\sin \theta_i} t+\frac{\tau_i}{\sin \theta_i} \}-\widetilde{C}))\frac{1}{\sqrt{1+|\nabla \varphi(\alpha t, \alpha x, \alpha)|^2}}}+L_{2}^{\beta}e^{-\alpha \beta c_f(y-\varphi(\alpha t, \alpha x, \alpha) / \alpha)}\\
		&\leq L_{2} e^{-\alpha c_f (\min_{1\leq i \leq n}\{x \cdot \nu_i \cot \theta_i +y-\frac{c_f}{\sin \theta_i} t+\frac{\tau_i}{\sin \theta_i} \}-\widetilde{C})\frac{1}{\sqrt{1+|\nabla \varphi(\alpha t, \alpha x, \alpha)|^2}}}\\
		&\quad+L_{2}^{\beta} e^{-\alpha \beta c_f(\min_{1\leq i \leq n}\{x \cdot \nu_i \cot \theta_i +y-\frac{c_f}{\sin \theta_i} t+\frac{\tau_i}{\sin \theta_i} \}-\widetilde{C})}\\
		&\leq L_{2} e^{c_f C}\cdot e^{-\alpha c_f \min_{1\leq i \leq n}\{x \cdot \nu_i \cot \theta_i +y-\frac{c_f}{\sin \theta_i} t+\frac{\tau_i}{\sin \theta_i} \}\frac{1}{\sqrt{1+|\nabla \varphi(\alpha t, \alpha x, \alpha)|^2}}}\\
		&\quad+L_{2} e^{\beta c_f C}\cdot e^{-\alpha\beta c_f \min_{1\leq i \leq n}\{x \cdot \nu_i \cot \theta_i +y-\frac{c_f}{\sin \theta_i} t+\frac{\tau_i}{\sin \theta_i} \}}\\
		&\leq 2K_{1}^{*} e^{-2v \min_{1\leq i \leq n}\{x \cdot \nu_i \cot \theta_i +y-\frac{c_f}{\sin \theta_i} t+\frac{\tau_i}{\sin \theta_i}\}},
	\end{align*}
where $K_{1}^{*} =\max\{L_{2}, L_{2} e^{c_f C}, L_{2} e^{\beta c_f C}\}$ and $0< v \leq\frac{\alpha c_f}{2}\min\{\beta, \frac{1}{\sqrt{1+(C_1+\max_{1\le i\le n} \{|\nu_i \cot \theta_i|\})^2}}\}$.

\textbf{Case 2:} $\min_{1\leq i\leq n} \left\{x \cdot \nu_i \cot \theta_i+y-\frac{c_f}{\sin \theta_i} 
t+\frac{\tau_i}{\sin \theta_i} \right\}>0$ and $0 \leq \xi(t, x, y)\leq1$. By \eqref{1.5}, one gets  
\begin{equation*}
	\left|\underline{V}(t, x, y)\right|\leq L_{2} e^{-c_f \min_{1 \leq i \leq n}\{x \cdot \nu_i \cos \theta_i+y \sin \theta_i-c_f t+ \tau_i\} }\leq L_{2} e^{-2v\min_{1\leq i\leq n} \left\{x \cdot \nu_i \cot \theta_i+y-\frac{c_f}{\sin \theta_i} t+\frac{\tau_i}{\sin \theta_i}   \right\}} 
\end{equation*}
for all $0< v \leq c_f \min_{1 \leq i \leq n} \{\sin\theta_i\} / 2$. Moreover, we have
\begin{equation*}
	1\geq\xi(t, x, y)=\frac{y-\varphi(\alpha t, \alpha x, \alpha) / \alpha}{\sqrt{1+|\nabla \varphi(\alpha t, \alpha x, \alpha)|^2}}\geq \frac{\min_{1\leq i \leq n}\{x \cdot \nu_i \cot \theta_i +y-\frac{c_f}{\sin \theta_i} t+\frac{\tau_i}{\sin \theta_i} \}-\widetilde{C}}{\sqrt{1+|\nabla \varphi(\alpha t, \alpha x, \alpha)|^2}}
\end{equation*}
and thereby
\begin{equation*}
	0<\min_{1\leq i \leq n}\{x \cdot \nu_i \cot \theta_i +y-\frac{c_f}{\sin \theta_i} t+\frac{\tau_i}{\sin \theta_i} \}\leq \sqrt{1+(C_1+\max_{1\le i\le n} \{|\nu_i \cot \theta_i|\})^2}+\widetilde{C}.
\end{equation*}
Denote $\widehat{C}:=\sqrt{1+(C_3+\max_{1\le i\le n} \{|\nu_i \cot \theta_i|\})^2}+C$. 
Since $0 \leq \xi(t, x, y)\leq 1$ and \eqref{1.5}, we obtain that
	\begin{align*}
		\bar{V}(t,x,y)&\leq \left|U(\xi)\right|+1\\
		&\leq L_{2} e^{c_f C}\cdot e^{-\alpha c_f \min_{1\leq i \leq n}\{x \cdot \nu_i \cot \theta_i +y-\frac{c_f}{\sin \theta_i} t+\frac{\tau_i}{\sin \theta_i} \}\frac{1}{\sqrt{1+|\nabla \varphi(\alpha t, \alpha x, \alpha)|^2}}}\\
		&\quad+e^{\alpha\beta c_f \min_{1\leq i \leq n}\{x \cdot \nu_i \cot \theta_i +y-\frac{c_f}{\sin \theta_i} t+\frac{\tau_i}{\sin \theta_i} \}} \cdot e^{-\alpha\beta c_f \min_{1\leq i \leq n}\{x \cdot \nu_i \cot \theta_i +y-\frac{c_f}{\sin \theta_i} t+\frac{\tau_i}{\sin \theta_i} \}}\\
		&\leq 2K_{2}^{*} e^{-2v \min_{1\leq i \leq n}\{x \cdot \nu_i \cot \theta_i +y-\frac{c_f}{\sin \theta_i} t+\frac{\tau_i}{\sin \theta_i}\}},
	\end{align*}
where $K_{2}^{*}:=\max\{ L_{2},  L_{2} e^{c_f C}, e^{c_f \widehat{C}}\}$ and $0< v \leq\frac{\alpha c_f}{2}\min\{\beta, \frac{1}{\sqrt{1+(C_1+\max_{1\le i\le n} \{|\nu_i \cot \theta_i|\})^2}}\}$.

\textbf{Case 3:} $\min_{1\leq i\leq n} \left\{x \cdot \nu_i \cot \theta_i+y-\frac{c_f}{\sin \theta_i} 
t+\frac{\tau_i}{\sin \theta_i} \right\}>0$ and $\xi(t, x, y)< 0$. By \eqref{1.5}, we know
\begin{equation*}
	\left|\underline{V}(t, x, y)\right|\leq L_{2} e^{-c_f \min_{1 \leq i \leq n}\{x \cdot \nu_i \cos \theta_i+y \sin \theta_i-c_f t+ \tau_i\} }\leq L_{2} e^{-2v\min_{1\leq i\leq n} \left\{x \cdot \nu_i \cot \theta_i+y-\frac{c_f}{\sin \theta_i} t+\frac{\tau_i}{\sin \theta_i}   \right\}} 
\end{equation*}
for all $0< v \leq c_f \min_{1 \leq i \leq n} \{\sin\theta_i\} / 2$. Furthermore, one gets
\begin{equation*}
	0>\xi(t, x, y)=\frac{y-\varphi(\alpha t, \alpha x, \alpha) / \alpha}{\sqrt{1+|\nabla \varphi(\alpha t, \alpha x, \alpha)|^2}}\geq \frac{\min_{1\leq i \leq n}\{x \cdot \nu_i \cot \theta_i +y-\frac{c_f}{\sin \theta_i} t+\frac{\tau_i}{\sin \theta_i} \}-\widetilde{C}}{\sqrt{1+|\nabla \varphi(\alpha t, \alpha x, \alpha)|^2}},
\end{equation*}
\begin{equation*}
	0<\min_{1\leq i \leq n}\{x \cdot \nu_i \cot \theta_i +y-\frac{c_f}{\sin \theta_i} t+\frac{\tau_i}{\sin \theta_i} \}<\widetilde{C},
\end{equation*}
and 
\begin{equation*}
\frac{y-\varphi(\alpha t, \alpha x, \alpha) / \alpha}{\sqrt{1+|\nabla \varphi(\alpha t, \alpha x, \alpha)|^2}}\leq \frac{y-\psi(\alpha t, \alpha x, \alpha)/ \alpha}{\sqrt{1+|\nabla \varphi(\alpha t, \alpha x, \alpha)|^2}}= \frac{\min_{1\leq i \leq n}\{x \cdot \nu_i \cot \theta_i +y-\frac{c_f}{\sin \theta_i} t+\frac{\tau_i}{\sin \theta_i} \}}{\sqrt{1+|\nabla \varphi(\alpha t, \alpha x, \alpha)|^2}}.
\end{equation*}
Since $\xi(t, x, y)< 0$ and \eqref{1.5}, we have that
	\begin{align*}
		\bar{V}(t, x, y)&\leq \left|U(\xi)\right|+1\\
		&\leq 1+L_3 e^{\alpha \beta_{0} \frac{y-\varphi(\alpha t, \alpha x, \alpha) / \alpha}{\sqrt{1+|\nabla \varphi(\alpha t, \alpha x, \alpha)|^2}}} +1\\
		&\leq L_3 e^{\alpha \beta_{0} \frac{y-\varphi(\alpha t, \alpha x, \alpha) / \alpha}{\sqrt{1+(C_1+\max_{1\le i\le n} \{|\nu_i \cot \theta_i|\})^2}}}+2\\
		&\leq L_3 e^{\alpha \beta_{0} \frac{\min_{1\leq i \leq n}\{x \cdot \nu_i \cot \theta_i +y-\frac{c_f}{\sin \theta_i} t+\frac{\tau_i}{\sin \theta_i} \}}{\sqrt{1+(C_1+\max_{1\le i\le n} \{|\nu_i \cot \theta_i|\})^2}}}+2\\
		&\leq L_3 e^{\alpha \beta_{0} \min_{1\leq i \leq n}\{x \cdot \nu_i \cot \theta_i +y-\frac{c_f}{\sin \theta_i} t+\frac{\tau_i}{\sin \theta_i} \} }+2\\
		&\leq L_3 e^{ 2 \alpha \beta_{0} \min_{1\leq i \leq n}\{x \cdot \nu_i \cot \theta_i +y-\frac{c_f}{\sin \theta_i} t+\frac{\tau_i}{\sin \theta_i} \} } \cdot e^{-\alpha \beta_{0} \min_{1\leq i \leq n}\{x \cdot \nu_i \cot \theta_i +y-\frac{c_f}{\sin \theta_i} t+\frac{\tau_i}{\sin \theta_i} \} }\\
		&\quad+2 e^{\alpha \beta_{0} \min_{1\leq i \leq n}\{x \cdot \nu_i \cot \theta_i +y-\frac{c_f}{\sin \theta_i} t+\frac{\tau_i}{\sin \theta_i} \} } \cdot e^{-\alpha \beta_{0} \min_{1\leq i \leq n}\{x \cdot \nu_i \cot \theta_i +y-\frac{c_f}{\sin \theta_i} t+\frac{\tau_i}{\sin \theta_i} \} }\\
		&\leq L_3 e^{2 \beta_{0} C } \cdot e^{-\alpha \beta_{0} \min_{1\leq i \leq n}\{x \cdot \nu_i \cot \theta_i +y-\frac{c_f}{\sin \theta_i} t+\frac{\tau_i}{\sin \theta_i} \} }\\
		&\quad+2 e^{ \beta_{0} C } \cdot e^{-\alpha \beta_{0} \min_{1\leq i \leq n}\{x \cdot \nu_i \cot \theta_i +y-\frac{c_f}{\sin \theta_i} t+\frac{\tau_i}{\sin \theta_i} \} }\\
		&\leq 3 K_{3}^{*} e^{-2 v \min_{1\leq i \leq n}\{x \cdot \nu_i \cot \theta_i +y-\frac{c_f}{\sin \theta_i} t+\frac{\tau_i}{\sin \theta_i} \} },
	\end{align*}
where $K_{3}^{*}:=\max\{L_3, L_3 e^{2 \beta_{0} C }, e^{ \alpha \beta_{0} C}\}$ and $0<v\leq\frac{ \alpha \beta_{0}}{2}$. In conclusion, let $C_{\star}>\max\{3K_{1}^{*},  3K_{2}^{*}, 4K_{3}^{*}\}$ and $0<v_{\star}<\min\{\frac{\alpha \beta_{0} }{2}, \frac{ c_f}{2}\min_{1\leq i \leq n}\{\sin\theta_i,\alpha\beta, \alpha, \frac{\alpha}{\sqrt{1+(C_3+\max_{1\le i\le n} \{|\nu_i \cot \theta_i|\})^2}}\}\}$, then \eqref{3.41} is valid.

Furthermore, we can also obtain that there exist positive constants $v^{\star}$ and $C^{\star}$ such that
\begin{equation}\label{3.46}
	\frac{\left|\bar{V}(t, x, y)-\underline{V}(t, x, y)\right|}{\min \left\{1, e^{-2 v^{\star} \min_{1\leq i\leq n} \left\{x \cdot \nu_i \cot \theta_i+y-\frac{c_f}{\sin \theta_i} t+\frac{\tau_i}{\sin \theta_i} \right\}}\right\}} \leq C^{\star} \varepsilon  \quad\text{uniformly as } d\left((t, x, y), \mathcal{R}\right) \rightarrow+\infty.
\end{equation}
Since \eqref{3.14} gives the result for the case $\min_{1\leq i\leq n} \left\{x \cdot \nu_i \cot 
\theta_i+y-\frac{c_f}{\sin \theta_i} t+\frac{\tau_i}{\sin \theta_i} \right\}\leq 0$, we just have to 
consider $\min_{1\leq i\leq n} \left\{x \cdot \nu_i \cot \theta_i+y-\frac{c_f}{\sin \theta_i} 
t+\frac{\tau_i}{\sin \theta_i} \right\}>0$. 

\textbf{Subcase 1:} For $(t, x, y) \in \mathbb{R} \times \mathbb{R}^N$ such that $d\left(( t, x, y), 
\partial \mathcal{P}\right) \geq \rho$ as $\rho \rightarrow+\infty$, it holds that 
\begin{equation*}
\min _{1 \leq i \leq n}\left\{x \cdot \nu_i \cos \theta_i+y \sin \theta_i-c_f t+\tau_i\right\} \rightarrow+\infty \text{ and } \min_{1\leq i\leq n} \left\{x \cdot \nu_i \cot \theta_i+y-\frac{c_f}{\sin \theta_i} t+\frac{\tau_i}{\sin \theta_i} \right\} \rightarrow+\infty.
\end{equation*}
One can obtain that $\min_{1 \leq i \leq n}\left\{\alpha x \cdot \nu_i \cos \theta_i+\alpha y \sin \theta_i- \alpha c_f t+\alpha \tau_i \right\} \rightarrow+\infty$ for $(\alpha t, \alpha x, \alpha y) \in \mathbb{R} \times \mathbb{R}^N$. Similar to Step $2$ in Lemma \ref{Lemma 3.4}, one gets that $\alpha y-\varphi(\alpha t, \alpha x, \alpha y)  \rightarrow+\infty$, $\xi(t, x, y)\rightarrow+\infty$ and $\eta(t, x, y)\rightarrow+\infty$. Let $\mu(t, x, y):=\min_{1 \leq i \leq n}\left\{x \cdot \nu_i \cos \theta_i+y \sin \theta_i- c_f t+\tau_i \right\}$ and then $\underline{V}(t, x, y)=U(\mu)$. We have
\begin{equation*}
	0<\min_{1\leq i\leq n} \left\{x \cdot \nu_i \cot \theta_i+y-\frac{c_f}{\sin \theta_i} t+\frac{\tau_i}{\sin \theta_i} \right\}\cdot \min_{1\leq i\leq n}\left\{\sin \theta_{i} \right\} \leq \mu(t, x, y)
\end{equation*}
and
\begin{equation*}
\frac{y-\psi(\alpha t, \alpha x, \alpha y)/\alpha-\widetilde{C}}{\sqrt{1+(C_1+\max_{1\le i\le n} \{|\nu_i \cot \theta_i|\})^2}}\leq \frac{y-\varphi(\alpha t, \alpha x, \alpha y) / \alpha}{\sqrt{1+(C_1+\max_{1\le i\le n} \{|\nu_i \cot \theta_i|\})^2}} \leq \xi(t, x, y) \leq \eta(t, x, y) .
\end{equation*}
Then it holds that
	\begin{align*}
	&\frac{\left|\bar{V}(t, x, y)-\underline{V}(t, x, y)\right|}{ e^{-2 v^{\star} \min_{1\leq i\leq n} \left\{x \cdot \nu_i \cot \theta_i+y-\frac{c_f}{\sin \theta_i} t+\frac{\tau_i}{\sin \theta_i} \right\}}} \\
	&=\frac{\left|U(\xi)+\varepsilon h(\alpha t, \alpha x, \alpha y) \times U^\beta(\eta)-U(\mu)\right|}{ e^{-2 v^{\star} \min_{1\leq i\leq n} \left\{x \cdot \nu_i \cot \theta_i+y-\frac{c_f}{\sin \theta_i} t+\frac{\tau_i}{\sin \theta_i} \right\}}}\\
	&\leq \frac{\left|U(\xi)\right|+\left|U(\mu)\right|+\left|\varepsilon h(\alpha t, \alpha x, \alpha y) \times U^\beta(\eta)\right|}{ e^{-2 v^{\star} \min_{1\leq i\leq n} \left\{x \cdot \nu_i \cot \theta_i+y-\frac{c_f}{\sin \theta_i} t+\frac{\tau_i}{\sin \theta_i} \right\}}}\\
	&\leq \frac{L_{2} e^{- \alpha c_{f} \xi}}{ e^{-2 v^{\star} \min_{1\leq i\leq n} \left\{x \cdot \nu_i \cot \theta_i+y-\frac{c_f}{\sin \theta_i} t+\frac{\tau_i}{\sin \theta_i} \right\}}}+\frac{L_{2} e^{- \alpha c_{f} \mu}}{ e^{-2 v^{\star} \min_{1\leq i\leq n} \left\{x \cdot \nu_i \cot \theta_i+y-\frac{c_f}{\sin \theta_i} t+\frac{\tau_i}{\sin \theta_i} \right\}}}\\
	&\quad+\frac{\left|\varepsilon L_{2}^\beta e^{- \alpha \beta c_{f}  \eta}\right|}{ e^{-2 v^{\star} \min_{1\leq i\leq n} \left\{x \cdot \nu_i \cot \theta_i+y-\frac{c_f}{\sin \theta_i} t+\frac{\tau_i}{\sin \theta_i} \right\}}}\\
	&=:O_{1}+O_{2}+O_{3},
	\end{align*}
where $\widetilde{C}=\widetilde{C}(\alpha)=\frac{\sup_{(\alpha t, \alpha x) \in\mathbb{R}\times\mathbb{R}^{N-1}}|\varphi(\alpha t, \alpha x, \alpha)-\psi(\alpha t, \alpha x, \alpha)|}{\alpha}$ and $\alpha \widetilde{C} \leq C$. Moreover, we have
	\begin{align*}
		O_{1}&=\frac{L_{2} e^{- \alpha c_{f} \xi}}{ e^{-2 v^{\star} \min_{1\leq i\leq n} \left\{x \cdot \nu_i \cot \theta_i+y-\frac{c_f}{\sin \theta_i} t+\frac{\tau_i}{\sin \theta_i} \right\}}}\\
		&\leq \frac{L_{2} e^{- \alpha c_{f} \left(\frac{y-\psi(\alpha t, \alpha x, \alpha y)/\alpha-\widetilde{C}}{\sqrt{1+(C_1+\max_{1\le i\le n} \{|\nu_i \cot \theta_i|\})^2}}\right)}}{ e^{-2 v^{\star} \min_{1\leq i\leq n} \left\{x \cdot \nu_i \cot \theta_i+y-\frac{c_f}{\sin \theta_i} t+\frac{\tau_i}{\sin \theta_i} \right\}}}\\
		& \leq \frac{L_{2} e^{- \alpha c_{f} \left(\frac{-\widetilde{C}}{\sqrt{1+(C_1+\max_{1\le i\le n} \{|\nu_i \cot \theta_i|\})^2}}\right)} \cdot e^{- \alpha c_{f} \left(\frac{y-\psi(\alpha t, \alpha x, \alpha y)/\alpha}{\sqrt{1+(C_1+\max_{1\le i\le n} \{|\nu_i \cot \theta_i|\})^2}}\right)}}{ e^{-2 v^{\star} \min_{1\leq i\leq n} \left\{x \cdot \nu_i \cot \theta_i+y-\frac{c_f}{\sin \theta_i} t+\frac{\tau_i}{\sin \theta_i} \right\}}}\\
		& \leq \frac{L_{2} e^{c_{f} \left(\frac{C}{\sqrt{1+(C_1+\max_{1\le i\le n} \{|\nu_i \cot \theta_i|\})^2}}\right)}\cdot e^{- \alpha c_{f}  \left(\frac{\min_{1\leq i\leq n} \left\{x \cdot \nu_i \cot \theta_i+y-\frac{c_f}{\sin \theta_i} t+\frac{\tau_i}{\sin \theta_i} \right\}\cdot \min_{1\leq i\leq n}\left\{\sin \theta_{i} \right\}}{\sqrt{1+(C_1+\max_{1\le i\le n} \{|\nu_i \cot \theta_i|\})^2}}\right)}}{ e^{-2 v^{\star} \min_{1\leq i\leq n} \left\{x \cdot \nu_i \cot \theta_i+y-\frac{c_f}{\sin \theta_i} t+\frac{\tau_i}{\sin \theta_i} \right\}}}.
	\end{align*}
Since $\min_{1\leq i\leq n} \left\{x \cdot \nu_i \cot \theta_i+y-\frac{c_f}{\sin \theta_i} t+\frac{\tau_i}{\sin \theta_i} \right\} \rightarrow+\infty$, one can have $O_{1}\leq L_{2} e^{c_{f}  C}\varepsilon$, by setting $0<v^{\star}<v_{1}^{\star}:=\frac{\alpha c_{f} \min_{1\leq i\leq n}\left\{\sin \theta_{i} \right\}}{2\sqrt{1+(C_1+\max_{1\le i\le n} \{|\nu_i \cot \theta_i|\})^2}}$.
If set $0<v^{\star}<v_{2}^{\star}:=\frac{c_{f}}{2} \min_{1\leq i\leq n}\left\{\sin \theta_{i} \right\}$, then we know
\begin{align*}
		O_{2}\leq L_{2} \frac{e^{-\alpha c_{f}  \min_{1\leq i\leq n} \left\{x \cdot \nu_i \cot \theta_i+y-\frac{c_f}{\sin \theta_i} t+\frac{\tau_i}{\sin \theta_i} \right\}\cdot \min_{1\leq i\leq n}\left\{\sin \theta_{i} \right\}}}{ e^{-2 v^{\star} \min_{1\leq i\leq n} \left\{x \cdot \nu_i \cot \theta_i+y-\frac{c_f}{\sin \theta_i} t+\frac{\tau_i}{\sin \theta_i} \right\}}}\leq L_{2} \varepsilon
\end{align*}
as $\min_{1\leq i\leq n} \left\{x \cdot \nu_i \cot \theta_i+y-\frac{c_f}{\sin \theta_i} t+\frac{\tau_i}{\sin \theta_i} \right\} \rightarrow+\infty$. Setting $0<v^{\star}<v_{3}^{\star}:=\frac{\alpha \beta c_{f}}{2} $, one has
\begin{align*}
	O_{3}&=\frac{\left|\varepsilon L_{2}^\beta e^{- \alpha \beta c_{f}  \eta}\right|}{ e^{-2 v^{\star} \min_{1\leq i\leq n} \left\{x \cdot \nu_i \cot \theta_i+y-\frac{c_f}{\sin \theta_i} t+\frac{\tau_i}{\sin \theta_i} \right\}}}\\
	&\leq\frac{\left|\varepsilon L_{2}^{\beta} e^{- \alpha \beta c_{f} (y-\psi(\alpha t, \alpha x, \alpha)/\alpha-\widetilde{C})}\right|}{ e^{-2 v^{\star} \min_{1\leq i\leq n} \left\{x \cdot \nu_i \cot \theta_i+y-\frac{c_f}{\sin \theta_i} t+\frac{\tau_i}{\sin \theta_i} \right\}}}\\
	&=\frac{\varepsilon L_{2}^{\beta} e^{-\alpha \beta c_{f} \left(\min_{1\leq i\leq n} \left\{x \cdot \nu_i \cot \theta_i+y-\frac{c_f}{\sin \theta_i} t+\frac{\tau_i}{\sin \theta_i} \right\}-\widetilde{C}\right)}}{ e^{-2 v^{\star} \min_{1\leq i\leq n} \left\{x \cdot \nu_i \cot \theta_i+y-\frac{c_f}{\sin \theta_i} t+\frac{\tau_i}{\sin \theta_i} \right\}}}\\
	&\leq \varepsilon L_{2}^{\beta} e^{\beta c_{f} C} \cdot \frac{ e^{-\alpha \beta c_{f} \min_{1\leq i\leq n} \left\{x \cdot \nu_i \cot \theta_i+y-\hat{c} t \right\}}}{ e^{-2 v^{\star} \min_{1\leq i\leq n} \left\{x \cdot \nu_i \cot \theta_i+y-\frac{c_f}{\sin \theta_i} t+\frac{\tau_i}{\sin \theta_i} \right\}}}\leq \varepsilon L_{2}^{\beta} e^{\beta c_{f} C} .
\end{align*}
Hence, if we set $0<v^{\star}<\min\{v_{1}^{\star}, v_{2}^{\star}, v_{3}^{\star}\}$ and $C_{1}^{\star}=L_{2} e^{c_{f}  C}+L_{2}+L_{2}^{\beta} e^{\beta c_{f} C}$, then it holds that
\begin{equation*}
	\frac{\left|\bar{V}(t, x, y)-\underline{V}(t, x, y)\right|}{ e^{-2 v^{\star} \min_{1\leq i\leq n} \left\{x \cdot \nu_i \cot \theta_i+y-\frac{c_f}{\sin \theta_i} t+\frac{\tau_i}{\sin \theta_i} \right\}}}\leq C_{1}^{\star}\varepsilon 
\end{equation*}
for $(t, x, y) \in \mathbb{R} \times \mathbb{R}^N$ such that $d\left((t, x, y), \partial \mathcal{P}\right) \geq \rho$ as $\rho \rightarrow+\infty$.

\textbf{Subcase 2:} For $(t, x, y) \in \mathbb{R} \times \mathbb{R}^N$ such that $d\left((t, x, y), 
\partial \mathcal{P}\right)<+\infty$ and $d\left((t, x, y), \mathcal{R}\right) \geq \rho$ as $\rho 
\rightarrow+\infty$, there is some $i \in\{1, \ldots, n\}$ such that
\begin{equation*}
	\begin{aligned}
		\left|x \cdot \nu_i \cos \theta_i+y \sin \theta_i-c_f t+\tau_i\right| \text { is bounded and } x \cdot \nu_j \cos \theta_j +y \sin \theta_j-c_f t+\tau_j \rightarrow+\infty \text { for } j \neq i.
	\end{aligned}
\end{equation*}
Thus there exists a constant $C_{3}>0$ such that $\left|x \cdot \nu_i \cos \theta_i+y \sin \theta_i-c_f t+\tau_i\right|\leq C_{3}$ and it holds that $\underline{V}(t, x, y)=U\left(x \cdot \nu_i \cos \theta_i+y \sin \theta_i-c_f t+\tau_i\right)$. Since $\alpha>0$, we have
\begin{equation*}
	q_i(\alpha t, \alpha x, \alpha y, \alpha) \text { is bounded and } q_j(\alpha t, \alpha x, \alpha y, \alpha) \rightarrow+\infty \text { for } j \neq i .
\end{equation*}
This implies $d((\alpha t, \alpha x, \alpha y), \partial \mathcal{Q}) < +\infty$ and $d((\alpha t, \alpha x, \alpha y), \mathcal{R}) \rightarrow+\infty$, then it holds that $(\alpha t, \alpha x) \in \widehat{Q}_i$ and $d((\alpha t, \alpha x), \widehat{G}) \rightarrow+\infty$. It follows from Lemmas \ref{Lemma 3.2} and \ref{Lemma 3.3} that
\begin{equation*}
	\left|\varphi(\alpha t, \alpha x, \alpha)-\left(-\alpha x \cdot \nu_i \cot \theta_i+\frac{c_f}{\sin \theta_i} \alpha t+\frac{\alpha \tau_i}{\sin \theta_i}\right)\right| \rightarrow 0
\end{equation*}
and
\begin{equation*}
\left|\nabla \varphi(\alpha t, \alpha x, \alpha)+\nu_i \cot \theta_i\right| \rightarrow 0 .
\end{equation*}
Therefore
\begin{equation}\label{3.47}
	\xi(t, x, y) \rightarrow x \cdot \nu_i \cos \theta_i+y \sin \theta_i-c_f t+\tau_i=\mu(t, x, y).
\end{equation}
Then it follows that
	\begin{align*}
		&\frac{\left|\bar{V}(t, x, y)-\underline{V}(t, x, y)\right|}{ e^{-2 v^{\star} \min_{1\leq i\leq n} \left\{x \cdot \nu_i \cot \theta_i+y-\frac{c_f}{\sin \theta_i} t+\frac{\tau_i}{\sin \theta_i} \right\}}} \\
		&\leq\frac{\left|U(\xi)+\varepsilon h(\alpha t, \alpha x, \alpha)-U(\mu)\right|}{ e^{-2 v^{\star} \min_{1\leq i\leq n} \left\{x \cdot \nu_i \cot \theta_i+y-\frac{c_f}{\sin \theta_i} t+\frac{\tau_i}{\sin \theta_i} \right\}}} \\
		&\leq \frac{\left|U(\xi)-U(\mu)\right|}{ e^{-2 v^{\star} \min_{1\leq i\leq n} \left\{x \cdot \nu_i \cot \theta_i+y-\frac{c_f}{\sin \theta_i} t+\frac{\tau_i}{\sin \theta_i} \right\}}}+\frac{\left|\varepsilon h(\alpha t, \alpha x, \alpha)\right|}{ e^{-2 v^{\star} \min_{1\leq i\leq n} \left\{x \cdot \nu_i \cot \theta_i+y-\frac{c_f}{\sin \theta_i} t+\frac{\tau_i}{\sin \theta_i} \right\}}}=: O_{4}+O_{5}.
	\end{align*}
It follows from \eqref{3.47} and $0<v^{\star}< v_{2}^{\star}$ that 
\begin{align*}
	O_{4}\leq  \varepsilon  e^{2 v^{\star} \min_{1\leq i\leq n} \left\{x \cdot \nu_i \cot \theta_i+y-\frac{c_f}{\sin \theta_i} t+\frac{\tau_i}{\sin \theta_i} \right\}}\leq  \varepsilon  e^{c_{f}  C_{3}}
\end{align*}
and
\begin{align*}
		O_{5}&= \frac{\left|\varepsilon h(\alpha t, \alpha x, \alpha)\right|}{ e^{-2 v^{\star} \min_{1\leq i\leq n} \left\{x \cdot \nu_i \cot \theta_i+y-\frac{c_f}{\sin \theta_i} t+\frac{\tau_i}{\sin \theta_i} \right\}}}\\
		&\leq \varepsilon e^{2 v^{\star} \min_{1\leq i\leq n} \left\{x \cdot \nu_i \cot \theta_i+y-\frac{c_f}{\sin \theta_i} t+\frac{\tau_i}{\sin \theta_i} \right\}}\leq  \varepsilon  e^{c_{f}  C_{3}} .
\end{align*}
Hence, by setting $0<v^{\star}<\min\{v_{1}^{\star}, v_{2}^{\star}, v_{3}^{\star}\}$ and $C_{2}^{\star}= 2 e^{c_{f}  C_{3}}$, one has 
\begin{equation*}
\frac{\left|\bar{V}(t, x, y)-\underline{V}(t, x, y)\right|}{ e^{-2 v^{\star} \min_{1\leq i\leq n} \left\{x \cdot \nu_i \cot \theta_i+y-\frac{c_f}{\sin \theta_i} t+\frac{\tau_i}{\sin \theta_i} \right\}}}\leq C_{2}^{\star}\varepsilon
\end{equation*}
for $(t, x, y) \in \mathbb{R} \times \mathbb{R}^N$ such that $d\left((t, x, y), \partial \mathcal{P}\right)<+\infty$ and $d\left((t, x, y), \mathcal{R}\right) \geq \rho$ as $\rho \rightarrow+\infty$.

In conclusion, there exist constants $0<v^{\star}<\min\{ v_{1}^{\star}, v_{2}^{\star}, v_{3}^{\star}\}$ and $C^{\star}>\max\{C_{1}^{\star}, C_{2}^{\star}\}$ such that \eqref{3.46} holds true.

\subsection{Proofs of Theorems \ref{Theorem 2.1} and \ref{Theorem 2.3}}
\noindent
In this section, we prove the following results regarding \eqref{3.1}, which obviously imply Theorems \ref{Theorem 2.1} and \ref{Theorem 2.3}. Let $V$ denote $V_{\boldsymbol{\nu}, \boldsymbol{\theta}, \boldsymbol{\tau}}$ in Propositions \ref{Proposition 3.5} and \ref{Proposition 3.6} for short.
\begin{proposition}\label{Proposition 3.5} 
 Suppose that $(F1)$ and $(F2)$ hold. Take $n$ $(\geq 2)$ unit vectors $\nu_i \in \mathbb{S}^{N-2}$ and $n$ angles $\theta_i\in (0, \pi / 2]$ $(i=1, \ldots, n)$ such that $\left(\nu_i, \theta_i\right) \neq\left(\nu_j, \theta_j\right)$ for $i \neq j$, and take $n$ constants $\tau_i$ $(i\in\{1, \ldots, n\})$. Let $\boldsymbol{\nu}=\left(\nu_1, \ldots, \nu_n\right)$, $\boldsymbol{\theta}=\left(\theta_1, \ldots, \theta_n\right)$ and $\boldsymbol{\tau}=\left(\tau_1, \ldots, \tau_n\right)$. Then there exists an entire solution $V_{\boldsymbol{\nu}, \boldsymbol{\theta}, \boldsymbol{\tau}}$ of \eqref{3.1} satisfying
\begin{equation*}
\underline{V}<V_{\boldsymbol{\nu}, \boldsymbol{\theta}, \boldsymbol{\tau}}<1 \quad \text{in } \mathbb{R} \times \mathbb{R}^N
\end{equation*}
and
\begin{equation}\label{3.48}
	\frac{\left|V_{\boldsymbol{\nu}, \boldsymbol{\theta}, \boldsymbol{\tau}}(t, z)-\underline{V}(t, z)\right|}{\min \left\{1, e^{-2 v^{\star} \min_{1\leq i\leq n} \left\{x \cdot \nu_i \cot \theta_i+y-\frac{c_f}{\sin \theta_i} t+\frac{\tau_i}{\sin \theta_i}  \right\}}\right\} } \rightarrow 0 \quad\text{uniformly as } d((t, x, y), \mathcal{R}) \rightarrow+\infty,
\end{equation} 
where $v^{\star}$ is a positive constant. Moreover,
\begin{description}
\item[(i)] For any $\left(t_0, x_0, y_0\right) \in \mathbb{R} \times \mathbb{R}^N$ satisfying $\min _{i=\{1, \ldots, n\}}\left\{x_0 \cdot \nu_i \cos \theta_i+y_0 \sin \theta_i-c_f t_0\right\} \geq 0$, one has
\begin{equation*}
V_{\boldsymbol{\nu}, \boldsymbol{\theta}, \boldsymbol{\tau}}\left(t-t_0, x-x_0, y-y_0\right) \geq V_{\boldsymbol{\nu}, \boldsymbol{\theta}, \boldsymbol{\tau}}(t, x, y).
\end{equation*}

\item[(ii)] For fixed $\left(\nu_i, \theta_i\right)$, $V_{\boldsymbol{\nu}, \boldsymbol{\theta}, \boldsymbol{\tau}}$ are decreasing in $\tau_i \in \mathbb{R}$.
\end{description}
\end{proposition}

\begin{proof}
Let $u_n(t, x, y)$ be the solution of following Cauchy problem
\begin{equation}\label{3.49}
\begin{cases}\partial_t u-\Delta_{x, y} u=f(u) & \text { when } t>-n, (x, y) \in \mathbb{R}^N ,\\ u(t, x, y)=\underline{V}(-n, x, y) & \text { when } t=-n, (x, y) \in \mathbb{R}^N.\end{cases}
\end{equation}
By Lemma \ref{Lemma 3.4} and the comparison principle, it holds that
\begin{equation}\label{3.50}
\underline{V}(t, x, y) \leq u_n(t, x, y) \leq \bar{V}_{\varepsilon, \alpha}(t, x, y)\quad \text {for } t \geq-n \text { and }(x, y) \in \mathbb{R}^N,
\end{equation}
where $0<\varepsilon<\varepsilon_0^{+}(\beta)$ and $0<\alpha<\alpha_0^{+}(\beta, \varepsilon)$ are sufficiently small constants that satisfy the conditions set in Lemma \ref{Lemma 3.4}. 

Since $\underline{V}(t, x, y)$ is a subsolution of \eqref{3.1}, the sequence $u_n(t, x, y)$ is increasing in $n$. By parabolic estimates, letting $n \rightarrow +\infty$, the sequence $\left\{u_n(t, x, y)\right\}_{n \in \mathbb{N}}$ converges locally uniformly in $C^{1, 2}$ to an entire solution $V(t, x, y)$ of \eqref{3.1}. Furthermore, it follows from \eqref{3.50} and the strong maximum principle that
\begin{equation*}
\underline{V}(t, x, y) < V(t, x, y) < \bar{V}_{\varepsilon, \alpha}(t, x, y) \quad \text {for }(t, x, y) \in \mathbb{R} \times \mathbb{R}^N.
\end{equation*}
By \eqref{3.45}, it holds that 
\begin{equation*}
\sqrt{V(t, x, y)-\underline{V}(t, x, y)} \leq C_{\star} \min \left\{1, e^{-2 v_{\star} \min_{1\leq i\leq n} \left\{x \cdot \nu_i \cot \theta_i+y-\frac{c_f}{\sin \theta_i} t+\frac{\tau_i}{\sin \theta_i}   \right\}}\right\} \quad \text{in }\mathbb{R} \times \mathbb{R}^N.
\end{equation*}
Since \eqref{3.14} and \eqref{3.46}, letting $\varepsilon \rightarrow 0$ in $\bar{V}_{\varepsilon, \alpha}(t, x, y)$ yields $0 < V < 1$ and \eqref{3.48} holds. It follows from \eqref{3.48} and the definition of $\underline{V}(t, x, y)$ that $V(t, x, y)$ is a transition front with sets
\begin{equation*}
\left\{\begin{array}{l}
	\Gamma_t=\left\{(x, y) \in \mathbb{R}^N ; \min _{1 \leq i \leq n}\left\{x \cdot \nu_i \cos \theta_i+y \sin \theta_i-c_f t+ \tau_i \right\}=0\right\},\\
	\Omega_t^{-}=\left\{(x, y) \in \mathbb{R}^N ;  \min _{1 \leq i \leq n}\left\{x \cdot \nu_i \cos \theta_i+y \sin \theta_i-c_f t+ \tau_i \right\}>0\right\},\\
	\Omega_t^{+}=\left\{(x, y) \in \mathbb{R}^N ;  \min _{1 \leq i \leq n}\left\{x \cdot \nu_i \cos \theta_i+y \sin \theta_i-c_f t+ \tau_i \right\}<0\right\}.
\end{array}\right.
\end{equation*}
It follows from \eqref{3.49} that for any $s>0$, $u_n(t+s, x, y)$ solves the following equation 
\begin{equation*}
\begin{cases}\partial_t u-\Delta_{x, y} u=f(u) & \text { when } t>-n,(x, y) \in \mathbb{R}^N, \\ u(t, x, y)=u_n(-n+s, x, y) & \text { when } t=-n,(x, y) \in \mathbb{R}^N,\end{cases}
\end{equation*}
for each $n \in \mathbb{N}$. Then \eqref{3.50} implies that
\begin{equation*}
u_n(-n+s, x, y) \geq \underline{V}(-n+s, x, y) \geq \underline{V}(-n, x, y),\quad \forall(x, y) \in \mathbb{R}^N.
\end{equation*}
By the comparison principle, it holds that
\begin{equation*}
u_n(t+s, x, y) \geq u_n(t, x, y), \quad \forall(t, x, y) \in(-n,+\infty) \times \mathbb{R}^N, \;\forall \tau>0,
\end{equation*}
which shows $\partial_t u_n(t, x, y) \geq 0$ in $(-n,+\infty) \times \mathbb{R}^N$. Hence, letting $n \rightarrow +\infty$, $\partial_t V(t, x, y) \geq 0$ in $\mathbb{R} \times \mathbb{R}^N$. It follows from the strong maximum principle that $\partial_t V(t, x, y) > 0$ in $\mathbb{R} \times \mathbb{R}^N$. 

Since $U^{\prime}<0$, for any $\left(t_0, x_0, y_0\right) \in \mathbb{R} \times \mathbb{R}^N$ such that $\min _{i=\{1, \ldots, n\}}\left\{x_0 \cdot \nu_i \cos \theta_i+y_0 \sin \theta_i-c_f t_0\right\} \geq 0$, it holds that
\begin{equation}\label{3.52}
\begin{aligned}
   &\underline{V}\left(t-t_0, x-x_0, y-y_0\right)\\
   &=\max _{1 \leq i \leq n}\left\{U \left(\left(x-x_0\right) \cdot \nu_i \cos \theta_i+\left(y-y_0\right) \sin \theta_i-c_f\left(t-t_0\right)+\tau_i\right)\right\} \geq \underline{V}(t, x, y).
\end{aligned}
\end{equation}
If $\min _{i=\{1, \ldots, n\}}\left\{x_0 \cdot \nu_i \cos \theta_i+y_0 \sin \theta_i-c_f t_0\right\}>0$, then one knows that \eqref{3.52} is strict. By the comparison principle, we obtain $u_n\left(t-t_0, x-x_0, y-y_0\right) \geq u_n(t, x, y)$. Letting $n \rightarrow+\infty$, one has $V\left(t-t_0, x-x_0, y-y_0\right) \geq V(t, x, y)$. 

The monotonicity of $V_{\boldsymbol{\nu}, \boldsymbol{\theta}, \boldsymbol{\tau}}$ with respect to $\tau_i$ can be derived in the same way. Fix any $\boldsymbol{\nu}=\left(\nu_1, \ldots, \nu_n\right)$ and $\boldsymbol{\theta}=\left(\theta_1, \ldots, \theta_n\right)$ such that $\left(\nu_i, \theta_i\right) \neq\left(\nu_j, \theta_j\right)$ for $i \neq j$. Denote $\boldsymbol{\tau}_{1}=\left(\tau_1^{1}, \ldots, \tau_n^{1}\right)$ and $\boldsymbol{\tau}_{2}=\left(\tau_1^{2}, \ldots, \tau_n^{2}\right)$, where $\tau_{l}^{1}<\tau_{l}^{2}$ for some $l\in\{1,\ldots,n\}$ and $\tau_{i}^{1}=\tau_{i}^{2}$ $(i \neq l)$. Define
\begin{equation*}
	\underline{V}_{\boldsymbol{\tau}_{1}}(t, x, y)=\max _{1 \leq i \leq n}\left\{U \left(\left(x-x_0\right) \cdot \nu_i \cos \theta_i+\left(y-y_0\right) \sin \theta_i-c_f\left(t-t_0\right)+\tau_i^{1}\right)\right\}
\end{equation*}
and
\begin{equation*}
	\underline{V}_{\boldsymbol{\tau}_{2}}(t, x, y)=\max _{1 \leq i \leq n}\left\{U \left(\left(x-x_0\right) \cdot \nu_i \cos \theta_i+\left(y-y_0\right) \sin \theta_i-c_f\left(t-t_0\right)+\tau_i^{2}\right)\right\}.
\end{equation*}
Then it holds that
\begin{equation*}
	\begin{aligned}
	&\min_{1 \leq i \leq n}\left\{U \left(\left(x-x_0\right) \cdot \nu_i \cos \theta_i+\left(y-y_0\right) \sin \theta_i-c_f\left(t-t_0\right)+\tau_i^{1}\right)\right\}\\
	&>\min_{1 \leq i \leq n}\left\{U \left(\left(x-x_0\right) \cdot \nu_i \cos \theta_i+\left(y-y_0\right) \sin \theta_i-c_f\left(t-t_0\right)+\tau_i^{2}\right)\right\}.
	\end{aligned}
\end{equation*}
Thus $\underline{V}_{\boldsymbol{\tau}_{1}}(t, x, y)>\underline{V}_{\boldsymbol{\tau}_{2}}(t, x, y)$. By the comparison principle, one has $u_{n, \boldsymbol{\tau}_{1}}(t, x, y)\geq u_{n, \boldsymbol{\tau}_{2}}(t, x, y)$. Letting $n \rightarrow+\infty$, we obtain $V_{\boldsymbol{\tau}_{1}}(t, x, y) \geq V_{\boldsymbol{\tau}_{2}}(t, x, y)$, and thereby $V_{\boldsymbol{\nu}, \boldsymbol{\theta}, \boldsymbol{\tau}}$ are decreasing in $\tau_i \in \mathbb{R}$ for fixed $\left(\nu_i, \theta_i\right)$.
\end{proof}

\begin{proposition}\label{Proposition 3.6}
 Suppose that $(F1)$ and $(F2)$ hold. If $\widetilde{V}$ is an entire solution of \eqref{3.1} satisfying
\begin{equation}\label{3.53}
	|\widetilde{V}(t, x, y)-\underline{V}(t, x, y)| \rightarrow 0 \quad\text{uniformly as } d((t, x, y), \mathcal{R}) \rightarrow+\infty,
\end{equation}
then $\widetilde{V} \equiv V$ in $\mathbb{R} \times \mathbb{R}^N$. Moreover, $V_{\boldsymbol{e}, \boldsymbol{\tau}}(t, z)$ depend continuously on $\left(\tau_1, \ldots, \tau_n\right) \in \mathbb{R}^n$ in the sense of $\mathcal{T}$.
\end{proposition}
\begin{proof}
	Since $V$ satisfies \eqref{3.48} and $\widetilde{V}$ satisfies \eqref{3.53}, they are both transition fronts connecting $0$ and $1$ with sets $\Gamma_t$, $\Omega_t^{-}$ and $\Omega_t^{+}$. Then there is a constant $R > 0$ such that $0<\widetilde{V}(t, x, y), V(t, x, y) \leq \gamma_{\star}$ for $(t, x, y) \in \omega^{-}$ and $1-\gamma_{\star} \leq \widetilde{V}(t, x, y), V(t, x, y)<1$ for $(t, x, y) \in \omega^{+}$,
	where
\begin{equation*}
	\omega^{+}:=\left\{(t, x, y) \in \mathbb{R} \times \mathbb{R}^N ; \min _{1 \leq i \leq n}\left\{x \cdot \nu_i \cos \theta_i+y \sin \theta_i-c_f t+\tau_i\right\} <-R\right\},
\end{equation*}
\begin{equation*}
	\omega^{-}:=\left\{(t, x, y) \in \mathbb{R} \times \mathbb{R}^N ; \min _{1 \leq i \leq n}\left\{x \cdot \nu_i \cos \theta_i+y \sin \theta_i-c_f t+\tau_i\right\} > R\right\}
\end{equation*}
and
\begin{equation*}
	\omega:=\left\{(t, x, y) \in \mathbb{R} \times \mathbb{R}^N ;-R \leq \min _{1 \leq i \leq n}\left\{x \cdot \nu_i \cos \theta_i+y \sin \theta_i-c_f t+\tau_i\right\} \leq R\right\} .
\end{equation*}
The remaining proof  is similar to \cite[Propositions 3.9 and 3.10]{H. Guo2024} and  is omitted 
here.
\end{proof}

Finally, similar to \cite[Lemma 3.11]{H. Guo2024}, it can be shown that $V$ is monotone in $t$.
\begin{lemma}\label{Lemma 3.7}
   In $\mathbb{R} \times \mathbb{R}^N$, $\partial_t V> 0$. For any $\rho>0$, there exists a $k(\rho)>0$ such that
\begin{equation*}
	\partial_t V \geq k(\rho) \quad \text{in }\{(t, x, y): d((t, x, y), \partial \mathcal{P}) \leq \rho\} .
\end{equation*}
\end{lemma}

\section{Stability of curved fronts}
\noindent
In this section, we study the stability of curved fronts in Theorem \ref{Theorem 2.1} and still consider $e_0 = (0, 0, \ldots, 1)$ for convenience. Firstly, we construct super- and subsolutions for Cauchy problem \eqref{2.5} (ignore initial condition).
\begin{lemma}\label{Lemma 4.1}
 For any $\beta \in\left(0, \beta^*\right]$ and any $0<\varepsilon<\varepsilon_0^{+}(\beta)$, there exist positive constants $\lambda(\beta)$ and $\varrho(\beta, \lambda)$ such that for each $\alpha>0$ small enough,
\begin{equation*}
	W_{\delta}^{+}(t, x, y):=\min\left\{\bar{V}(\varpi, x, y)+\delta e^{-\lambda t} \times\left[U^\beta(\eta(\varpi, x, y)) \omega(\eta(\varpi, x, y))+(1-\omega(\eta(\varpi, x, y)))\right], 1\right\}
\end{equation*}
	is a supersolution of \eqref{3.1} for $t \geq 0$ and $(x, y) \in \mathbb{R}^N$, for all $\delta \in\left(0,  \gamma_{\star} / 8\right]$, where $\varpi=\varpi(t):=t-\varrho \delta e^{-\lambda t}+\varrho \delta$, and $\beta^*$, $\varepsilon_0^{+}(\beta)$, $\bar{V}$ are given in Lemma \ref{Lemma 3.4}, and $\eta$, $\omega$, $\gamma_{\star}$ are given in \eqref{1.4}, \eqref{3.7} and \eqref{3.12}. 
\end{lemma}
\begin{proof}
To show that $W_{\delta}^{+}(t, x, y)$ is a supersolution, we only to prove the domain where $W_{\delta}^{+}(t, x, y)<1$. Let $X^{\prime}>1$ and $X^{\prime \prime}>1$ be given in Lemma \ref{Lemma 3.4}, then consider the inequality
\begin{equation*}
\mathcal{L} W_{\delta}^{+}:=\partial_t	W_{\delta}^{+}-\Delta_{x, y} W_{\delta}^{+}-f\left(W_{\delta}^{+}\right) \geq 0, \quad \forall(t, x, y) \in[0,+\infty) \times \mathbb{R}^N,
\end{equation*}
in three cases $\eta(\varpi, x, y)>X^{\prime}$, $\eta(\varpi, x, y)<-X^{\prime \prime}$, and $\eta(\varpi, x, y) \in\left[-X^{\prime \prime}, X^{\prime}\right]$, respectively. Since $\bar{V}$ is a supersolution of \eqref{3.1} from Lemma \ref{Lemma 3.4}, one has
\begin{equation}\label{4.1}
\begin{aligned}
	\mathcal{L} W_{\delta}^{+} \geq & \varrho \delta \lambda e^{-\lambda t} \bar{V}_\varpi +f\left(\bar{V}\right)-f\left(W_{\delta}^{+}\right)+\left(\partial_t-\Delta_{x, y}\right)\left(\delta e^{-\lambda t} \times\left[U^\beta(\eta) \omega(\eta)+(1-\omega(\eta))\right]\right)
\end{aligned}
\end{equation}
in $\mathbb{R} \times \mathbb{R}^N$, where $\xi$, $\eta$, $\bar{V}$ and all of its derivatives are evaluated at $(\varpi(t), x, y)$. Take 
\begin{equation*}
	 0< \lambda <\min\left\{-\frac{f^{\prime}(1)}{4}, \frac{\beta c_{f}^{2}}{16 }\right\}\quad \text{and} \quad \varrho= 3 \frac{\left\|f_u\right\|_{L^{\infty}([0, 1])}+\lambda+C^{*}(\beta)}{\lambda \kappa c_{f}},
\end{equation*} 
where $C^{*}(\beta)>0$ is to be determined. Choose $\alpha>0$ small enough such that
\begin{equation*}
	0<\alpha<\min\left\{\frac{\beta_{0} e}{\varrho C_{2} L_{3}}, \frac{ 3 L_{1}^{\beta} e}{4 \varrho C_{2} L_{2}}, \alpha_0^{+}(\beta, \varepsilon)\right\},
\end{equation*}
where $\alpha_0^{+}(\beta, \varepsilon)$ is given in Lemma \ref{Lemma 3.4}.

\textbf{Case 1:} $\eta(\varpi(t), x, y)>X^{\prime}$ and $t \geq 0$.
In this case, $\omega(\eta) \equiv 1$. Hence,
\begin{equation*}
	\mathcal{L} W_{\delta}^{+} \geq \varrho \delta \lambda e^{-\lambda t} \bar{V}_\varpi +\left(\partial_t-\Delta_{x, y}\right)\left(\delta e^{-\lambda t} U^\beta(\eta)\right)+f\left(\bar{V}\right)-f\left(W_{\delta}^{+}\right).
\end{equation*}
It can be calculated that
	\begin{align}\label{4.2}		
		&\left(\partial_t-\Delta_{x, y}\right)\left(\delta e^{-\lambda t} U^\beta(\eta) \right) \nonumber\\
		&= -\lambda\delta e^{- \lambda t}U^\beta(\eta)+\delta\beta e^{- \lambda t} U^{\beta-1}(\eta) U^{\prime}(\eta)\partial_{\varpi}\eta (1+\varrho \delta \lambda e^{- \lambda t}) \nonumber\\
		&\quad-\delta\beta(\beta-1) e^{- \lambda t} U^{\beta-2}(\eta) (U^{\prime} (\eta))^2\left(\sum_{k=1}^{N-1}(\partial_{x_k}\eta)^2 + (\partial_{y}\eta)^2\right)\nonumber\\
		&\quad-\delta\beta e^{- \lambda t} U^{\beta-1}(\eta) U^{\prime \prime}(\eta) \left(\sum_{k=1}^{N-1}(\partial_{x_k}\eta)^2 + (\partial_{y}\eta)^2\right)-\delta\beta e^{- \lambda t} U^{\beta-1}(\eta) U^{\prime}(\eta)\left(\sum_{k=1}^{N-1}\partial_{x_k x_k}\eta + \partial_{y y}\eta\right)\nonumber\\
		&\geq -\lambda\delta e^{- \lambda t}U^\beta(\eta) +\delta\beta e^{- \lambda t} U^{\beta-1}(\eta) U^{\prime}(\eta)(-\partial_{\alpha \varpi}\varphi(\alpha \varpi, \alpha x, \alpha)) (1+\varrho \delta \lambda e^{- \lambda t}) \nonumber\\
		&\quad -\delta\beta e^{- \lambda t}  U^{\beta}(\eta)\left[\beta(\beta-1)\frac{(U^{\prime}(\eta))^{2}}{U^{2}(\eta)}\left(\sum_{k=1}^{N-1}(-\partial_{\zeta_k}\varphi(\alpha \varpi, \alpha x, \alpha))^2+1\right)\right.\nonumber\\
		&\quad \left. +\frac{U^{\prime \prime}(\eta)}{U(\eta)}\left(\sum_{k=1}^{N-1}(-\partial_{\zeta_k}\varphi(\alpha \varpi, \alpha x, \alpha))^2+1\right)-\frac{U^{\prime}(\eta)}{U(\eta)}\left(\alpha \sum_{k=1}^{N-1}\partial_{\zeta_k \zeta_k}\varphi(\alpha \varpi, \alpha x, \alpha)\right)\right]\nonumber\\
		&\geq -\lambda\delta e^{- \lambda t}U^\beta(\eta) +\delta\beta e^{- \lambda t}  U^{\beta}(\eta) \frac{U^{\prime}(\eta)}{U(\eta)}\left(\alpha \sum_{k=1}^{N-1}\partial_{\zeta_k \zeta_k}\varphi(\alpha \varpi, \alpha x, \alpha)\right) \nonumber\\
		&\quad -\delta\beta e^{- \lambda t}  U^{\beta}(\eta)\left[\beta(\beta-1)\frac{(U^{\prime}(\eta))^{2}}{U^{2}(\eta)}\left(\sum_{k=1}^{N-1}(-\partial_{\zeta_k}\varphi(\alpha \varpi, \alpha x, \alpha))^2+1\right)\right.\nonumber\\
		&\quad \left. +\frac{U^{\prime \prime}(\eta)}{U(\eta)}\left(\sum_{k=1}^{N-1}(-\partial_{\zeta_k}\varphi(\alpha \varpi, \alpha x, \alpha))^2+1\right) + \frac{U^{\prime}(\eta)}{U(\eta)}\partial_{\alpha \varpi}\varphi(\alpha \varpi, \alpha x, \alpha) \right]\nonumber\\
		&=: -\lambda\delta e^{- \lambda t}U^\beta(\eta) + Q_1 +Q_2.
	\end{align}
From Lemma \ref{Lemma 3.4}, recall that
\begin{equation*}
	\beta^{*}\leq \frac{1}{4 \left((C_{1}+\max_{1\le i\le n} \{|\nu_i \cot \theta_i|\})^2+1\right)}.
\end{equation*}
Similar to the proof of Case $1$ in Lemma \ref{Lemma 3.4}, for any $\beta \in\left(0, \beta^{*}\right]$, one has
\begin{equation}\label{4.3}
	\begin{aligned}
		Q_2\geq\frac{c_{f}^{2}}{4}\delta\beta e^{- \lambda t}  U^{\beta}(\eta) \quad\text{for all  } \eta \in\left(X^{\prime}, +\infty\right) \text{ and } t \geq 0.
	\end{aligned}
\end{equation}
 It follows from \eqref{4.2} and \eqref{4.3} that for arbitrary $0<\alpha < \alpha_0^{+}(\beta, \varepsilon) \leq \alpha_1^{+}(\beta)$ ($\alpha_0^{+}(\beta, \varepsilon),\alpha_1^{+}(\beta)$ are given in Lemma \ref{Lemma 3.4}),
\begin{equation}\label{4.4}
Q_1+Q_2>\delta e^{- \lambda t} U^{\beta}(\eta)\times \beta\frac{c_{f}^{2}}{8}\quad\text{for all  } \eta \in\left(X^{\prime}, +\infty\right) \text{ and } t \geq 0.
\end{equation}
Then it follows from \eqref{4.1}-\eqref{4.4} that
\begin{equation*}
	\mathcal{L} W_{\delta}^{+} \geq \varrho \delta \lambda e^{-\lambda t} \bar{V}_\varpi -\lambda\delta e^{- \lambda t}U^{\beta}(\eta)+\delta e^{- \lambda t} U^{\beta}(\eta)\times \beta\frac{c_{f}^{2}}{8 }+f\left(\bar{V}\right)-f\left(W_{\delta}^{+}\right).
\end{equation*}
By \eqref{1.2} and definitions of $\bar{V}$ and $W_{\delta}^{+}$, we know $W_{\delta}^{+}, \bar{V} \leq 3 \gamma_{\star}$ for $\eta> X^{\prime}$ and $t\geq 0$. Thus
\begin{equation*}
	f\left(\bar{V}\right)-f\left(W_{\delta}^{+}\right)= 0,\quad\forall\eta \in\left(X^{\prime},+\infty\right) \text{ and } t \geq 0.
\end{equation*}
It follows from \eqref{3.13} and Lemma \ref{Lemma 3.3} that 
\begin{equation*}
	\begin{aligned}
		\frac{\partial}{\partial \varpi}\left(\bar{V}(\varpi, x, y)\right) & 
		 = -U^{\prime}(\xi) \left(\frac{\partial_t \varphi}{\sqrt{1+|\nabla \varphi|^2}}+\alpha \xi \frac{\nabla \varphi \cdot \nabla \partial_t \varphi}{1+|\nabla \varphi|^2}\right)\\
		 &\geq - c_{f} U^{\prime}(\xi) + C_{2} \alpha h(\alpha t, \alpha x, \alpha) U^{\prime}(\xi) \left|\xi\right|,
	\end{aligned}
\end{equation*}
where $\xi$ is evaluated at $(\varpi, x, y)$. Thus we have
\begin{equation*}
	\mathcal{L} W_{\delta}^{+} \geq  \delta e^{-\lambda t} \left(- \varrho \lambda  c_{f} U^{\prime}(\xi) + \varrho \lambda  C_{2} \alpha h(\alpha t, \alpha x, \alpha) U^{\prime}(\xi) \left|\xi\right| -\lambda U^{\beta}(\eta)+ U^{\beta}(\eta)\times \beta\frac{c_{f}^{2}}{8 }\right).
\end{equation*}

Since $0<\beta\leq \beta^{*}$, it holds that 
\begin{equation*}
0< \beta \eta\leq\frac{\eta}{4 \sqrt{(C_1 +\max_{1\le i\le n} \{|\nu_i \cot \theta_i|\})^2+1}}\leq \frac{y-\frac{1}{\alpha} \varphi(\alpha t, \alpha x, \alpha)}{4 \sqrt{1+|\nabla \varphi(\alpha t, \alpha x, \alpha)|^2}}=\frac{1}{4}\xi.
\end{equation*}
By \eqref{1.5}, $0<\lambda<\frac{\beta c_{f}^{2}}{16 }$ and the fact $h=\sum_{i, j \in\{1, \ldots, n\} ; i\neq j}e^{-\left(\hat{q}_i+\hat{q}_j\right)}\le 1$, we obtain
\begin{equation*}
	\begin{aligned}
		\mathcal{L} W_{\delta}^{+} &\geq  \delta e^{-\lambda t} \left(  \frac{\beta c_{f}^{2}}{16 } \varrho C_{2} \alpha h(\alpha t, \alpha x, \alpha) U^{\prime}(\xi) \left|\xi\right|+ U^{\beta}(\eta)\times \frac{\beta c_{f}^{2}}{16 }\right)\\
		&\geq  \frac{\beta c_{f}^{2}}{16 } \delta e^{-\lambda t} \left(  \varrho C_{2} \alpha U^{\prime}(\xi) \left|\xi\right|+ U^{\beta}(\eta)\right)\\
		&\geq  \frac{\beta c_{f}^{2}}{16 } \delta e^{-\lambda t} \left(  \varrho C_{2} \alpha \left(-L_{2} e^{- c_{f} \xi}\right) \left|\xi\right|+ L_{1}^{\beta} e^{- \beta c_{f} \eta}\right)\\
		&\geq  \frac{\beta c_{f}^{2}}{16 } \delta e^{-\lambda t} \left( - \varrho L_{2} C_{2} \alpha e^{- c_{f} \xi} \left|\xi\right|+ L_{1}^{\beta} e^{- \frac{c_{f}}{4}\xi  }\right)
	\end{aligned}
\end{equation*}
It follows from $0<\alpha< \frac{ 3 L_{1}^{\beta} e}{4 \varrho C_{2} L_{2}}$ that $L_{1}^{\beta} e^{- \frac{c_{f}}{4}\xi  }> \varrho L_{2} C_{2} \alpha e^{- c_{f} \xi} \left|\xi\right|$ for $\eta> X^{\prime}$ and $t\geq 0$. Then one knows $\mathcal{L} W_{\delta}^{+} >0$ in Case $1$.

\textbf{Case 2:} $\eta(\varpi(t), x, y)<-X^{\prime\prime}$ and $t \geq 0$.
In this case, $\omega(\eta) \equiv 0$. Thus one has
\begin{equation*}
	\mathcal{L} W_{\delta}^{+} \geq \varrho \delta \lambda e^{-\lambda t} \bar{V}_\varpi -\lambda\delta e^{- \lambda t} +f\left(\bar{V}\right)-f\left(W_{\delta}^{+}\right).
\end{equation*}
Since $\delta \in\left(0,  \gamma_{\star} / 8\right]$ and $0<\varepsilon<\gamma_{\star} / 6$, by definitions of $\bar{V}$ and $W_{\delta}^{+}$, one has $W_{\delta}^{+}, \bar{V} \in [1-2\gamma_{\star},1+2\gamma_{\star}]$ for $\eta<-X^{\prime\prime}$ and $t \geq 0$. 
Then we obtain
\begin{equation*}
	f\left(\bar{V}\right)-f\left(W_{\delta}^{+}\right) \geq -\frac{f^{\prime}(1)}{2} \delta  e^{-\lambda t},\quad\forall\eta \in\left(-\infty, -X^{\prime \prime}\right) \text{ and } t \geq 0.
\end{equation*}
It follows from \eqref{3.13} and \eqref{3.20} that 
\begin{equation*}
	\begin{aligned}
		\frac{\partial}{\partial \varpi}\left(\bar{V}(\varpi, x, y)\right) & = U^{\prime}(\xi) \xi_\varpi + \varepsilon \frac{\partial}{\partial \varpi}\left[h(\alpha \varpi, \alpha x, \alpha) \times U^\beta(\eta) \right]\\
		& = -U^{\prime}(\xi) \left(\frac{\partial_t \varphi}{\sqrt{1+|\nabla \varphi|^2}}+\alpha \xi \frac{\nabla \varphi \cdot \nabla \partial_t \varphi}{1+|\nabla \varphi|^2}\right)\\
		&\quad + \varepsilon \alpha \frac{\partial h(\alpha \varpi, \alpha x, \alpha)}{\partial \alpha \varpi} \times U^\beta(\eta) +\varepsilon h(\alpha \varpi, \alpha x, \alpha) \times\frac{\partial}{\partial \varpi}\left(U^\beta(\eta) \right),
	\end{aligned}
\end{equation*}
where $\xi$ and $\eta$ are evaluated at $(\varpi, x, y)$. Then we can get
\begin{equation*}
	\varepsilon \alpha \frac{\partial h(\alpha \varpi, \alpha x, \alpha)}{\partial \alpha  \varpi} \times U^\beta(\eta) =2 \varepsilon \alpha  c_{f} h(\alpha \varpi, \alpha x, \alpha) U^\beta(\eta) \geq 0,
\end{equation*}
\begin{equation*}
	\begin{aligned}
		\varepsilon h(\alpha \varpi, \alpha x, \alpha) \times\frac{\partial}{\partial \varpi}\left(U^\beta(\eta)\right)= \varepsilon \beta h(\alpha \varpi, \alpha x, \alpha) U^{\beta-1}(\eta) \partial_\eta U(\eta) \partial_\varpi \eta\geq 0,
	\end{aligned}
\end{equation*}
and
\begin{equation*}
	\begin{aligned}
		&-U^{\prime}(\xi) \left(\frac{\partial_t \varphi}{\sqrt{1+|\nabla \varphi|^2}}+\alpha \xi \frac{\nabla \varphi \cdot \nabla \partial_t \varphi}{1+|\nabla \varphi|^2}\right)\geq - c_{f} U^{\prime}(\xi) - C_{2} \alpha h(\alpha t, \alpha x, \alpha) \left|U^{\prime}(\xi) \xi\right|.
	\end{aligned}
\end{equation*}
We know 
\begin{equation*}
	\mathcal{L} W_{\delta}^{+} \geq \varrho \delta \lambda e^{-\lambda t} \left(- c_{f} U^{\prime}(\xi) - C_{2} \alpha h(\alpha t, \alpha x, \alpha) \left|U^{\prime}(\xi) \xi\right|\right) -\lambda\delta e^{- \lambda t}-\frac{f^{\prime}(1)}{2} \delta  e^{-\lambda t}.
\end{equation*}
By \eqref{1.5}, $0<\lambda<-\frac{f^{\prime}(1)}{4}$, $0<\beta\leq \beta^{*}$ and the fact $h=\sum_{i, j \in\{1, \ldots, n\} ; i\neq j}e^{-\left(\hat{q}_i+\hat{q}_j\right)}\leq 1$, one has
\begin{equation*}
	\mathcal{L} W_{\delta}^{+} \geq -\frac{f^{\prime}(1)}{4} \delta e^{-\lambda t} \left(-\varrho C_{2} \alpha \left|U^{\prime}(\xi) \xi\right|+1\right)\geq -\frac{f^{\prime}(1)}{4} \delta e^{-\lambda t} \left(-\varrho C_{2} \alpha L_3 e^{\beta_{0} \xi} \left|\xi\right|+1\right) .
\end{equation*}
Since $0<\alpha<\frac{\beta_{0} e}{\varrho C_{2} L_{3}}$, we get $1>\varrho C_{2} \alpha L_3 e^{\beta_{0} \xi} \left|\xi\right|$ for $\eta<-X^{\prime\prime}$ and $t \geq 0$. Then we know $\mathcal{L} W_{\delta}^{+} >0$ in Case $2$.

\textbf{Case 3:} $-X^{\prime\prime}\leq \eta(\varpi(t), x, y)\leq X^{\prime}$ and $t \geq 0$.
	\begin{align*}
	&\left(\partial_t-\Delta_{x, y}\right)\left(\delta e^{-\lambda t} \times\left[U^\beta(\eta) \omega(\eta)+(1-\omega(\eta))\right]\right) \\
	&\geq-\delta \lambda e^{- \lambda t}+\delta \beta e^{- \lambda t} U^{\beta-1}(\eta) U^{\prime}(\eta)\omega(\eta) (-\partial_{\alpha \varpi}\varphi(\alpha \varpi, \alpha x, \alpha)) (1+\varrho \delta \lambda e^{- \lambda t})\\ 
	&\quad+\delta \beta e^{- \lambda t} U^{\beta}(\eta) \omega^{\prime}(\eta) \partial_{\varpi} \eta (1+\varrho \delta \lambda e^{- \lambda t}) + \delta \beta e^{- \lambda t} (-\omega^{\prime}(\eta)) \partial_{\varpi} \eta (1+\varrho \delta \lambda e^{- \lambda t}) \\
	&\quad-\delta \beta(\beta-1) e^{- \lambda t} U^{\beta-2}(\eta)(U^{\prime}(\eta))^2 \omega(\eta)\left(\sum_{k=1}^{N-1}(\partial_{x_k}\eta)^2 + (\partial_{y}\eta)^2\right) \\
	&\quad-\delta \beta e^{- \lambda t} U^{\beta-1}(\eta) U^{\prime \prime}(\eta) \omega(\eta) \left(\sum_{k=1}^{N-1}(\partial_{x_k}\eta)^2 + (\partial_{y}\eta)^2\right)\\
	&\quad-\delta\beta  e^{- \lambda t} U^{\beta-1}(\eta) U^{\prime}(\eta) \omega(\eta) \left(\sum_{k=1}^{N-1}\partial_{x_k x_k}\eta + \partial_{y y}\eta\right)\\
	&\quad-2\delta \beta e^{- \lambda t}U^{\beta-1}(\eta) U^{\prime}(\eta) \omega^{\prime}(\eta) \left[\sum_{k=1}^{N-1}(\partial_{x_k} \eta)^2+(\partial_{y} \eta)^2\right]- \delta e^{- \lambda t} U^{\beta}(\eta) \omega^{\prime \prime}(\eta) \left[\sum_{k=1}^{N-1}(\partial_{x_k} \eta)^2+(\partial_{y} \eta)^2\right]\\
	&\quad-\delta e^{- \lambda t} U^{\beta}(\eta) \omega^{\prime}(\eta) \left[\sum_{k=1}^{N-1}\partial_{x_k x_k}\eta+\partial_{y y}\eta\right]+\delta e^{- \lambda t} \omega^{\prime \prime}(\eta) \left[\sum_{k=1}^{N-1}(\partial_{x_k} \eta)^2+(\partial_{y} \eta)^2\right]\\
	&\quad+\delta e^{- \lambda t}\omega^{\prime}(\eta) \left[\sum_{k=1}^{N-1}\partial_{x_k x_k}\eta+\partial_{y y}\eta\right].
	\end{align*}
By calculation, we can obtain
\[
\begin{aligned}
	\delta \beta e^{- \lambda t} U^{\beta-1}(\eta) U^{\prime}(\eta)\omega(\eta) (-\partial_{\alpha \varpi}\varphi(\alpha \varpi, \alpha x, \alpha)) (1+\varrho \delta \lambda e^{- \lambda t})\geq 0,
\end{aligned}
\]
\[
\begin{aligned}
\delta \beta e^{- \lambda t} (U^{\beta}(\eta)-1) \omega^{\prime}(\eta) \partial_{\varpi} \eta (1+\varrho \delta \lambda e^{- \lambda t})\geq \delta \beta e^{- \lambda t} (1-U^{\beta}(\eta)) \omega^{\prime}(\eta) c_{f} (1+\varrho \delta \lambda e^{- \lambda t})\geq 0 ,
\end{aligned}
\]
\[
\begin{aligned}
	&-\delta \beta(\beta-1) e^{- \lambda t} U^{\beta-2}(\eta)(U^{\prime}(\eta))^2 \omega(\eta)\left(\sum_{k=1}^{N-1}(\partial_{x_k}\eta)^2 + (\partial_{y}\eta)^2\right)\\
	&\quad\geq-\delta \beta  e^{- \lambda t} \left(\max\{L_2 e^{-c_{f} |\eta|}, L_3 e^{-\beta_{0} |\eta|}\}\right)^2 \left[1+(C_1 h(\alpha \varpi, \alpha x, \alpha)+\max_{1\le i\le n} \{|\nu_i \cot \theta_i|\})^2\right],
\end{aligned}
\]
\[
\begin{aligned}
	&-\delta \beta e^{- \lambda t} U^{\beta-1}(\eta) U^{\prime \prime}(\eta) \omega(\eta) \left(\sum_{k=1}^{N-1}(\partial_{x_k}\eta)^2 + (\partial_{y}\eta)^2\right)\\
	&\quad\geq-\delta \beta  e^{- \lambda t}\max\{L_2 e^{-c_{f} |\eta|}, L_3 e^{-\beta_{0} |\eta|}\}\left[1+(C_1 h(\alpha \varpi, \alpha x, \alpha)+\max_{1\le i\le n} \{|\nu_i \cot \theta_i|\})^2\right],
\end{aligned}
\]
\[
\begin{aligned}
	&-\delta\beta  e^{- \lambda t} U^{\beta-1}(\eta) U^{\prime}(\eta) \omega(\eta) \left(\sum_{k=1}^{N-1}\partial_{x_k x_k}\eta + \partial_{y y}\eta\right)\\
	&\quad\geq-\delta \alpha\beta (N-1) C_{1} e^{- \lambda t}  h(\alpha \varpi, \alpha x, \alpha) \max\{L_2 e^{-c_{f} |\eta|}, L_3 e^{-\beta_{0} |\eta|}\},
\end{aligned}
\]
\[
\begin{aligned}
	&-2\delta \beta e^{- \lambda t}U^{\beta-1}(\eta) U^{\prime}(\eta) \omega^{\prime}(\eta) \left[\sum_{k=1}^{N-1}(\partial_{x_k} \eta)^2+(\partial_{y} \eta)^2\right]\\
	&\quad\geq -2\delta \beta e^{- \lambda t}\omega^{\prime}  \max\{L_2 e^{-c_{f} |\eta|}, L_3 e^{-\beta_{0} |\eta|}\} \left[1+(C_1 h(\alpha \varpi, \alpha x, \alpha)+\max_{1\le i\le n} \{|\nu_i \cot \theta_i|\})^2\right],
\end{aligned}
\]
\[
\begin{aligned}
- \delta e^{- \lambda t} U^{\beta}(\eta) \omega^{\prime \prime}(\eta) \left[\sum_{k=1}^{N-1}(\partial_{x_k} \eta)^2+(\partial_{y} \eta)^2\right]\geq - \delta e^{- \lambda t} |\omega^{\prime \prime}| \left[1+(C_1 h(\alpha \varpi, \alpha x, \alpha)+\max_{1\le i\le n} \{|\nu_i \cot \theta_i|\})^2\right],
\end{aligned}
\]
\[
\begin{aligned}
-\delta e^{- \lambda t} U^{\beta}(\eta) \omega^{\prime}(\eta) \left[\sum_{k=1}^{N-1}\partial_{x_k x_k}\eta+\partial_{y y}\eta\right]\geq - \delta\alpha C_1 (N-1) e^{- \lambda t} \omega^{\prime} h(\alpha \varpi, \alpha x, \alpha),
\end{aligned}
\]
\[
\begin{aligned}
	\delta e^{- \lambda t} \omega^{\prime \prime}(\eta) \left[\sum_{k=1}^{N-1}(\partial_{x_k} \eta)^2+(\partial_{y} \eta)^2\right] \geq -\delta e^{- \lambda t} |\omega^{\prime \prime}| \left[1+(C_1 h(\alpha \varpi, \alpha x, \alpha)+\max_{1\le i\le n} \{|\nu_i \cot \theta_i|\})^2\right]
\end{aligned}
\]
and
\[
\begin{aligned}
	\delta e^{- \lambda t}\omega^{\prime}(\eta) \left[\sum_{k=1}^{N-1}\partial_{x_k x_k}\eta+\partial_{y y}\eta\right] \geq - \delta\alpha C_1 (N-1) e^{- \lambda t} \omega^{\prime} h(\alpha \varpi, \alpha x, \alpha). 
\end{aligned}
\]
Thus there exists a constant $C^{*}(\beta)>0$ such that 
\begin{equation*}
	\left(\partial_t-\Delta_{x, y}\right)\left(\delta e^{-\lambda t} \times\left[U^\beta(\eta) \omega(\eta)+(1-\omega(\eta))\right]\right)\geq -\delta \lambda e^{- \lambda t}-\delta C^{*}(\beta) e^{- \lambda t}
\end{equation*}
and one has
\begin{equation}\label{4.5}
	\mathcal{L} W_{\delta}^{+} \geq \varrho \delta \lambda e^{-\lambda t} \bar{V}_\varpi -\delta\lambda e^{- \lambda t}-\delta C^{*}(\beta) e^{- \lambda t} +f\left(\bar{V}\right)-f\left(W_{\delta}^{+}\right).
\end{equation}
It follows from \eqref{3.13} and \eqref{3.20} that 
\begin{equation}\label{4.6}
\begin{aligned}
	\frac{\partial}{\partial \varpi}\left(\bar{V}(\varpi, x, y)\right) & = U^{\prime}(\xi) \xi_\varpi + \varepsilon \frac{\partial}{\partial \varpi}\left[h(\alpha \varpi, \alpha x, \alpha) \times\left(U^\beta(\eta) \omega(\eta)+(1-\omega(\eta))\right)\right]\\
	& = -U^{\prime}(\xi) \left(\frac{\partial_t \varphi}{\sqrt{1+|\nabla \varphi|^2}}+\alpha \xi \frac{\nabla \varphi \cdot \nabla \partial_t \varphi}{1+|\nabla \varphi|^2}\right)\\
	&\quad + \varepsilon \alpha \frac{\partial h(\alpha \varpi, \alpha x, \alpha)}{\partial \alpha \varpi} \times\left(U^\beta(\eta) \omega(\eta)+(1-\omega(\eta))\right)\\
	&\quad+\varepsilon h(\alpha \varpi, \alpha x, \alpha) \times\frac{\partial}{\partial \varpi}\left[\left(U^\beta(\eta) \omega(\eta)+(1-\omega(\eta))\right)\right],
\end{aligned}
\end{equation}
where $\xi$ and $\eta$ are evaluated at $(\tau, x, y)$. Furthermore, we obtain
\begin{equation*}
	\varepsilon \alpha \frac{\partial h(\alpha \varpi, \alpha x, \alpha)}{\partial \alpha  \varpi} \times\left(U^\beta(\eta) \omega(\eta)+(1-\omega(\eta))\right)=2 \varepsilon \alpha  c_{f} h(\alpha \varpi, \alpha x, \alpha)\left(U^\beta(\eta) \omega(\eta)+(1-\omega(\eta))\right)\geq 0,
\end{equation*}
\begin{equation*}
\begin{aligned}
	&\varepsilon h(\alpha \varpi, \alpha x, \alpha) \times\frac{\partial}{\partial \varpi}\left[\left(U^\beta(\eta) \omega(\eta)+(1-\omega(\eta))\right)\right]\\
	&=\varepsilon h(\alpha \varpi, \alpha x, \alpha) \times\frac{\partial}{\partial \varpi}\left[\left(U^\beta(\eta) -1\right) \omega(\eta)\right]\\
	&=\varepsilon h(\alpha \varpi, \alpha x, \alpha)\left(U^\beta(\eta) -1\right)\omega^{\prime}(\eta)\partial_\varpi \eta + \varepsilon \beta h(\alpha \varpi, \alpha x, \alpha) \omega(\eta) U^{\beta-1}(\eta) \partial_\eta U(\eta) \partial_\varpi \eta\geq 0,
\end{aligned}
\end{equation*}
and
\begin{equation*}
	\begin{aligned}
		&-U^{\prime}(\xi) \left(\frac{\partial_t \varphi}{\sqrt{1+|\nabla \varphi|^2}}+\alpha \xi \frac{\nabla \varphi \cdot \nabla \partial_t \varphi}{1+|\nabla \varphi|^2}\right)\geq - c_{f} U^{\prime}(\xi)  -\alpha \xi \frac{\nabla \varphi \cdot \nabla \partial_t \varphi}{1+|\nabla \varphi|^2}U^{\prime}(\xi).
	\end{aligned}
\end{equation*}
Since $-X^{\prime\prime}\leq \eta(\varpi(t), x, y)\leq X^{\prime}$, we know $\xi(\varpi(t), x, y)$ is bounded and $U^{\prime}(\xi(\varpi(t), x, y)) \xi(\varpi(t), x, y)$ is bounded. Assume there is a constant $M_{1}>0$ such that $|U^{\prime}(\xi(\varpi(t), x, y)) \xi(\varpi(t), x, y)|\leq M_{1}$.
According to \eqref{4.6} and Lemma \ref{Lemma 3.3}, for any $0<\alpha<\alpha_0^{+}\leq \frac{\kappa c_{f}}{2 M_{1} C_{2}}$, it holds that
\begin{equation}\label{4.7}
	\frac{\partial}{\partial \varpi}\left(\bar{V}(\varpi, x, y)\right)\geq -c_{f} U^{\prime}(\xi) -\alpha M_{1} C_{2}h(\alpha \varpi, \alpha x, \alpha)\geq c_{f} \min_{\xi\in[-X^{\prime\prime}, X^{\prime}]}\left|U^{\prime}(\xi)\right| -\alpha M_{1} C_{2}\geq \frac{1}{2} \kappa c_{f}.
\end{equation}
By \eqref{4.5} and \eqref{4.7}, we know
\begin{equation*}
  \begin{aligned}
	\mathcal{L} W_{\delta}^{+} &\geq \frac{1}{2} \kappa c_{f} \varrho \delta \lambda e^{-\lambda t}  -\delta\lambda e^{- \lambda t}-\delta C^{*}(\beta) e^{- \lambda t} +f\left(\bar{V}\right)-f\left(W_{\delta}^{+}\right)\\
	&\geq \frac{1}{2} \kappa c_{f} \varrho \delta \lambda e^{-\lambda t}  -\delta\lambda e^{- \lambda t}-\delta C^{*}(\beta) e^{- \lambda t} -\left\|f_u\right\|_{L^{\infty}([0, 1])}\delta e^{- \lambda t}\\
	&\geq \delta e^{- \lambda t}\left(\frac{1}{2} \kappa c_{f} \varrho \lambda- \lambda - C^{*}(\beta)-\left\|f_u\right\|_{L^{\infty}([0, 1])}\right)
  \end{aligned}
\end{equation*}	
 in $\left\{(t, x, y): \eta(\varpi(t), x, y) \in\left[-X^{\prime \prime}, X^{\prime}\right], t \geq 0\right\}$. Since 
 \begin{equation*}
 	 \varrho= 3 \frac{\left\|f_u\right\|_{L^{\infty}([0, 1])}+\lambda+C^{*}(\beta)}{\lambda \kappa c_{f}},
 \end{equation*}
 we prove that $\mathcal{L} W_{\delta}^{+}>0$ in Case $3$. 
 
 In conclusion, by setting $\alpha\in\left(0, \min\left\{\frac{\beta_{0} e}{\varrho C_{2} L_{3}}, \frac{ 3 L_{1}^{\beta} e}{4 \varrho C_{2} L_{2}}, \alpha_0^{+}(\beta, \varepsilon)\right\}\right)$ small enough, 
 \begin{equation*}
 	0< \lambda <\min\left\{-\frac{f^{\prime}(1)}{4}, \frac{\beta c_{f}^{2}}{16 }\right\}\quad \text{and} \quad \varrho= 3 \frac{\left\|f_u\right\|_{L^{\infty}([0, 1])}+\lambda+C^{*}(\beta)}{\lambda \kappa c_{f}},
 \end{equation*}
 the proof of Lemma \ref{Lemma 4.1} is completed. 
\end{proof}

\begin{proposition}\label{Proposition 3.4}
	Suppose that $(F1)$ and $(F2)$ hold. For the Cauchy problem
	\begin{equation}\label{4.12}
		\begin{cases}\partial_t u-\Delta u=f(u), &  t>0 , (x, y) \in \mathbb{R}^N, \\ u(t, x, y)=u_0 (x, y), &  t=0 , (x, y) \in \mathbb{R}^N,\end{cases}
	\end{equation}
assume that $u_0 \in C\left(\mathbb{R}^N,[0,1]\right)$ satisfies
\begin{equation}\label{4.13}
	\underline{V}(0, x, y) \leq u_0(x, y)
\end{equation}
for all $(x, y) \in \mathbb{R}^N$, and
\begin{equation}\label{4.14}
	\frac{\left|u_0(x, y)-\underline{V}(0, x, y)\right|}{\min \left\{1, e^{-2 v \min_{1\leq i\leq n} 
	\left\{x \cdot \nu_i \cot \theta_i+y+\frac{\tau_i}{\sin \theta_i}  \right\}}\right\} }\rightarrow 0 
	\quad\text{uniformly as } 
	d\left((x, y), \mathcal{R}_{0}\right) \rightarrow+\infty
\end{equation}
for some constant $v>0$, where $\mathcal{R}_0$ is the time slice of the ridges 
$\mathcal{R}$ at $t=0$. Then the solution $u(t, x, y)$ of Cauchy problem \eqref{4.12} for $t \geq 
0$ with initial condition $u(0, x, y)=u_0(x, y)$ satisfies
\begin{equation*}
	\lim _{t \rightarrow+\infty} \sup _{(x, y) \in \mathbb{R}\times\mathbb{R}^N}|u(t, x, y)-V(t, x, y)|=0 .
\end{equation*}
\end{proposition}
\begin{proof}
It follows from \eqref{3.7} and $\varphi(t, x, \alpha) \geq \psi(t, x, \alpha)$ in $\mathbb{R}\times\mathbb{R}^{N-1}$ that 
\begin{equation}\label{4.15}
	\eta(t,x,y)=y-\frac{\varphi(\alpha t, \alpha x, \alpha)}{\alpha} \leq y-\frac{\psi(\alpha t, \alpha x, \alpha)}{\alpha} = \min_{1\leq i\leq n} \left\{x \cdot \nu_i \cot \theta_i+y-\frac{c_f}{\sin \theta_i} t+\frac{\tau_i}{\sin \theta_i} \right\}.
\end{equation}

{\it Step 1: constructing super-solutions of  \eqref{4.12}.} 
Set $\beta_1:=\min \left\{v / c_{f}, \beta^*\right\}$, where positive constants $v$ and $\beta^*$ are given in \eqref{4.14} and Lemma \ref{Lemma 4.1}, respectively. For any $\delta \in\left(0,  \gamma_{\star} / 8\right]$, it follows from \eqref{4.14} that there exists a constant $R_{\delta}>0$ such that
\begin{equation}\label{4.16}
	u_0(x, y) \leq \underline{V}(0, x, y)+\delta\left(\frac{L_1}{2}\right)^{\beta_1} \min \left\{1, e^{- v \min_{1\leq i\leq n} \left\{x \cdot \nu_i \cot \theta_i+y+\frac{\tau_i}{\sin \theta_i}  \right\}}\right\}
\end{equation}
for all $d\left((x, y), \mathcal{R}_{0}\right)>R_{\delta}$, where the constant $L_1>0$ is given by \eqref{1.5}. We claim that
\begin{equation}\label{4.17}
	W_{\delta}^{+}(0, x, y) \geq u_0(x, y) \quad\text{in } \mathbb{R}^N
\end{equation}
for all $\delta\in\left(0, \gamma_{\star} / 8\right]$, where parameters in Lemma \ref{Lemma 4.1} are taken as $\beta=\beta_1$, $\lambda=\lambda\left(\beta_1\right)$, $\varrho=\varrho\left(\beta_1, \lambda\right)$, $\forall \varepsilon \in\left(0, \varepsilon_0^{+}\left(\beta_1\right)\right)$, and $0<\alpha<\min \left\{\frac{\beta_{0} e}{\varrho C_{2} L_{3}}, \frac{ 3 L_{1}^{\beta_{1}} e}{4 \varrho C_{2} L_{2}}, \alpha_0^{+}\left(\beta_1, \varepsilon\right)\right\}$ is to be determined.
	
\textbf{Case 1: }$\min_{1\leq i\leq n} \left\{x \cdot \nu_i \cot \theta_i+y+\frac{\tau_i}{\sin \theta_i} \right\}>0$. By \eqref{1.5}, there is a positive constant $X_*>0$ such that
\begin{equation}\label{4.18}
	U(\eta) \geq \frac{L_1}{2} e^{-c_{f} \eta}, \quad \forall(\eta, x, y) \in\left(X_*,+\infty\right) \times \mathbb{R}^N.
\end{equation}
Since $\beta_1 \leq v / K$, it follows from Lemma \ref{Lemma 4.1}, \eqref{4.15}, \eqref{4.16}, \eqref{4.18} and $\varphi(t, x, \alpha) \geq \psi(t, x, \alpha)$ in $\mathbb{R}\times\mathbb{R}^{N-1}$ that
\begin{equation*}
   \begin{aligned}
	W_{\delta}^{+}(0, x, y) &\geq \min\left\{\bar{V}(0, x, y)+\delta U^{\beta_1}(\eta(\varpi(0), x, y)), 1\right\}\\
   	& \geq \min\left\{\underline{V}(0, x, y)+\delta\left(\frac{L_1}{2}\right)^{\beta_1} e^{-\beta_1 c_{f} \eta} , 1\right\}\\
   	& \geq \min\left\{\underline{V}(0, x, y)+\delta\left(\frac{L_1}{2}\right)^{\beta_1} e^{-v \eta} , 1\right\}\\
   	& \geq \min\left\{\underline{V}(0, x, y)+\delta\left(\frac{L_1}{2}\right)^{\beta_1} e^{-v \min_{1\leq i\leq n} \left\{x \cdot \nu_i \cot \theta_i+y+\frac{\tau_i}{\sin \theta_i} \right\}} , 1\right\}\\
   	& \geq u_0(x, y)
   \end{aligned}
\end{equation*}
in $\left\{(x, y): \eta(0, x, y)>X_*, d\left((x, y), \mathcal{R}_{0}\right)>R_{\delta}\right\}$. By \eqref{1.5}, we can get
\begin{equation*}
	\begin{aligned}
		W_{\delta}^{+}(0, x, y) &\geq \min\left\{\bar{V}(0, x, y)+{\delta} U^{\beta_1}\left(X_*\right), 1\right\} \\
		&\geq \min\left\{\bar{V}(0, x, y)+{\delta} U\left(X_*\right), 1\right\}\\
		&\geq \min\left\{\bar{V}(0, x, y)+r_1, 1\right\}
	\end{aligned}
\end{equation*}
in $\left\{(x, y): \eta(0, x, y) \leq X_*\right\}$ for some constant $r_1>0$. Therefore, even if it means increasing $R_{\delta}$, it follows from \eqref{4.14} that
\begin{equation*}
	 u_0(x, y)-\underline{V}(0, x, y) \leq r_1 \quad\text{in }\left\{(x, y): \eta(0, x, y) \leq X_*, d\left((x, y), \mathcal{R}_{0}\right)>R_{\delta}\right\}
\end{equation*}
and
\begin{equation*}
W_{\delta}^{+}(0, x, y)\geq  \min\left\{\bar{V}(0, x, y)+r_1, 1\right\} \geq \min\left\{\underline{V}(0, x, y)+r_1 , 1\right\}\geq u_0(x, y)
\end{equation*}
in $\left\{(x, y): \eta(0, x, y) \leq X_*, d\left((x, y), \mathcal{R}_{0}\right)>R_{\delta}\right\}$. By \cite[Lemma 4.2]{H. Guo2024}, we can obtain
\begin{equation*}
\xi(0, x, y)=\frac{y-\varphi(0, \alpha x, \alpha) / \alpha}{\sqrt{1+|\nabla \varphi(0, \alpha x, \alpha)|^2}} \rightarrow-\infty \quad\text{as } \alpha \rightarrow 0^{+} \text{ uniformly in } \left\{(x, y):d\left((x, y), \mathcal{R}_{0}\right)\leq R_{\delta}\right\}.
\end{equation*}
By \eqref{1.3}, \eqref{1.5}, \eqref{4.15} and Lemma \ref{Lemma 4.1}, there is a constant 
$\alpha_6^{+}(\delta)>0$ such that for any $0<\alpha<\alpha_6^{+}$,
\begin{equation}\label{4.19}
W_{\delta}^{+}(0, x, y) \geq \min\left\{U(\xi)+\delta U^{\beta_1}\left(\eta\right), 1\right\}\geq \min\left\{ U(\xi)+\delta U\left(\eta\right) , 1\right\}= 1 \geq u_0(x, y)
\end{equation}
in $\left\{(x, y):d\left((x, y), \mathcal{R}_{0}\right)\leq R_{\delta}\right\}$. Hence, we prove $W_{\delta}^{+}(0, x, y) \geq u_0(x, y)$ in Case $1$.
	
\textbf{Case 2: }$\min_{1\leq i\leq n} \left\{x \cdot \nu_i \cot \theta_i+y+\frac{\tau_i}{\sin \theta_i} \right\}\leq 0$. It follows from \eqref{1.5} and \eqref{4.15} that
\begin{equation*}
	\begin{aligned}
    W_{\delta}^{+}(0, x, y) &\geq \min\left\{\bar{V}(0, x, y)+\delta U^{\beta_1}(0, x, y), 1\right\}\\
    &\geq \min\left\{\bar{V}(0, x, y)+\delta U(0, x, y), 1\right\}\\
    &\geq \min\left\{\underline{V}(0, x, y)+r_2, 1\right\}		
	\end{aligned}
\end{equation*}
for some constant $r_2>0$. By \eqref{4.14}, even if it means increasing $R_{\delta}$, one knows
\begin{equation*}
	u_0(x, y)-\underline{V}(0, x, y) \leq r_2 \quad\text{in }\left\{(x, y): d\left((x, y), \mathcal{R}_{0}\right)>R_{\delta}\right\}
\end{equation*}
and
\begin{equation*}
W_{\delta}^{+}(0, x, y)\geq  \min\left\{\underline{V}(0, x, y)+r_2 , 1\right\} \geq u_0(x, y) \quad\text{in }\left\{(x, y): d\left((x, y), \mathcal{R}_{0}\right)>R_{\delta}\right\}.
\end{equation*}
Similar to arguments of \eqref{4.19}, there exists a constant $\alpha_7^{+}(\delta)>0$ such that for each $0<\alpha<\alpha_7^{+}$,
\begin{equation*}
W_{\delta}^{+}(0, x, y) \geq u_0(x, y) \quad\text{in }\left\{(x, y): d\left((x, y), \mathcal{R}_{0}\right)\leq R_{\delta}\right\}.
\end{equation*}
Thus one has $W_{\delta}^{+}(0, x, y) \geq u_0(x, y)$ in Case $2$.

Finally, we get that \eqref{4.17} is true for any $0<\alpha<\min \left\{\frac{\beta_{0} e}{\varrho C_{2} L_{3}}, \frac{ 3 L_{1}^{\beta_{1}} e}{4 \varrho C_{2} L_{2}}, \alpha_0^{+}\left(\beta_1, \varepsilon\right), \alpha_6^{+}(\delta), \alpha_7^{+}(\delta)\right\}$.

{\it Step 2: construct a time sequence.}
By \eqref{4.13}, Lemma \ref{Lemma 4.1} and Step $1$, using the comparison principle, we obtain
\begin{equation}\label{4.20}
\underline{V}(t, x, y) \leq u(t, x, y) \leq W_{\delta}^{+}(t, x, y) \quad\text{in }[0,+\infty) \times \mathbb{R}^N
\end{equation}
for any $\delta \in\left(0,  \gamma_{\star} / 8\right]$. It follows from \eqref{4.20} that
\begin{equation}\label{4.21}
\underline{V}(t, x, y) \leq u(t, x, y) \leq \bar{V}\left(t, x, y\right)+ \left\|\partial_t \bar{V} \right\|_{L^{\infty}([0,1])}  \varrho \delta +  \delta \quad\text{in }[0,+\infty) \times \mathbb{R}^N
\end{equation}
Setting $\delta>0$ small enough, one has that for any $i \in\{1, \ldots, n\}$,
\begin{equation*}
	\underline{V}(t, x, y) \leq u(t, x, y) \leq \bar{V}\left(t, x, y\right)+ \varepsilon \quad\text{in }[0,+\infty) \times \mathbb{R}^N.
\end{equation*}
By Lemma \ref{Lemma 3.4}, for any $\varepsilon>0$, there is an $\rho(\varepsilon)>0$ such that
\begin{equation}\label{4.22}
|u(t, x, y)-\underline{V}(t, x, y)| \leq 3 \varepsilon \quad\text{in }\{d((t, x, y), \mathcal{R}) \geq \rho(\varepsilon)\} .
\end{equation}

Assume by the contrary that there exists an $\varepsilon>0$ and a sequence $\left\{\left(t_k, x_k, y_k\right)\right\}$ such that as $k \rightarrow+\infty$, $t_k \rightarrow+\infty$ and
\begin{equation}\label{4.23}
\left|u\left(t_k, x_k, y_k\right)-V\left(t_k, x_k, y_k\right)\right| > 3 \varepsilon .
\end{equation}
By \eqref{4.22}, it holds that
\begin{equation*}
\limsup _{k \rightarrow+\infty} d\left(\left(t_k, x_k, y_k\right), \mathcal{R}\right) \leq \rho(\varepsilon) .
\end{equation*}
Divide the set $\{1, \ldots, n\}$ into two subsets $I$ and $J$ such that $x_k \cdot \nu_i \cos \theta_i+y_k \sin \theta_i- c_f t_k+\tau_i$ is uniformly bounded for $i \in I$, while for $j \in J$,
\begin{equation}\label{4.24}
\lim _{k \rightarrow+\infty}\left(x_k \cdot \nu_j \cos \theta_j+y_k \sin \theta_i-c_f t_k+\tau_j\right)=+\infty.
\end{equation}
The set $I$ is not empty, but the set $J$ may be empty.

After passing to a subsequence of $\left(t_k, x_k, y_k\right)$ if necessary, there are $\eta_i \in \mathbb{R}$ $(i \in I)$ such that 
\begin{equation*}
\lim _{k \rightarrow+\infty}\left(x_k \cdot \nu_i \cos \theta_i+y_k \sin \theta_i-c_f t_k+\tau_i\right)=\eta_i \quad \text{for each } i \in I .
\end{equation*}
Hence, $\partial \mathcal{P}-\left(t_k, x_k, y_k\right) \rightarrow \partial \mathcal{P}_{\infty}$, where $\mathcal{P}_{\infty}$ is the polytope enclosed by hyperplanes
\begin{equation*}
\left\{(t, x, y) \in \mathbb{R} \times \mathbb{R}^N ; x \cdot \nu_i \cos \theta_i+y \sin \theta_i-c_f t+\eta_i=0\right\},\; i \in I.
\end{equation*}
It follows from \eqref{4.24} that $U\left(x \cdot \nu_j \cos \theta_j+y \sin \theta_j-c_f t+\eta_j\right) \rightarrow 0$ for each $j \in J$. Therefore, $\underline{V}\left(t+ t_k, x+x_k, y+y_k\right)$ converge to
\begin{equation*}
u_{I, *}(t, x, y):=\max _{i \in I}\left\{U\left(x \cdot \nu_i \cos \theta_i+y \sin \theta_i-c_f t+\eta_i\right)\right\}.
\end{equation*}

Denote $u_k(t, x, y)=u\left(t+t_k, x+x_k, y+y_k\right)$. By parabolic estimates and \eqref{4.22}, $u_k$ converge (up to a subsequence) to a solution $u_{\infty}$ of \eqref{3.1}. By \eqref{3.14} and \eqref{4.21} (with $\delta$ small enough), we obtain
\begin{equation*}
\left|u_{\infty}(t, x, y)-u_{I, *}(t, x, y)\right| \leq 3 \varepsilon \quad\text{in } \{d\left((t, x, y), \partial \mathcal{P}_{\infty}\right) \geq \rho(\varepsilon)\}
\end{equation*}
and 
\begin{equation*}
\left|u_{\infty}(t, x, y)-u_{I, *}(t, x, y)\right| \leq 4 \varepsilon  \quad\text{in }\{ d\left((t, x, y), \mathcal{R}_{\infty}\right) \geq \rho(\varepsilon)\} .
\end{equation*}
Since $\varepsilon$ is arbitrary, one has
\begin{equation*}
\left|u_{\infty}(t, x, y)-u_{I, *}(t, x, y)\right| \rightarrow 0 \quad \text {uniformly as } d\left((t, x, y), \mathcal{R}_{\infty}\right) \rightarrow+\infty .
\end{equation*}
According to Proposition \ref{Proposition 3.6}, it holds that $u_{\infty} \equiv V_I$ on $\mathbb{R} \times \mathbb{R}^N$, where $V_I:=V\left(\cdot ;\left(e_i, \eta_i \right)_{i \in I}\right)$. However, by Proposition \ref{Proposition 3.5} and \eqref{4.24}, $V\left(t+t_k, x+x_k, y+y_k\right)$ also converge to $V_I$ as $k \rightarrow+\infty$. This is a contradiction with \eqref{4.23}.
The proof is complete.
\end{proof}

\vspace{0.3cm}
\begin{proof}[Prof of Theorem \ref{Theorem 2.4}]
Since \eqref{1.1} is invariant under rotations, the stability of $V_{\boldsymbol{e}, 
\boldsymbol{\tau}}$ in Theorem \ref{Theorem 2.1}, i.e. Theorem \ref{Theorem 2.4}, is a direct 
consequence of the Proposition \ref{Proposition 3.4}.
\end{proof}

\section*{Acknowledgments}

This work was partially supported by NSF of China (12171120) and by the Fundamental Research 
Funds for the Central Universities (No. 2023FRFK030022, 2022FRFK060028).

\section*{ Data availability statements}  

We do not analyze or generate any datasets, because our work proceeds within a theoretical 
and mathematical approach.

\section*{ Conflict of interest} 

There is no conflict of interest to declare.

\end{document}